\documentclass[11pt,onecolumn]{IEEEtran}

\usepackage[pdftex]{graphicx}				
\usepackage{amssymb,amsmath,amsthm}
\usepackage{color}

\newtheorem{Def}{Definition}
\newtheorem{Thm}{Theorem}
\newtheorem{Lem}{Lemma}
\newtheorem{Cor}{Corollary}
\newtheorem{Alg}{Algorithm}
\newtheorem{Fact}{Fact}
\newtheorem{Assumption}{Assumption}

\newtheorem{Rem}{Remark}

\DeclareMathOperator*{\argmin}{argmin}
\DeclareMathOperator*{\argmax}{argmax}

\title{On the Convergence Time of Dual Subgradient Methods for Strongly Convex Programs}
\author{Hao Yu and Michael J. Neely\\University of Southern California
\thanks{Hao Yu (email: yuhao@usc.edu) and Michael J. Neely (web: http://www-rcf.usc.edu/$\sim$mjneely) are with the Department of Electrical Engineering, University of Southern California, Los Angeles, CA 90089, USA.}
\thanks{This work was presented in part at IEEE Conference on Decision and Control (CDC), Osaka, Japan, December, 2015 \cite{YuNeely15CDC}. This work is supported in part  by the NSF grant CCF-0747525.}}

\begin{document}
\maketitle

\begin{abstract}
This paper studies the convergence time of dual gradient methods for general (possibly non-differentiable) strongly convex programs. For general convex programs, the convergence time of  dual subgradient/gradient methods with simple running averages (running averages started from iteration $0$) is known to be $O(\frac{1}{\epsilon^{2}})$.  This paper shows that the convergence time for general strongly convex programs is $O(\frac{1}{\epsilon})$. This paper also considers a variation of the average scheme, called the sliding running averages, and shows that if the dual function of the strongly convex program is locally quadratic (Note that the locally quadratic property is implied by the locally strongly concave property.) then the convergence time of the dual gradient method with sliding running averages is $O(\log(\frac{1}{\epsilon}))$.  The convergence time analysis is further verified by numerical experiments.
\end{abstract}

\section{Introduction}
Consider the following strongly convex program:
\begin{align}
\min \quad &f(\mathbf{x}) \label{eq:program-objective}\\
\text{s.t.} \quad  &  g_k(\mathbf{x})\leq 0, \forall k\in\{1,2,\ldots,m\} \label{eq:program-inequality-constraint}\\
			 &  \mathbf{x}\in \mathcal{X} \label{eq:program-set-constraint}
\end{align}
where set $\mathcal{X}\subseteq{R}^n$ is closed and convex; function $f(\mathbf{x})$ is continuous and strongly convex on $\mathcal{X}$ (strong convexity is defined in Section \ref{sec:preliminaries}); functions $g_k(\mathbf{x}),\forall k \in\{1,2,\ldots,m\}$ are Lipschitz continuous and convex on $\mathcal{X}$.  Note that the functions $f(\mathbf{x}), g_1(\mathbf{x}), \ldots, g_m(\mathbf{x})$ are not necessarily differentiable.   It is assumed throughout that problem \eqref{eq:program-objective}-\eqref{eq:program-set-constraint} has an optimal solution.   Strong convexity of $f$ implies the optimum is unique. 

Convex program \eqref{eq:program-objective}-\eqref{eq:program-set-constraint} arises often in control applications such as model predictive control (MPC) \cite{Necoara14TAC}, decentralized multi-agent control \cite{Terelius11IFAC}, and network flow control \cite{Low99TON, Low03TON}. More specifically, the model predictive control problem is to solve problem \eqref{eq:program-objective}-\eqref{eq:program-set-constraint} where $f(\mathbf{x})$ is a quadratic function and each $g_k(\mathbf{x})$ is a linear function \cite{Necoara14TAC}.  In decentralized multi-agent control \cite{Terelius11IFAC}, our goal is to develop distributive algorithms to solve problem \eqref{eq:program-objective}-\eqref{eq:program-set-constraint} where $f(\mathbf{x})$ is the sum utility of individual agents and constraints $g_k(\mathbf{x})\leq 0$ capture the communication or resource allocation constraints among individual agents.  The network flow control and ``TCP" protocols in computer networks can be interpreted as the dual subgradient algorithm for solving
a problem of the form \eqref{eq:program-objective}-\eqref{eq:program-set-constraint} \cite{Low99TON, Low03TON}.  In particular, Section \ref{sec:num} shows that the dual subgradient method based
\emph{online flow control} rapidly converges to optimality when utilities are strongly convex.

\subsection{The $\epsilon$-Approximate Solution}
\begin{Def}
Let $\mathbf{x}^\ast$ be the optimal solution to problem  \eqref{eq:program-objective}-\eqref{eq:program-set-constraint}. For any $\epsilon>0$, a point $\mathbf{x}^\epsilon \in \mathcal{X}$ is said to be an {\it $\epsilon$-approximate solution} if $f(\mathbf{x}^{\epsilon}) \leq f(\mathbf{x}^{\ast}) + \epsilon$ and $g_{k}(\mathbf{x}^{\epsilon}) \leq \epsilon, \forall k\in \{1,\ldots, m\}$.
\end{Def}

\begin{Def}
Let $\mathbf{x}(t), t\in \{1,2,\ldots\}$ be the solution sequence yielded by an iterative algorithm. The {\it convergence time} (to an $\epsilon$ approximate solution) is the number of iterations required to achieve an $\epsilon$ approximate solution. That is, this algorithm is said to have  convergence time $O(h(\epsilon))$ if $\{\mathbf{x}(t), t\geq O(h(\epsilon))\}$ is a sequence of $\epsilon$-approximate solutions.
\end{Def}
Note if $\mathbf{x}(t)$ satisfies $f(\mathbf{x}(t)) \leq f(\mathbf{x}^{\ast}) +\frac{1}{t}$ and $g_{k}(\mathbf{x}(t)) \leq \frac{1}{t}, \forall k\in \{1,\ldots, m\}$ for all $t\geq 1$, then error decays with time like $O(\frac{1}{t})$ and the convergence time  is $O(\frac{1}{\epsilon})$.

\subsection{The Dual Subgradient/Gradient Method}
The dual subgradient method is a conventional method to solve \eqref{eq:program-objective}-\eqref{eq:program-set-constraint} \cite{book_NonlinearProgrammingTA, book_NonlinearProgramming_Bertsekas}.  It is an iterative algorithm that, every iteration, removes the inequality constraints \eqref{eq:program-inequality-constraint} and chooses primal variables to minimize a 
function over the set $\mathcal{X}$.  This can be decomposed into parallel smaller problems if the objective and constraint functions are separable. The dual subgradient method can be interpreted as a subgradient/gradient method applied to the Lagrangian dual function of convex program  \eqref{eq:program-objective}-\eqref{eq:program-set-constraint} and allows for many different step size rules \cite{book_NonlinearProgramming_Bertsekas}.  This paper focuses on  
 a constant step size due to its simplicity for practical implementations. Note that by Danskin's theorem (Proposition B.25(a) in \cite{book_NonlinearProgramming_Bertsekas}), the Lagrangian dual function of a strongly convex program is differentiable, thus the dual subgradient method for strongly convex program \eqref{eq:program-objective}-\eqref{eq:program-set-constraint} is in fact a dual gradient method.  The constant step size dual subgradient/gradient method solves problem \eqref{eq:program-objective}-\eqref{eq:program-set-constraint} as follows:

\begin{Alg}\label{alg:dual-subgradient} [The Dual Subgradient/Gradient Method] Let $c >0$ be a constant step size.  Let $\boldsymbol{\lambda}(0) \geq \mathbf{0},$ be a given constant vector. At each iteration $t$, update $\mathbf{x}(t)$ and $\boldsymbol{\lambda}(t+1)$ as follows:
\begin{itemize}
\item  $\mathbf{x}(t) = \argmin\limits_{\mathbf{x}\in \mathcal{X}} \left[ f(\mathbf{x}) + \sum_{k=1}^m \lambda_k(t) g_k(\mathbf{x})\right].$
\item  $\lambda_k(t+1) = \max\left\{\lambda_k(t) + c g_k(\mathbf{x}(t)), 0\right\}, \forall k.$
\end{itemize} 
\end{Alg}

Rather than using $\mathbf{x}(t)$ from Algorithm \ref{alg:dual-subgradient} as the primal solutions, the following running average schemes are considered:
\begin{enumerate}
\item {\bf Simple Running Averages:} Use $\overline{\mathbf{x}}(t) =\frac{1}{t}\sum_{\tau =0}^{t-1} \mathbf{x}(\tau)$ as the solution at each iteration $t\in\{1,2,\ldots\}$.
\item {\bf Sliding Running Averages:} Use  $\widetilde{\mathbf{x}}(t) = \mathbf{x}(0)$ and
\[
\widetilde{\mathbf{x}}(t) = \left \{ \begin{array}{lc} \frac{1}{\frac{t}{2}}\sum_{\tau = \frac{t}{2}}^{t-1} \mathbf{x}(\tau) &~\text{if}~ t ~\text{is even} \\ 
\widetilde{\mathbf{x}}(t-1) &~\text{if}~t~\text{is odd} \end{array}\right.
\]
as the solution at each iteration $t\in\{1,2,\ldots\}$.
\end{enumerate}
The simple running average sequence $\overline{\mathbf{x}}(t)$ is also called the ergodic sequence in \cite{Larsson99_MathematicalProgramming}. The idea of using the running average $\overline{\mathbf{x}}(t)$ as the solutions, rather than the original primal variables $\mathbf{x}(t)$, dates back to Shor \cite{book_MinimizeNonDifferentiableFun_Shor} and has been further developed in \cite{Sherali96ORL} and \cite{Larsson99_MathematicalProgramming}. The constant step size dual subgradient algorithm with simple running averages is also a special case of the drift-plus-penalty algorithm, which was originally developed to solve more general stochastic optimization \cite{book_Neely10} and used for deterministic convex programs in \cite{Neely05DCDIS}. (See Section I.C in \cite{YuNeely15CDC} for more discussions.) This paper proposes a new running average scheme, called sliding running averages. This paper shows that the sliding running averages can have better convergence time when the dual function of the convex program satisfies additional assumptions. 
\subsection{Related Work}
A lot of literature focuses on the convergence time of dual subgradient methods to an $\epsilon$-approximiate solution.  For general convex programs in the form of \eqref{eq:program-objective}-\eqref{eq:program-set-constraint}, where the objective function$f(\mathbf{x})$ is convex but not necessarily strongly convex, the convergence time of the drift-plus-penalty algorithm is shown to be $O(\frac{1}{\epsilon^2})$ in \cite{Neely05DCDIS,Neely14Arxiv_ConvergenceTime}. A similar $O(\frac{1}{\epsilon^2})$ convergence time of the dual subgradient algorithm with the averaged primal sequence is shown in \cite{Nedic09}. A recent work \cite{Sucha14CDC} shows that the convergence time of the drift-plus-penalty algorithm is $O(\frac{1}{\epsilon})$ if the dual function is locally polyhedral and the convergence time is $O(\frac{1}{\epsilon^{3/2}})$ if the dual function is locally quadratic. For a special class of strongly convex programs in the form of \eqref{eq:program-objective}-\eqref{eq:program-set-constraint}, where $f(\mathbf{x})$ is second-order differentiable and strongly convex and $g_k(\mathbf{x}), \forall k\in\{1,2,\ldots,m\}$ are second-order differentiable and have bounded Jacobians, the convergence time of the dual subgradient algorithm is shown to be $O(\frac{1}{\epsilon})$ in \cite{Necoara14TAC}. 

Note that convex program \eqref{eq:program-objective}-\eqref{eq:program-set-constraint} with second order differentiable $f(\mathbf{x})$ and $g_k(\mathbf{x}), k\in\{1,2,\ldots,m\}$ in general can be solved via interior point methods with linear convergence time. However, to achieve good convergence time in practice, the barrier parameters must be scaled carefully and the computation complexity associated with each iteration can be high. In contrast, the dual subgradient algorithm is a Lagrangian dual method and can yield distributed implementations with low computation complexity when the objective and constraint functions are separable.

This paper considers a class of strongly convex programs that is more general than those treated in \cite{Necoara14TAC}.\footnote{Note that bounded Jacobians imply Lipschitz continuity.  
Work \cite{Necoara14TAC} also considers the effect of inaccurate solutions for the primal updates. The analysis in this paper can also deal with inaccurate  updates. In this case, there will be an error term $\delta$ on the right of \eqref{eq:drift-plus-penalty-bound}.} Besides the strong convexity of $f(\mathbf{x})$, we only require the constraint functions $g_k(\mathbf{x})$ to be Lipschitz continuous. The functions  $f(\mathbf{x})$ and $g_k(\mathbf{x})$ can even be non-differentiable. Thus,  this paper can deal with non-smooth optimization. For example, the $l_1$ norm $\Vert \mathbf{x} \Vert_1$ is non-differentiable and often appears as part of the objective or constraint functions in machine learning, compressed sensing and image processing applications. This paper shows that the convergence time of the dual subgradient method with simple running averages for general strongly convex programs is $O(\frac{1}{\epsilon})$ and the convergence time can be improved to $O(\log(\frac{1}{\epsilon}))$ by using sliding running averages when the dual function is locally quadratic.

A closely related recent work is \cite{Necoara16OMS} that considers strongly convex programs with strongly convex and second order differentiable objective functions $f(\mathbf{x})$ and conic constraints in the form of $\mathbf{G}\mathbf{x} + h\in \mathcal{K}$, where $\mathcal{K}$ is a proper cone.  
The authors in \cite{Necoara16OMS} show that a hybrid algorithm using both dual subgradient and dual fast gradient methods can have convergence time  $O(\frac{1}{\epsilon^{2/3}})$; and the dual subgradient method can have convergence time $O(\log(\frac{1}{\epsilon}))$ if the strongly convex program satisfies an error bound property.  Results in the current paper are developed independently and consider general nonlinear convex
constraint functions; and show that the dual subgradient/gradient method with a different averaging scheme has an $O(\log(\frac{1}{\epsilon}))$ convergence time when the dual function is locally quadratic.  Another independent parallel work is \cite{Necoara15Springer} that considers strongly convex programs with strongly convex and smooth objective functions $f(\mathbf{x})$ and general constraint functions $\mathbf{g}(\mathbf{x})$ with bounded Jacobians. The authors in \cite{Necoara15Springer} shows that the dual subgradient/gradient method with simple running averages has $O(\frac{1}{\epsilon})$ convergence.

This paper and independent parallel works \cite{Necoara16OMS,Necoara15Springer} obtain similar convergence times of the dual subgradient/gradient method with different averaging schemes for strongly convex programs under slightly different assumptions. However, the proof technique in this paper is fundamentally different from that used in \cite{Necoara16OMS} and \cite{Necoara15Springer}. Works \cite{Necoara16OMS, Necoara15Springer} and other previous works, e.g., \cite{Necoara14TAC}, follow the classical optimization analysis approach based on the descent lemma, while this paper is based on the drift-plus-penalty analysis that was originally developed for stochastic optimization in dynamic queuing systems \cite{PhD_Thesis_Neely,book_Neely10}. Using the drift-plus-penalty technique, we further propose a new Lagrangian dual type algorithm with $O(\frac{1}{\epsilon})$ convergence for general convex programs (possibly without strong convexity) in a following work \cite{YuNeely17SIOPT}.

\subsection{Notations}
Denote the $i$-th element of vector $\mathbf{x}\in \mathbb{R}^n$ as $x_i$. Denote $\mathbb{R}^{n}_{+} = \big\{\mathbf{x}\in \mathbb{R}^{n}: x_{i}\geq 0, \forall i\in\{1,\ldots,n\}\big\}$. For all $ \mathbf{x}, \mathbf{y} \in \mathbb{R}^{n}$, we write $\mathbf{x} \geq \mathbf{y}$ if and only if $x_{i} \geq y_{i}, \forall 1\leq i\leq n$. Denote the Euclidean norm of vector $\mathbf{x} = [x_1,x_2,\ldots, x_n]^T \in \mathbb{R}^n$ as $\Vert \mathbf{x}\Vert = \sqrt{\sum_{i=1}^n x_i^2}$.   For real numbers $a, b\in \mathbb{R}$ such that $a\leq b$, the projection operator onto the interval $[a,b]$ is denoted as $$[x]_{a}^{b} = \left\{\begin{array}{cl} b & \text{if}~x> b \\ x &\text{if}~a\leq x\leq b \\ a &\text{if}~x<a\end{array}\right..$$ For a convex function $f(\mathbf{x})$, the set of all subgradients of $f$ at point $\mathbf{x} = \mathbf{x}^\ast$ is denoted as $\partial f(\mathbf{x}^\ast)$; the gradient of $f$ at $\mathbf{x} = \mathbf{x}^{\ast}$ is denoted as $\nabla_{\mathbf{x}} f(\mathbf{x}^{\ast})$ if $f$ is differentiable; and the Hessian of $f$ at $\mathbf{x}=\mathbf{x}^{\ast}$ is denoted as $\nabla_{\mathbf{x}}^{2} f(\mathbf{x}^{\ast})$ if $f$ is second order differentiable. For all $x\in \mathbb{R}$, the largest integer less than or equal to $x$ is denoted as $\lfloor x\rfloor$ and the smallest integer larger than or equal to $x$ is denoted as $\lceil x \rceil$. If a square matrix $\mathbf{A}$ is positive definite, we write $A\succ 0$; and if $A$ is negative definite, we write $A\prec 0$. The $n\times n$ identity matrix is denoted as $\mathbf{I}_{n}$.

\section{Preliminaries and Basic Analysis}
This section presents useful preliminaries in convex analysis and important facts of the dual subgradient/gradient method.

\subsection{Preliminaries}\label{sec:preliminaries}
\begin{Def}[Lipschitz Continuity] \label{def:Lipschitz-continuous}
Let $\mathcal{X} \subseteq \mathbb{R}^n$ be a convex set. Function $h: \mathcal{X}\rightarrow \mathbb{R}^m$ is said to be Lipschitz continuous  on $\mathcal{X}$ with modulus $L$ if there exists $L> 0$ such that $\Vert h(\mathbf{y}) - h(\mathbf{x}) \Vert \leq L \Vert\mathbf{y} - \mathbf{x}\Vert$  for all $ \mathbf{x}, \mathbf{y} \in \mathcal{X}$. 
\end{Def}

\begin{Def}[Strongly Convex Functions]
Let $\mathcal{X} \subseteq \mathbb{R}^n$ be a convex set. Function $h$ is said to be strongly convex on $\mathcal{X}$ with modulus $\alpha$ if there exists a constant $\alpha>0$ such that $h(\mathbf{x}) - \frac{1}{2} \alpha \Vert \mathbf{x} \Vert^2$ is convex on $\mathcal{X}$.
\end{Def}

\begin{Lem}[Theorem 6.1.2 in \cite{book_FundamentalConvexAnalysis}] \label{lm:strong-convex}
Let $f(\mathbf{x})$ be strongly convex on $\mathcal{X}$ with modulus $\alpha$. Then $ f(\mathbf{y}) \geq  f(\mathbf{x}) +\mathbf{d}^T (\mathbf{y} - \mathbf{x}) + \frac{\alpha}{2}\Vert \mathbf{y} - \mathbf{x} \Vert^2$ for all $\mathbf{x}, \mathbf{y}\in \mathcal{X}$ and all $\mathbf{d}\in \partial f(\mathbf{x})$.
\end{Lem}

\begin{Lem}[Proposition B.24 (f) in \cite{book_NonlinearProgramming_Bertsekas}] \label{lm:first-order-optimality}
Let $\mathcal{X}\subseteq \mathbb{R}^n$ be a convex set. Let function $h$ be convex on $\mathcal{X}$ and and $\mathbf{x}^{opt}$ is a global minimum of $h$ on $\mathcal{X}$.  Then there exists $\mathbf{d} \in \partial f(\mathbf{x}^{opt})$ such that $\mathbf{d}^T(\mathbf{y} - \mathbf{x}^{opt} )\geq 0$ for all $\mathbf{y}\in \mathcal{X}$. 
\end{Lem}

Combining Lemma \ref{lm:strong-convex} and Lemma \ref{lm:first-order-optimality} gives the following:
\begin{Cor}\label{cor:strong-convex-optimum}
Let $h(\mathbf{x})$ be strongly convex on convex set $\mathcal{X}$ with modulus $\alpha$. If $\mathbf{x}^{opt}$ is a global minimum, then $h(\mathbf{x}^{opt}) \leq h(\mathbf{y}) - \frac{\alpha}{2}\Vert \mathbf{y}-\mathbf{x}^{opt}\Vert^2, \forall \mathbf{y}\in \mathcal{X}$.
\end{Cor}
\begin{IEEEproof}
Fix $\mathbf{y}\in \mathcal{X}$. By Lemma \ref{lm:first-order-optimality}, there exists $\mathbf{d} \in \partial h(\mathbf{x}^{opt})$ such that $\mathbf{d}^T(\mathbf{y} - \mathbf{x}^{opt}) \geq 0$. By Lemma \ref{lm:strong-convex}, we also have
\begin{align*}
h(\mathbf{y}) &\geq h(\mathbf{x}^{opt}) +  \mathbf{d}^{T} (\mathbf{y} - \mathbf{x}^{opt}) + \frac{\alpha}{2} \Vert  \mathbf{y} -\mathbf{x}^{opt} \Vert^{2}\\
&\overset{(a)}{\geq} h(\mathbf{x}^{opt}) + \frac{\alpha}{2} \Vert  \mathbf{y} -\mathbf{x}^{opt} \Vert^{2},
\end{align*}
where $(a)$ follows from the fact that $\mathbf{d}^T(\mathbf{y} - \mathbf{x}^{opt}) \geq 0$.
\end{IEEEproof}

\subsection{Assumptions}

Denote the stacked vector of multiple functions $g_1(\mathbf{x}), g_2(\mathbf{x}), \ldots, g_m(\mathbf{x})$ as 
$$\mathbf{g}(\mathbf{x}) = \big[g_1(\mathbf{x}), g_2(\mathbf{x}), \ldots, g_m(\mathbf{x})\big]^T.$$

Throughout this paper, the following assumptions are imposed on convex program \eqref{eq:program-objective}-\eqref{eq:program-set-constraint}.

\begin{Assumption}\label{as:strongly-convex-problem} In convex program \eqref{eq:program-objective}-\eqref{eq:program-set-constraint}, function $f(\mathbf{x})$ is strongly convex on $\mathcal{X}$ with modulus $\alpha$; and function $\mathbf{g}(\mathbf{x})$ is Lipschitz continuous on $\mathcal{X}$ with modulus $\beta$. 
\end{Assumption}
\begin{Assumption}[Existence of Lagrange Multipliers] \label{as:strong-duality}
There exists a Lagrange multiplier vector $\boldsymbol{\lambda}^\ast = [\lambda_1^\ast, \lambda_2^\ast, \ldots, \lambda_m^\ast] \geq \mathbf{0}$ attaining the strong duality for convex program  \eqref{eq:program-objective}-\eqref{eq:program-set-constraint}. That is, 
\begin{align*}
q(\boldsymbol{\lambda}^\ast) = \min\limits_{\mathbf{x}\in \mathcal{X}}\left\{f(\mathbf{x}) : g_k(\mathbf{x})\leq 0, \forall k\in\{1,2,\ldots,m\}\right\} 
\end{align*}
where $q(\boldsymbol{\lambda}) = \min\limits_{\mathbf{x}\in \mathcal{X}}\{f(\mathbf{x})+ \sum_{k=1}^m \lambda_k g_k(\mathbf{x})\}
$ is the {\it Lagrangian dual function} of problem \eqref{eq:program-objective}-\eqref{eq:program-set-constraint}.\end{Assumption}

\subsection{Properties of the Drift}

Denote $\boldsymbol{\lambda}(t) = \big[ \lambda_1(t), \ldots, \lambda_m(t)\big]^T$.  Define {Lyapunov function} $L(t) = \frac{1}{2} \Vert \boldsymbol{\lambda}(t)\Vert^2 $ and {drift} $\Delta (t) = L(t+1) - L(t)$.
\begin{Lem}\label{lm:drift}
At each iteration $t$ in Algorithm \ref{alg:dual-subgradient}, 
\begin{align}
\frac{1}{c}\Delta(t) = \boldsymbol{\lambda}^T(t + 1) \mathbf{g}(\mathbf{x}(t))  - \frac{1}{2c} \Vert\boldsymbol{\lambda}(t+1) - \boldsymbol{\lambda}(t)\Vert^{2}  \label{eq:drift}
\end{align}
\end{Lem}

\begin{IEEEproof}
The update equations $\lambda_k(t+1) = \max\{\lambda_k(t) + cg_k(\mathbf{x}(t)), 0\}, \forall k\in \{1,2,\ldots,m\}$ can be rewritten as
\begin{align}
\lambda_k(t+1) = \lambda_k(t) + c\tilde{g}_k(\mathbf{x}(t)), \forall k\in \{1,2,\ldots,m\}, \label{eq:modified-virtual-queue}
\end{align} 
where $\tilde{g}_k(\mathbf{x}(t)) = \left \{ \begin{array}{cl}  g_k(\mathbf{x}(t)), & \text{if}~\lambda_k(t) + cg_k(\mathbf{x}(t)) \geq 0\\
-\frac{1}{c}\lambda_k(t),  & \text{else} \end{array} \right.$, $\forall k\in \{1,2,\ldots, m\}$.

Fix $k\in\{1,2,\ldots, m\}$. Squaring both sides of \eqref{eq:modified-virtual-queue} and dividing by factor $2$ yields: 
\begin{align*}
&\frac{1}{2}[\lambda_k (t+1) ]^2 \\
= &\frac{1}{2}[\lambda_k(t)]^2 + \frac{c^2}{2} [\tilde{g}_k(\mathbf{x}(t))]^2 + c \lambda_k (t) \tilde{g}_k(\mathbf{x}(t)) \\
=& \frac{1}{2}[\lambda_k(t)]^2 + \frac{c^2}{2} [\tilde{g}_k(\mathbf{x}(t))]^2 +  c \lambda_k (t) g_k(\mathbf{x}(t)) +  c\lambda_k (t)[ \tilde{g}_k(\mathbf{x}(t)) -g_k(\mathbf{x}(t)) ]  \\
\overset{(a)}{=}& \frac{1}{2}[\lambda_k(t)]^2 + \frac{c^2}{2} [\tilde{g}_k(\mathbf{x}(t))]^2 + c \lambda_k (t) g_k(\mathbf{x}(t)) - c^2 \tilde{g}_k (\mathbf{x}(t)) [ \tilde{g}_k(\mathbf{x}(t)) -g_k(\mathbf{x}(t)) ]\\
=& \frac{1}{2}[\lambda_k(t)]^2 - \frac{c^2}{2} [\tilde{g}_k(\mathbf{x}(t))]^2 + c [\lambda_k (t)+ c\tilde{g}_k (\mathbf{x}(t)) ] g_k(\mathbf{x}(t)) \\
\overset{(b)}{=}& \frac{1}{2}[\lambda_k(t)]^2 - \frac{1}{2} [\lambda_k(t+1) - \lambda_k(t)]^2 + c \lambda_k (t+1) g_k(\mathbf{x}(t))
\end{align*}
where $(a)$ follows from $\lambda_k (t) [ \tilde{g}_k (\mathbf{x}(t)) -g_k(\mathbf{x}(t)) ] = -c \tilde{g}_k (\mathbf{x}(t)) [\tilde{g}_k (\mathbf{x}(t)) -g_k(\mathbf{x}(t)) ]$, which can be shown by considering $\tilde{g}_k (\mathbf{x}(t)) = g_k(\mathbf{x}(t))$ and $\tilde{g}_k (\mathbf{x}(t)) \neq g_k(\mathbf{x}(t))$, separately; and $(b)$ follows from the fact that $\lambda_k(t+1) = \lambda_k(t) + c\tilde{g}_k (\mathbf{x}(t))$.
Summing over $k\in\{1,2,\ldots,m\}$ yields $\frac{1}{2}\Vert\boldsymbol{\lambda}(t+1)\Vert^{2} = \frac{1}{2}\Vert\boldsymbol{\lambda}(t)\Vert^{2} - \frac{1}{2} c^2 \Vert\boldsymbol{\lambda}(t+1) - \boldsymbol{\lambda}(t)\Vert^{2} + c \boldsymbol{\lambda}^T(t + 1) \mathbf{g}(\mathbf{x}(t))$. Rearranging the terms and dividing both sides by factor $c$ yields the result.

\end{IEEEproof}

\section{Convergence Time Analysis} \label{sec:general-strongly-convex}
This section analyzes the convergence time of Algorithm \ref{alg:dual-subgradient} for strongly convex program \eqref{eq:program-objective}-\eqref{eq:program-set-constraint}.

\subsection{Objective Value Violations}

\begin{Lem}\label{lm:drift-plus-penalty}
Let $\mathbf{x}^\ast \in \mathcal{X}$ be the optimal solution to problem \eqref{eq:program-objective}-\eqref{eq:program-set-constraint}. At each iteration $t$ in Algorithm \ref{alg:dual-subgradient}, we have
\begin{align}
\frac{1}{c}\Delta(t) + f(\mathbf{x}(t)) \leq  f (\mathbf{x}^\ast) , \forall t\geq 0, \label{eq:drift-plus-penalty-bound}
\end{align}

\end{Lem}

\begin{IEEEproof}
Fix $t\geq0$. Since $f(\mathbf{x})$ is strongly convex with modulus $\alpha$; $g_k(\mathbf{x}), \forall k\in\{1,2,\ldots,m\}$ are convex; and $\lambda_k(t), \forall k\in\{1,2,\ldots,m\}$ are non-negative at each iteration $t$, the function $f(\mathbf{x}) + \sum_{k=1}^m \lambda_k(t) g_k(\mathbf{x})$ is also strongly convex with modulus $\alpha$ at each iteration $t$. Note that $\mathbf{x}(t) = \argmin\limits_{\mathbf{x}\in \mathcal{X}} \Big[f(\mathbf{x}) + \sum_{k=1}^m \lambda_k(t) g_k(\mathbf{x})\Big]$.  By Corollary \ref{cor:strong-convex-optimum} with $\mathbf{x}^{opt} = \mathbf{x}(t)$ and  $\mathbf{y} = \mathbf{x}^\ast$, we have 
\begin{align*}
\big[ f(\mathbf{x}(t)) + \sum_{k=1}^m \lambda_k (t) g_k (\mathbf{x}(t))\big] \leq \big[ f(\mathbf{x}^\ast) +  \sum_{k=1}^m \lambda_k (t) g_k (\mathbf{x}^\ast)\big] - \frac{\alpha}{2} \Vert \mathbf{x}(t) - \mathbf{x}^\ast \Vert^2.
\end{align*}

Hence,  $f(\mathbf{x}(t))  \leq  f(\mathbf{x}^\ast) +  \mathbf{\lambda}^T(t) \big[ \mathbf{g}(\mathbf{x}^\ast) -\mathbf{g}(\mathbf{x}(t)) \big] -\frac{\alpha}{2} \Vert \mathbf{x}(t) - \mathbf{x}^\ast \Vert^2$. Adding this inequality to equation \eqref{eq:drift} yields 
\begin{align*}
&\frac{1}{c}\Delta(t) + f(\mathbf{x}(t)) \\
\leq & f (\mathbf{x}^\ast) -\frac{1}{2c} \Vert \boldsymbol{\lambda}(t+1) - \boldsymbol{\lambda}(t) \Vert^{2} - \frac{\alpha}{2} \Vert \mathbf{x}(t) - \mathbf{x}^\ast \Vert^2 + \boldsymbol{\lambda}^{T}(t) [ \mathbf{g}(\mathbf{x}^\ast) - \mathbf{g} (\mathbf{x}(t))] +\boldsymbol{\lambda}^{T}(t+1)\mathbf{g}(\mathbf{x}(t)).
\end{align*} 

Define 
\begin{align*}
B(t) =  -\frac{1}{2c} \Vert \boldsymbol{\lambda}(t+1) - \boldsymbol{\lambda}(t) \Vert^{2} - \frac{\alpha}{2} \Vert \mathbf{x}(t) - \mathbf{x}^\ast \Vert^2 + \boldsymbol{\lambda}^{T}(t) [ \mathbf{g}(\mathbf{x}^\ast) - \mathbf{g} (\mathbf{x}(t))] +\boldsymbol{\lambda}^{T}(t+1)\mathbf{g}(\mathbf{x}(t)).
\end{align*}
Next, we need to show that $B(t) \leq 0$.


Since $\mathbf{x}^\ast$ is the optimal solution to problem \eqref{eq:program-objective}-\eqref{eq:program-set-constraint}, we have $g_k(\mathbf{x}^\ast)\leq 0, \forall k \in\{1,2,\ldots,m\}$. Note that $\lambda_k(t+1) \geq 0, \forall k\in \{1,2,\ldots,m\}, \forall t\geq 0$. Thus,  
\begin{align}
\boldsymbol{\lambda}^T(t+1) \mathbf{g}(\mathbf{x}^\ast) \leq 0, \quad \forall t\geq 0 \label{eq:proof_eq_1}
\end{align}
Now we have,
\begin{align*}
B(t) = &-\frac{1}{2c} \Vert \boldsymbol{\lambda}(t+1) - \boldsymbol{\lambda}(t) \Vert^{2} - \frac{\alpha}{2} \Vert \mathbf{x}(t) - \mathbf{x}^\ast \Vert^2 + \boldsymbol{\lambda}^{T}(t) [ \mathbf{g}(\mathbf{x}^\ast) - \mathbf{g} (\mathbf{x}(t))] +\boldsymbol{\lambda}^{T}(t+1)\mathbf{g}(\mathbf{x}(t)) \\
\overset{(a)}{\leq} & -\frac{1}{2c} \Vert \boldsymbol{\lambda}(t+1) - \boldsymbol{\lambda}(t) \Vert^{2} - \frac{\alpha}{2} \Vert \mathbf{x}(t) - \mathbf{x}^\ast \Vert^2  + \boldsymbol{\lambda}^{T}(t) [ \mathbf{g}(\mathbf{x}^\ast) - \mathbf{g} (\mathbf{x}(t))] +\boldsymbol{\lambda}^{T}(t+1)\mathbf{g}(\mathbf{x}(t)) \\
&\quad - \boldsymbol{\lambda}^T(t+1) \mathbf{g}(\mathbf{x}^\ast) \\
= & -\frac{1}{2c} \Vert \boldsymbol{\lambda}(t+1)- \boldsymbol{\lambda}(t) \Vert^{2} - \frac{\alpha}{2} \Vert \mathbf{x}(t) - \mathbf{x}^\ast \Vert^2 +[ \boldsymbol{\lambda}^{T}(t) - \boldsymbol{\lambda}^T(t+1)]  [ \mathbf{g}(\mathbf{x}^\ast) - \mathbf{g} (\mathbf{x}(t))] \\
\overset{(b)}{\leq} & -\frac{1}{2c} \Vert \mathbf{Q}(t+1) - \mathbf{Q}(t) \Vert^{2} - \frac{\alpha}{2} \Vert \mathbf{x}(t) - \mathbf{x}^\ast \Vert^2  + \Vert \boldsymbol{\lambda}(t) - \boldsymbol{\lambda}(t+1)\Vert  \Vert \mathbf{g} (\mathbf{x}(t)) - \mathbf{g}(\mathbf{x}^\ast)\Vert \\
\overset{(c)}{\leq} & -\frac{1}{2c} \Vert \boldsymbol{\lambda}(t+1) - \boldsymbol{\lambda}(t) \Vert^{2} - \frac{\alpha}{2} \Vert \mathbf{x}(t) - \mathbf{x}^\ast \Vert^2  + \beta \Vert \boldsymbol{\lambda}(t) - \boldsymbol{\lambda}(t+1)\Vert  \Vert \mathbf{x}(t) - \mathbf{x}^\ast\Vert \\
= & -\frac{1}{2c} \Big( \Vert \boldsymbol{\lambda}(t+1) - \boldsymbol{\lambda}(t) \Vert - c\beta  \Vert \mathbf{x}(t) - \mathbf{x}^\ast\Vert \Big)^2  - \frac{1}{2}(\alpha - c \beta^2) \Vert \mathbf{x}(t) - \mathbf{x}^\ast \Vert^2 \\
\overset{(d)}{\leq} & 0
\end{align*}

where $(a)$ follows from \eqref{eq:proof_eq_1}; $(b)$ follows from the Cauchy-Schwarz inequality; $(c)$ follows from Assumption \ref{as:strongly-convex-problem}; and $(d)$ follows from $c\leq \frac{\alpha}{\beta^2}$.
\end{IEEEproof}

\begin{Thm}[Objective Value Violations] \label{thm:utility-convergence-time}
Let $\mathbf{x}^\ast \in \mathcal{X}$ be the optimal solution to problem \eqref{eq:program-objective}-\eqref{eq:program-set-constraint}. If $c\leq \frac{\alpha}{\beta^2}$ in Algorithm \ref{alg:dual-subgradient}, then
\begin{align*}
f(\overline{\mathbf{x}}(t)) \leq  f(\mathbf{x}^\ast) + \frac{\Vert \boldsymbol{\lambda}(0)\Vert^2}{2ct}, \forall t\geq 1.
\end{align*}
\end{Thm}
\begin{IEEEproof}
By Lemma \ref{lm:drift-plus-penalty}, we have $\frac{1}{c}\Delta(\tau) + f(\mathbf{x}(\tau)) \leq  f(\mathbf{x}^\ast)$ for all $\tau \in\{0,1,\ldots,t-1\}$. Summing over $\tau \in \{0,1,\ldots,t-1\}$ we have:
\begin{align*}
&\frac{1}{c}\sum_{\tau=0}^{t-1} \Delta(\tau) + \sum_{\tau=0}^{t-1} f(\mathbf{x}(\tau)) \leq t f(\mathbf{x}^\ast) \\
\Rightarrow\quad&\frac{1}{c} [L(t) - L(0)] + \sum_{\tau=0}^{t-1} f(\mathbf{x}(\tau)) \leq  t f(\mathbf{x}^\ast) \\
\Rightarrow\quad& \frac{1}{t}\sum_{\tau=0}^{t-1} f(\mathbf{x}(\tau)) \leq f(\mathbf{x}^\ast) + \frac{L(0) - L(t)}{ct} \leq f(\mathbf{x}^\ast) + \frac{L(0)}{ct}
\end{align*}
Note that $\overline{\mathbf{x}}(t) = \frac{1}{t}\sum_{\tau=0}^{t-1} \mathbf{x}(\tau)$ and by the convexity of $f(\mathbf{x})$, we have
\begin{align*}
f(\overline{\mathbf{x}}(t)) \leq \frac{1}{t}\sum_{\tau=0}^{t-1} f(\mathbf{x}(\tau)) \leq f(\mathbf{x}^\ast) + \frac{ L(0)}{ct} = f(\mathbf{x}^\ast) + \frac{\Vert \boldsymbol{\lambda}(0)\Vert^2}{2ct}
\end{align*} 
\end{IEEEproof}

\begin{Rem}
Similarly, we can prove that $ f(\widetilde{\mathbf{x}}(2t)) \leq f(\mathbf{x}^\ast) + \frac{\Vert \boldsymbol{\lambda}(t)\Vert^2}{2ct}$ since $\widetilde{\mathbf{x}}(2t) = \frac{1}{t} \sum_{\tau = t}^{2t-1} \mathbf{x}(\tau)$.  A later lemma (Lemma \ref{lm:queue-bound}) guarantees that $\Vert \boldsymbol{\lambda}(t)\Vert \leq \sqrt{\Vert \boldsymbol{\lambda}(0)\Vert^2 + \Vert \boldsymbol{\lambda}^\ast\Vert^2} + \Vert \boldsymbol{\lambda}^\ast\Vert $, where $\boldsymbol{\lambda}^\ast$ is defined in Assumption \ref{as:strong-duality}.  Thus, $f(\widetilde{\mathbf{x}}(2t)) \leq  f(\mathbf{x}^\ast) + \frac{\big(\sqrt{\Vert \boldsymbol{\lambda}(0)\Vert^2 + \Vert \boldsymbol{\lambda}^\ast\Vert^2} + \Vert \boldsymbol{\lambda}^\ast\Vert\big)^2}{2ct}, \forall t\geq 1.$ 
\end{Rem}

\subsection{Constraint Violations}

The analysis of constraint violations is similar to that in \cite{Neely14Arxiv_ConvergenceTime} for general convex programs. However, using the improved upper bound in Lemma \ref{lm:drift-plus-penalty}, the convergence time of constraint violations in strongly convex programs is order-wise better than that in general convex programs.

\begin{Lem}\label{lm:queue_constraint_inequality}
For any $t_2 > t_1 \geq 0 $, $\lambda_k(t_2) \geq \lambda_k(t_1)  + c\displaystyle{\sum_{\tau=t_1}^{t_2-1} g_k(\mathbf{x}(\tau))}, \forall k\in\{1,2,\ldots,m\}$. In particular, for any $t>0$,  $\lambda_k(t) \geq \lambda_k(0)  + c\displaystyle{\sum_{\tau=0}^{t-1} g_k(\mathbf{x}(\tau))}, \forall k\in\{1,2,\ldots,m\}$
\end{Lem}
\begin{IEEEproof}
Fix $k\in\{1,2,\ldots,m\}$. Note that $\lambda_k(t_1+1) = \max\{\lambda_k(t_1)+cg_k(\mathbf{x}(t_1)), 0\} \geq \lambda_k(t_1) + cg_k(\mathbf{x}(t_1))$. By induction, this lemma follows.
\end{IEEEproof}

\begin{Lem}\label{lm:queue-bound}
Let $\boldsymbol{\lambda}^\ast \geq \mathbf{0}$ be given in Assumption \ref{as:strong-duality}. If $c\leq \frac{\alpha}{\beta^2}$ in Algorithm \ref{alg:dual-subgradient}, then $\boldsymbol{\lambda}(t)$ satisfies
\begin{align}
\Vert \boldsymbol{\lambda}(t)\Vert \leq \sqrt{\Vert \boldsymbol{\lambda}(0)\Vert^2 + \Vert \boldsymbol{\lambda}^\ast\Vert^2} + \Vert \boldsymbol{\lambda}^\ast\Vert, \forall t\geq 1.
\end{align}
\end{Lem}
\begin{IEEEproof}
Let $\mathbf{x}^\ast$ be the optimal solution to problem \eqref{eq:program-objective}-\eqref{eq:program-set-constraint}. AAssumption \ref{as:strong-duality} implies that  
\begin{align*}
f(\mathbf{x}^\ast) =q(\boldsymbol{\lambda}^\ast) \leq f(\mathbf{x}(\tau)) + \sum_{k=1}^m \lambda_k^\ast g_k(\mathbf{x}(\tau)),\forall \tau\in\{0,1,\ldots\},
\end{align*}
where the inequality follows from the definition of $q(\boldsymbol{\lambda})$.

Thus, we have $f(\mathbf{x}^\ast) - f(\mathbf{x}(\tau)) \leq \sum_{k=1}^m \lambda_k^\ast g_k(\mathbf{x}(\tau)),\forall \tau \in\{0,1,\ldots\}$. Summing over $\tau\in \{0,1,\ldots, t-1\}$ yields 
\begin{align}
t f(\mathbf{x}^\ast) - \sum_{\tau=0}^{t-1} f(\mathbf{x}(\tau)) &\leq \sum_{\tau=0}^{t-1} \sum_{k=1}^m\lambda_k^\ast g_k(\mathbf{x}(\tau))\nonumber \\
&=\sum_{k=1}^m  \lambda_k^\ast \Big[\sum_{\tau=0}^{t-1}g_k(\mathbf{x}(\tau))\Big] \nonumber\\
&\overset{(a)}{\leq} \frac{1}{c}\sum_{k=1}^m  \lambda_k^\ast [\lambda_k(t) - \lambda_k(0)] \nonumber\\
&\leq \frac{1}{c}\sum_{k=1}^m  \lambda_k^\ast \lambda_k(t) \nonumber\\
&\overset{(b)}{\leq} \frac{1}{c} \Vert \boldsymbol{\lambda}^\ast\Vert \Vert \boldsymbol{\lambda}(t)\Vert  \label{eq:proof_lm8_eq1}
\end{align}
where $(a)$ follows from Lemma \ref{lm:queue_constraint_inequality} and $(b)$ follows from the Cauchy-Schwarz inequality. On the other hand, summing \eqref{eq:drift-plus-penalty-bound} in Lemma \ref{lm:drift-plus-penalty} over $\tau\in\{0,1,\ldots, t-1\}$ yields
\begin{align}
t f(\mathbf{x}^\ast) - \sum_{\tau=0}^{t-1} f(\mathbf{x}(\tau)) & \geq \frac{L(t) - L(0)}{c} \nonumber \\
&= \frac{\Vert \boldsymbol{\lambda}(t)\Vert^2 - \Vert \boldsymbol{\lambda}(0)\Vert^2}{2c} \label{eq:proof_lm8_eq2}
\end{align}
Combining \eqref{eq:proof_lm8_eq1} and \eqref{eq:proof_lm8_eq2} yields
\begin{align*}
&\frac{\Vert \boldsymbol{\lambda}(t)\Vert^2 - \Vert\boldsymbol{\lambda}(0)\Vert^2}{2c}  \leq \frac{1}{c}\Vert \boldsymbol{\lambda}^\ast\Vert \Vert \boldsymbol{\lambda}(t)\Vert \\
\Rightarrow\quad & \Big( \Vert \boldsymbol{\lambda}(t)\Vert - \Vert \boldsymbol{\lambda}^\ast\Vert \Big)^2 \leq  \Vert \boldsymbol{\lambda}(0)\Vert^2 + \Vert \boldsymbol{\lambda}^\ast\Vert^2 \\
\Rightarrow\quad & \Vert \boldsymbol{\lambda}(t)\Vert \leq \sqrt{\Vert \boldsymbol{\lambda}(0)\Vert^2 + \Vert \boldsymbol{\lambda}^\ast\Vert^2} + \Vert \boldsymbol{\lambda}^\ast\Vert
\end{align*}
\end{IEEEproof}

\begin{Thm}[Constraint Violations]\label{thm:constraint-convergence-time}
Let $\boldsymbol{\lambda}^\ast \geq \mathbf{0}$ be defined in Assumption \ref{as:strong-duality}. If $c\leq \frac{\alpha}{\beta^2}$ in Algorithm \ref{alg:dual-subgradient}, then the constraint functions satisfy $g_k (\overline{\mathbf{x}}(t)) \leq \frac{\sqrt{\Vert \boldsymbol{\lambda}(0)\Vert^2 + \Vert \boldsymbol{\lambda}^\ast\Vert^2} + \Vert \boldsymbol{\lambda}^\ast\Vert}{ct}, \forall k\in\{1,2,\ldots,m\}, \forall t\geq 1$.
\end{Thm}
\begin{IEEEproof}
Fix $t\geq 1$ and $k\in\{1,2,\ldots,m\}$. Recall that $\overline{\mathbf{x}}(t) = \frac{1}{t}\sum_{\tau=0}^{t-1} \mathbf{x}(\tau)$. Thus,
\begin{align*}
g_k (\overline{\mathbf{x}}(t)) &\overset{(a)}{\leq} \frac{1}{t} \sum_{\tau=0}^{t-1} g_k(\mathbf{x}(\tau)) \\
 &\overset{(b)}{\leq} \frac{\lambda_k(t) - \lambda_k(0)}{ct} \\
 &\leq \frac{\lambda_k(t)}{ct} \\
 &\leq \frac{\Vert \boldsymbol{\lambda}(t)\Vert}{ct} \\
 &\overset{(c)}{\leq} \frac{\sqrt{\Vert \boldsymbol{\lambda}(0)\Vert^2 + \Vert \boldsymbol{\lambda}^\ast\Vert^2} + \Vert \boldsymbol{\lambda}^\ast\Vert}{ct}
\end{align*}
where $(a)$ follows from the convexity of $g_k(\mathbf{x}), k\in\{1,2,\ldots,m\}$; $(b)$ follows from Lemma \ref{lm:queue_constraint_inequality}; and $(c)$ follows from Lemma \ref{lm:queue-bound}.
\end{IEEEproof}
\begin{Rem}
Similarly, we can prove that $g_k (\widetilde{\mathbf{x}}(2t)) \leq \frac{\sqrt{\Vert \boldsymbol{\lambda}(0)\Vert^2 + \Vert \boldsymbol{\lambda}^\ast\Vert^2} + \Vert \boldsymbol{\lambda}^\ast\Vert}{ct}, \forall k\in\{1,2,\ldots,m\}, \forall t\geq 1$
\end{Rem}

The next corollary provides a lower bound of $f(\overline{\mathbf{x}}(t))$ and follows directly from strong duality for convex programs and Theorem \ref{thm:constraint-convergence-time}.

\begin{Cor} Let $\boldsymbol{\lambda}^\ast \geq \mathbf{0}$ be defined in Assumption 2. If $c\leq \frac{\alpha}{\beta^{2}}$ in Algorithm 1, then $f(\overline{\mathbf{x}}(t))$ has a lower bound given by $
f(\overline{\mathbf{x}}(t)) \geq f(\mathbf{x}^{\ast}) - \frac{1}{t}\frac{\sqrt{\Vert \boldsymbol{\lambda}(0)\Vert^{2} + \Vert \boldsymbol{\lambda}^{\ast}\Vert^{2}} +\Vert \boldsymbol{\lambda}^{\ast}\Vert}{c}\sum_{k=1}^{m} \lambda_{k}^{\ast}, \forall t\geq 1$.
\end{Cor}
\begin{IEEEproof}
Fix $t\geq 1$. By the strong duality, we have 
\begin{align*}
f(\overline{\mathbf{x}}(t)) + \sum_{k=1}^{m} \lambda_{k}^{\ast} g_{k}(\overline{\mathbf{x}}(t)) \geq q(\boldsymbol{\lambda}^{\ast}) = f(\mathbf{x}^{\ast}) + \sum_{k=1}^{m} \lambda_{k}^{\ast} g_{k}(\mathbf{x}^{\ast}) =  f(\mathbf{x}^{\ast}) 
\end{align*}
Thus, we have
\begin{align*}
f(\overline{\mathbf{x}}(t)) \geq& f(\mathbf{x}^{\ast})  - \sum_{k=1}^{m} \lambda_{k}^{\ast} g_{k}(\overline{\mathbf{x}}(t))\\
\overset{(a)}{\geq} & f(\mathbf{x}^{\ast}) -\sum_{k=1}^{m} \lambda_{k}^{\ast} \frac{\sqrt{\Vert \boldsymbol{\lambda}(0)\Vert^{2} + \Vert \boldsymbol{\lambda}^{\ast}\Vert^{2}} + \Vert \boldsymbol{\lambda}^{\ast}\Vert}{ct} \\
=&f(\mathbf{x}^{\ast}) - \frac{1}{t}\frac{\sqrt{\Vert \boldsymbol{\lambda}(0)\Vert^{2} + \Vert \boldsymbol{\lambda}^{\ast}\Vert^{2}} +\Vert \boldsymbol{\lambda}^{\ast}\Vert}{c}\sum_{k=1}^{m} \lambda_{k}^{\ast} 
\end{align*}
where (a) follows from the constraint violation bound, i.e., Theorem \ref{thm:constraint-convergence-time}, and the fact that $\lambda_{k}^{\ast}\geq 0, \forall k\in\{1,2,\ldots,m\}$.
\end{IEEEproof}

\subsection{Convergence Time of Algorithm \ref{alg:dual-subgradient}}
The next theorem summarizes Theorem \ref{thm:utility-convergence-time} and Theorem \ref{thm:constraint-convergence-time}. 
\begin{Thm}\label{thm:overall-convergence-time}
Let $\mathbf{x}^\ast \in \mathcal{X}$ be the optimal solution to problem \eqref{eq:program-objective}-\eqref{eq:program-set-constraint}. Let $\boldsymbol{\lambda}^\ast \geq \mathbf{0}$ be given in Assumption \ref{as:strong-duality}. If $c\leq \frac{\alpha}{\beta^2}$ in Algorithm \ref{alg:dual-subgradient}, then for all $t\geq1$, 
\begin{align*}
&f(\overline{\mathbf{x}}(t)) \leq  f(\mathbf{x}^\ast) + \frac{\Vert \boldsymbol{\lambda}(0)\Vert^2}{2ct}. \\
&g_k (\overline{\mathbf{x}}(t)) \leq \frac{\sqrt{\Vert \boldsymbol{\lambda}(0)\Vert^2 + \Vert \boldsymbol{\lambda}^\ast\Vert^2} + \Vert \boldsymbol{\lambda}^\ast\Vert}{t}, \forall k\in\{1,2,\ldots,m\}.
\end{align*}
Specifically, if $\boldsymbol{\lambda}(0) = \mathbf{0}$, then
\begin{align*}
&f(\overline{\mathbf{x}}(t)) \leq  f(\mathbf{x}^\ast). \\
&g_k (\overline{\mathbf{x}}(t)) \leq \frac{2 \Vert \boldsymbol{\lambda}^\ast\Vert}{ct}, \forall k\in\{1,2,\ldots,m\}.
\end{align*}
In summary, if $c\leq \frac{\alpha}{\beta^2}$ in Algorithm \ref{alg:dual-subgradient}, then $\overline{\mathbf{x}}(t)$ ensures that error decays like $O(\frac{1}{t})$ and provides an $\epsilon$-approximiate solution with convergence time $O(\frac{1}{\epsilon})$.
\end{Thm}

\begin{Rem}
If $c\leq \frac{\alpha}{\beta^2}$ in Algorithm \ref{alg:dual-subgradient}, then $\widetilde{\mathbf{x}}(t)$ also ensures that error decays like $O(\frac{1}{t})$ and provides an $\epsilon$-approximiate solution with convergence time $O(\frac{1}{\epsilon})$. 
\end{Rem}

\section{Extensions}\label{sec:extensions}
This section  shows that the convergence time of sliding running averages $\widetilde{\mathbf{x}}(t)$ is $O(\log(\frac{1}{\epsilon}))$ when the dual function of problem \eqref{eq:program-objective}-\eqref{eq:program-set-constraint} satisfies additional assumptions.

\subsection{Smooth Dual Functions}

\begin{Def}[Smooth Functions]
Let $\mathcal{X} \subseteq \mathbb{R}^n$ and function $h(\mathbf{x})$ be continuously differentiable on $\mathcal{X}$. Function $h(\mathbf{x})$ is said to be smooth on $\mathcal{X}$ with modulus $L$ if $\nabla_{\mathbf{x}} h(\mathbf{x})$ is Lipschitz continuous on $\mathcal{X}$ with modulus $L$.
\end{Def}

Define $q(\boldsymbol{\lambda}) = \min\limits_{\mathbf{x}\in \mathcal{X}}\{f(\mathbf{x})+ \boldsymbol{\lambda}^T\mathbf{g}(\mathbf{x})\}
$ as the {\it dual function} of problem \eqref{eq:program-objective}-\eqref{eq:program-set-constraint}. Recall that $f(\mathbf{x})$ is strongly convex with modulus $\alpha$ by Assumption \ref{as:strongly-convex-problem}. For fixed $\boldsymbol{\lambda}\in \mathbb{R}^{m}_{+}$, $f(\mathbf{x}) + \boldsymbol{\lambda}^T\mathbf{g}(\mathbf{x})$ is strongly convex with respect to $\mathbf{x}\in \mathcal{X}$ with modulus $\alpha$. Define $\mathbf{x}(\boldsymbol{\lambda}) =  \argmin_{\mathbf{x}\in \mathcal{X}} \{f(\mathbf{x}) + \boldsymbol{\lambda}^T \mathbf{g}(\mathbf{x})\}$. By Danskin's theorem (Proposition B.25 in \cite{book_NonlinearProgramming_Bertsekas}), $q(\boldsymbol{\lambda})$ is differentiable with gradient $\nabla_{\boldsymbol{\lambda}} q(\boldsymbol{\lambda}) = \mathbf{g}(\mathbf{x}(\boldsymbol{\lambda}))$. 

\begin{Lem}[Smooth Dual Functions] \label{lm:dual-smooth}
The dual function $q(\boldsymbol{\lambda})$ is smooth on $\mathbb{R}^m_+$ with modulus $\gamma = \frac{\beta^2}{\alpha}$. 
\end{Lem}
\begin{IEEEproof}
Fix $\boldsymbol{\lambda}, \boldsymbol{\mu} \in \mathbb{R}^{m}_{+}$. Let $\mathbf{x}(\boldsymbol{\lambda}) =  \argmin_{\mathbf{x}\in \mathcal{X}} \{f(\mathbf{x}) + \boldsymbol{\lambda}^T \mathbf{g}(\mathbf{x})\}$ and $\mathbf{x}(\boldsymbol{\mu}) =  \argmin_{\mathbf{x}\in \mathcal{X}} \{f(\mathbf{x}) + \boldsymbol{\mu}^T\mathbf{g}(\mathbf{x})\}$. By Corollary \ref{cor:strong-convex-optimum}, we have 
\begin{align*}
f(\mathbf{x}(\boldsymbol{\lambda})) + \boldsymbol{\lambda}^T \mathbf{g}(\mathbf{x}(\boldsymbol{\lambda})) \leq& f(\mathbf{x}(\boldsymbol{\mu})) + \boldsymbol{\lambda}^T \mathbf{g}(\mathbf{x}(\boldsymbol{\mu})) - \frac{\alpha}{2} \Vert \mathbf{x}(\boldsymbol{\lambda}) - \mathbf{x}(\boldsymbol{\mu})\Vert^{2}\\
f(\mathbf{x}(\boldsymbol{\mu})) + \boldsymbol{\mu}^T \mathbf{g}(\mathbf{x}(\boldsymbol{\mu})) \leq &f(\mathbf{x}(\boldsymbol{\lambda})) + \boldsymbol{\mu}^T \mathbf{g}(\mathbf{x}(\boldsymbol{\lambda})) - \frac{\alpha}{2} \Vert \mathbf{x}(\boldsymbol{\lambda}) - \mathbf{x}(\boldsymbol{\mu})\Vert^{2}
\end{align*}

 Summing the above two inequalities and simplifying gives
\begin{align*}
\alpha \Vert \mathbf{x}(\boldsymbol{\lambda}) - \mathbf{x}(\boldsymbol{\mu})\Vert^{2} &\leq [\boldsymbol{\mu} - \boldsymbol{\lambda}]^T [\mathbf{g}(\mathbf{x}(\boldsymbol{\lambda})) - \mathbf{g}(\mathbf{x}(\boldsymbol{\mu}))]  \\
&\overset{(a)}{\leq} \Vert \boldsymbol{\mu} - \boldsymbol{\lambda} \Vert \Vert \mathbf{g}(\mathbf{x}(\boldsymbol{\lambda})) - \mathbf{g}(\mathbf{x}(\boldsymbol{\mu}))\Vert \\
&\overset{(b)}{\leq} \beta  \Vert \boldsymbol{\mu} - \boldsymbol{\lambda} \Vert \Vert \mathbf{x}(\boldsymbol{\lambda}) - \mathbf{x}(\boldsymbol{\mu})\Vert
\end{align*}
where (a) follows from the Cauchy-Schwarz inequality and (b) follows because $\mathbf{g}(\mathbf{x})$ is Lipschitz coninuous. This implies
\begin{align}
\Vert \mathbf{x}(\boldsymbol{\lambda}) - \mathbf{x}(\boldsymbol{\mu})\Vert \leq \frac{\beta}{\alpha}\Vert  \boldsymbol{\lambda} - \boldsymbol{\mu}\Vert  \label{eq:pf-dual-smooth-eq1}
\end{align}
Thus, we have 
\begin{align*}
\Vert \nabla q(\boldsymbol{\lambda}) - \nabla q(\boldsymbol{\mu}) \Vert &\overset{(a)}{=} \Vert \mathbf{g}(\mathbf{x}(\boldsymbol{\lambda})) - \mathbf{g}(\mathbf{x}(\boldsymbol{\mu}))\Vert \\
&\overset{(b)}{\leq} \beta \Vert \mathbf{x}(\boldsymbol{\lambda}) - \mathbf{x}(\boldsymbol{\mu})\Vert\\
&\overset{(c)}{\leq} \frac{\beta^{2}}{\alpha} \Vert  \boldsymbol{\lambda} - \boldsymbol{\mu}\Vert
\end{align*}
where (a) follows from $\nabla_{\boldsymbol{\lambda}} q(\boldsymbol{\lambda}) = \mathbf{g}(\mathbf{x}(\boldsymbol{\lambda}))$; (b) follows from the Lipschitz continuity of $\mathbf{g}(\mathbf{x})$; and (c) follows from \eqref{eq:pf-dual-smooth-eq1}.

Thus, $q(\boldsymbol{\lambda})$ is smooth on $\mathbb{R}^{m}_{+}$ with modulus $L = \frac{\beta^{2}}{\alpha}$.
\end{IEEEproof}

Since $\nabla_{\boldsymbol{\lambda}} q(\boldsymbol{\lambda}(t)) = \mathbf{g}(\mathbf{x}(t))$, the dynamic of $\boldsymbol{\lambda}(t)$ can be interpreted as the projected gradient method with step size $c$ to solve $\max_{\boldsymbol{\lambda}\in \mathbb{R}^m_+} \left\{q(\boldsymbol{\lambda})\right\}$ where $\mathbf{q}(\cdot)$ is a smooth function by Lemma \ref{lm:dual-smooth}. Thus, we have the next lemma.

\begin{Lem} \label{lm:dual_problem_convergence}
Assume problem \eqref{eq:program-objective}-\eqref{eq:program-set-constraint} satisfies Assumptions \ref{as:strongly-convex-problem}-\ref{as:strong-duality}. If $c\leq \frac{\alpha}{\beta^2}$ in Algorithm \ref{alg:dual-subgradient}, then 
\begin{align*}
q(\boldsymbol{\lambda}^\ast) - q(\boldsymbol{\lambda}(t)) \leq \frac{1}{2ct} \Vert \boldsymbol{\lambda}(0) - \boldsymbol{\lambda}^\ast\Vert^2, \quad \forall t\geq 1.
\end{align*}
\end{Lem}
\begin{IEEEproof}
Recall that a projected gradient descent algorithm with step size $c< \frac{1}{
\gamma}$ converges to the maximum of a concave function with smooth modulus $\gamma$ with the error decaying like $O(\frac{1}{t})$. Thus, this lemma follows. The proof is essentially the same as the convergence time proof of the projected gradient method for set constrained smooth optimization in \cite{book_ConvexOpt_Nesterov}. See Appendix \ref{app:dual_problem_convergence} for the detailed proof. 
\end{IEEEproof}

\subsection{Convergence Time Analysis of Problems with Locally Quadratic Dual Functions}

In addition to Assumptions \ref{as:strongly-convex-problem}-\ref{as:strong-duality}, we further require the next assumption in this subsection.

\begin{Assumption} [Locally Quadratic Dual Functions] \label{as:dual-locally-quadratic}
Let $\boldsymbol{\lambda}^\ast$ be a Lagrange multiplier of problem \eqref{eq:program-objective}-\eqref{eq:program-set-constraint} defined in Assumption \ref{as:strong-duality}. There exists $D_q>0$ and $L_q>0$, where the subscript $q$ denotes locally ``quadratic", such that for all $\boldsymbol{\lambda}\in\{ \boldsymbol{\lambda} \in \mathbb{R}^m_+:\Vert \boldsymbol{\lambda} - \boldsymbol{\lambda}^\ast \Vert \leq D_q\}$, the dual function $q(\boldsymbol{\lambda}) = \min\limits_{\mathbf{x}\in \mathcal{X}}\Big\{f(\mathbf{x})+ \sum_{k=1}^m \lambda_k g_k(\mathbf{x})\Big\}$ satisfies
\begin{align*}
q(\boldsymbol{\lambda}^\ast) \geq q(\boldsymbol{\lambda}) + L_q \Vert \boldsymbol{\lambda} - \boldsymbol{\lambda}^\ast\Vert^2.
\end{align*} 
\end{Assumption}


\begin{Lem}\label{lm:small-dual-difference}
Suppose problem \eqref{eq:program-objective}-\eqref{eq:program-set-constraint} satisfies Assumptions \ref{as:strongly-convex-problem}, \ref{as:strong-duality} and \ref{as:dual-locally-quadratic}. Let $q(\boldsymbol{\lambda}), \boldsymbol{\lambda}^\ast, D_q$ and $L_q$ be defined in Assumption \ref{as:dual-locally-quadratic}. We have the following properties:
\begin{enumerate}
\item If $\boldsymbol{\lambda} \in \mathbb{R}^m_+$ and $q(\boldsymbol{\lambda}^\ast) - q(\boldsymbol{\lambda}) \leq L_{q} D_q^2$, then $\Vert \boldsymbol{\lambda} - \boldsymbol{\lambda}^\ast\Vert \leq D_q$.
\item The Lagrange multiplier defined in Assumption \ref{as:strong-duality} is unique.
\end{enumerate}

\end{Lem}

\begin{IEEEproof}~
\begin{enumerate}
\item Assume not and there exists $\boldsymbol{\lambda}^{\prime}\in\mathbb{R}^{m}_{+}$ such that $q(\boldsymbol{\lambda}^\ast) - q(\boldsymbol{\lambda}^{\prime}) \leq L_{q} D_q^2$ and $\Vert \boldsymbol{\lambda}^{\prime} - \boldsymbol{\lambda}^\ast\Vert  > D_q$. Define $\boldsymbol{\lambda} = (1-\eta) \boldsymbol{\lambda}^\ast + \eta \boldsymbol{\lambda}^{\prime}$ for some $ \eta \in (0,1)$. Note that $\Vert \boldsymbol{\lambda} - \boldsymbol{\lambda}^{\ast}\Vert = \Vert \eta (\boldsymbol{\lambda}^{\prime} - \boldsymbol{\lambda}^{\ast})\Vert = \eta \Vert (\boldsymbol{\lambda}^{\prime} - \boldsymbol{\lambda}^{\ast})\Vert$. Thus, we can choose $\eta\in (0,1)$ such that $\Vert \boldsymbol{\lambda} - \boldsymbol{\lambda}^{\ast}\Vert = D_q$, i.e.,  $ \eta = \frac{D_q}{\Vert \boldsymbol{\lambda}^{\prime} - \boldsymbol{\lambda}^\ast\Vert }$. Note that $\boldsymbol{\lambda}\in \mathbb{R}^{m}_{+}$ because $\boldsymbol{\lambda}^{\prime}\in\mathbb{R}^{m}_{+}$ and $\boldsymbol{\lambda}^{\ast}\in\mathbb{R}^{m}_{+}$.  Since the dual function $q(\cdot)$ is a concave function, we have $q(\boldsymbol{\lambda}) \geq (1-\eta) q(\boldsymbol{\lambda}^{\ast}) + \eta q(\boldsymbol{\lambda}^{\prime}) $. Thus, $q(\boldsymbol{\lambda}^\ast) - q(\boldsymbol{\lambda}) \leq q(\boldsymbol{\lambda}^\ast) -\big( (1-\eta) q(\boldsymbol{\lambda}^{\ast}) + \eta q(\boldsymbol{\lambda}^{\prime}) \big) =  \eta (q(\boldsymbol{\lambda}^{\ast}) - q(\boldsymbol{\lambda}^{\prime})) \leq \eta L_q D_q^2$. This contradicts Assumption \ref{as:dual-locally-quadratic} that $q(\boldsymbol{\lambda}^{\ast}) - q(\boldsymbol{\lambda}) \geq L_{q} \Vert \boldsymbol{\lambda} - \boldsymbol{\lambda}^{\ast}\Vert^{2} = L_{q} D_q^{2}$. 

\item Assume not and there exists ${\boldsymbol{\mu}}^\ast \neq \boldsymbol{\lambda}^\ast$ such that ${\boldsymbol{\mu}}^\ast \in \mathbb{R}^m_+$ and $q(\boldsymbol{\mu}^\ast) = q(\boldsymbol{\lambda}^\ast)$. By part (1), $\Vert {\boldsymbol{\mu}}^\ast - {\boldsymbol{\lambda}}^\ast\Vert \leq D_q$. Thus, we have 
\begin{align*}
q(\boldsymbol{\mu}^{\ast}) &\overset{(a)}{\leq} q(\boldsymbol{\lambda}^{\ast}) - L_{q} \Vert {\boldsymbol{\mu}}^\ast -  \boldsymbol{\lambda}^\ast\Vert^{2} \\
& \overset{(b)}{< }q(\boldsymbol{\lambda}^{\ast})
\end{align*}
where $(a)$ follows from Assumption \ref{as:dual-locally-quadratic} and $(b)$ follows from the assumption that ${\boldsymbol{\mu}}^\ast \neq \boldsymbol{\lambda}^\ast$. This contradicts the assumption that $q(\boldsymbol{\mu}^\ast) = q(\boldsymbol{\lambda}^\ast)$.
\end{enumerate}
\end{IEEEproof}

Define
\begin{align}
T_q =  \frac{\Vert \boldsymbol{\lambda}(0) - \boldsymbol{\lambda}^{\ast}\Vert^2}{2cL_{q}D_q^{2}}, \label{eq:Tq}
\end{align}
where the subscript $q$ denotes locally ``quadratic".

\begin{Lem} \label{lm:local-smooth-arrive-time}
Assume problem \eqref{eq:program-objective}-\eqref{eq:program-set-constraint} satisfies Assumptions \ref{as:strongly-convex-problem}-\ref{as:dual-locally-quadratic}. If $c \leq \frac{\alpha}{\beta^2}$ in Algorithm \ref{alg:dual-subgradient}, then $\Vert \boldsymbol{\lambda}(t) - \boldsymbol{\lambda}^{\ast}\Vert \leq D_q$ for all $t \geq T_q$, where $T_q$ is defined in \eqref{eq:Tq}.
\end{Lem}
\begin{IEEEproof}
By Lemma \ref{lm:dual_problem_convergence} and Lemma \ref{lm:small-dual-difference}, if $\frac{1}{2ct} \Vert \boldsymbol{\lambda}(0) - \boldsymbol{\lambda}^\ast\Vert^2 \leq L_{q} D_q^{2}$, then $\Vert \boldsymbol{\lambda}(t) - \boldsymbol{\lambda}^{\ast}\Vert \leq D_q$.  It can be checked that $t \geq \frac{\Vert \boldsymbol{\lambda}(0) - \boldsymbol{\lambda}^{\ast}\Vert^2}{2cL_{q}D_q^{2}}$ implies that $\frac{1}{2ct} \Vert \boldsymbol{\lambda}(0) - \boldsymbol{\lambda}^\ast\Vert^2 \leq L_{q} D_q^{2}$.
\end{IEEEproof}
\begin{Lem}\label{lm:local-quadratic-dual-arrive-time}
Assume problem \eqref{eq:program-objective}-\eqref{eq:program-set-constraint} satisfies Assumptions \ref{as:strongly-convex-problem}-\ref{as:dual-locally-quadratic}. If $c\leq \frac{\alpha}{\beta^2}$ in Algorithm \ref{alg:dual-subgradient}, then 
\begin{enumerate}
\item $\Vert \boldsymbol{\lambda}(t) - \boldsymbol{\lambda}^\ast\Vert \leq \frac{1}{\sqrt{t}} \frac{1}{\sqrt{2cL_q}} \Vert \boldsymbol{\lambda}(0) - \boldsymbol{\lambda}^\ast\Vert, \forall t\geq T_q$, where $T_q$ is defined in \eqref{eq:Tq}.
\item $\Vert \boldsymbol{\lambda}(t) - \boldsymbol{\lambda}^\ast\Vert \leq \big( \sqrt{\frac{1}{1+2cL_q }} \big)^{t-T_q}\Vert \boldsymbol{\lambda}(T_q) - \boldsymbol{\lambda}^\ast\Vert\leq  \big(\frac{1}{\sqrt{1+2cL_q}}\big)^{t} D_q (1+2cL_q)^{\frac{T_q}{2}}$, $\forall t\geq T_q$ , where $T_q$ is defined in \eqref{eq:Tq}. 
\end{enumerate}
\end{Lem}
\begin{IEEEproof}
\begin{enumerate}
\item By Lemma \ref{lm:dual_problem_convergence}, $q(\boldsymbol{\lambda}^\ast) - q(\boldsymbol{\lambda}(t)) \leq\frac{1}{2ct} \Vert \boldsymbol{\lambda}(0) - \boldsymbol{\lambda}^\ast\Vert^2, \forall t\geq1$. By Lemma \ref{lm:local-smooth-arrive-time} and Assumption \ref{as:dual-locally-quadratic}, $q(\boldsymbol{\lambda}^\ast) - q(\boldsymbol{\lambda}(t)) \geq L_q \Vert \boldsymbol{\lambda} (t)- \boldsymbol{\lambda}^\ast\Vert^2, \forall t\geq T_q$. Thus, we have $L_q \Vert \boldsymbol{\lambda} (t)- \boldsymbol{\lambda}^\ast\Vert^2 \leq \frac{1}{2ct} \Vert \boldsymbol{\lambda}(0) - \boldsymbol{\lambda}^\ast\Vert^2, \forall t\geq T_q$, which implies that  $\Vert \boldsymbol{\lambda}(t) - \boldsymbol{\lambda}^\ast\Vert \leq \frac{1}{\sqrt{t}} \frac{1}{\sqrt{2cL_q}} \Vert \boldsymbol{\lambda}(0) - \boldsymbol{\lambda}^\ast\Vert, \forall t\geq T_q$.
\item By part (1), we know $\Vert \boldsymbol{\lambda}(t) -\boldsymbol{\lambda}^\ast\Vert \leq D_q, \forall t\geq T_q$. The second part is essentially a local version of Theorem 12 in \cite{Necoara15LinearConvergence}, which shows that the projected gradient method for set constrained smooth convex optimization converge geometrically if the objective function satisfies a quadratic growth condition.  See Appendix \ref{app:local-quadratic-dual-geometric-convergence} for the detailed proof.
\end{enumerate}
\end{IEEEproof}

\begin{Cor}\label{cor:local-quadratic-dual-distance-bound}
Assume problem \eqref{eq:program-objective}-\eqref{eq:program-set-constraint} satisfies Assumptions \ref{as:strongly-convex-problem}-\ref{as:dual-locally-quadratic}. If $c\leq \frac{\alpha}{\beta^2}$ in Algorithm \ref{alg:dual-subgradient}, then $\Vert \boldsymbol{\lambda}(2t) - \boldsymbol{\lambda}(t)\Vert \leq 2\big(\frac{1}{\sqrt{1+2cL_q}}\big)^{t} D_q(1+2cL_q)^{\frac{T_q}{2}},\forall t\geq T_q$,
where $T_q$ be defined in \eqref{eq:Tq}.
\end{Cor}
\begin{IEEEproof}
\begin{align*}
&\Vert \boldsymbol{\lambda}(2t) - \boldsymbol{\lambda}(t)\Vert \leq \Vert \boldsymbol{\lambda}(2t) - \boldsymbol{\lambda}^\ast\Vert + \Vert \boldsymbol{\lambda}(t) - \boldsymbol{\lambda}^\ast\Vert  \\
\overset{(a)}{\leq} &\big(\frac{1}{\sqrt{1+2cL_q}}\big)^{2t} D_q (1+2cL_q)^{\frac{T_q}{2}} + \big(\frac{1}{\sqrt{1+2cL_q}}\big)^{t} D_q (1+2cL_q)^{\frac{T_q}{2}}  \\
\overset{(b)}{\leq} & 2\big(\frac{1}{\sqrt{1+2cL_q}}\big)^{t} D_q (1+2cL_q)^{\frac{T_q}{2}},
\end{align*}
where (a) follows from part (2) in Lemma \ref{lm:local-quadratic-dual-arrive-time}; and (b) follows from $\frac{1}{\sqrt{1+2cLq}} <1$.
\end{IEEEproof}
\begin{Thm}\label{thm:local-quadratic-utility-convergence-time}
Assume problem \eqref{eq:program-objective}-\eqref{eq:program-set-constraint} satisfies Assumptions \ref{as:strongly-convex-problem}-\ref{as:dual-locally-quadratic}. Let $\mathbf{x}^\ast$ be the optimal solution and $\boldsymbol{\lambda}^\ast$ be defined in Assumption \ref{as:dual-locally-quadratic}.  If $c \leq \frac{\alpha}{\beta^2}$ in Algorithm \ref{alg:dual-subgradient}, then $$f(\widetilde{\mathbf{x}}(2t)) \leq   f(\mathbf{x}^\ast) + \frac{1}{t}\big(\frac{1}{\sqrt{1+2cL_q}}\big)^{t}\eta_q,\forall t\geq T_q,$$ where $$\eta_q = \frac{2D_{q}^{2} (1+2cL_q)^{T_{q}} +2D_{q} (1+2cL_q)^{\frac{T_{q}}{2}}(\sqrt{\Vert \boldsymbol{\lambda}(0)\Vert^{2} + \Vert \boldsymbol{\lambda}^{\ast}\Vert^{2}} + \Vert \boldsymbol{\lambda}^{\ast}\Vert)}{c}$$ and $T_q$ is defined in \eqref{eq:Tq}.
\end{Thm}
\begin{IEEEproof}
Fix $t  \geq T_q$. By Lemma \ref{lm:drift-plus-penalty}, we have $\frac{1}{c}\Delta(\tau) + f(\mathbf{x}(\tau)) \leq  f(\mathbf{x}^\ast)$ for all $\tau \in\{0,1,\ldots\}$. Summing over $\tau \in \{t,t+1,\ldots,2t-1\}$ yields $\frac{1}{c}\sum_{\tau=t}^{2t-1} \Delta(\tau) +  \sum_{\tau=t}^{2t-1} f(\mathbf{x}(\tau)) \leq t f(\mathbf{x}^\ast)$. Dividing by factor $t$ yields 
\begin{align}
\frac{1}{t}\sum_{\tau=t}^{2t-1} f(\mathbf{x}(\tau)) \leq f(\mathbf{x}^\ast) + \frac{L(t) - L(2t)}{ct}  \label{eq:proof-thm5-eq1}
\end{align}
Thus, we have
\begin{align*}
&\quad f(\widetilde{\mathbf{x}}(2t)) \overset{(a)}{\leq} \frac{1}{t}\sum_{\tau=t}^{2t-1} f(\mathbf{x}(\tau)) \overset{(b)}{\leq} f(\mathbf{x}^\ast) + \frac{ L(t) - L(2t)}{ct}\\
&= f(\mathbf{x}^\ast) + \frac{\Vert \boldsymbol{\lambda}(t)\Vert^2 - \Vert \boldsymbol{\lambda}(2t)\Vert^2}{2ct}\\
&= f(\mathbf{x}^\ast) + \frac{\Vert \boldsymbol{\lambda}(t) - \boldsymbol{\lambda}(2t) + \boldsymbol{\lambda}(2t)\Vert^2 - \Vert \boldsymbol{\lambda}(2t)\Vert^2}{2ct}\\
&\overset{(c)}{\leq} f(\mathbf{x}^\ast) + \frac{ \Vert \boldsymbol{\lambda}(t) - \boldsymbol{\lambda}(2t)\Vert^2 + 2 \Vert \boldsymbol{\lambda}(2t)\Vert \Vert\boldsymbol{\lambda}(t) - \boldsymbol{\lambda}(2t)\Vert}{2ct}\\
&\overset{(d)}{\leq} f(\mathbf{x}^\ast) + \frac{ \Big(2\big(\frac{1}{\sqrt{1+2cL_q}}\big)^{t} D_q (1+2cL_q)^{\frac{T_q}{2}}\Big)^2 }{2ct}+ \frac{4 \big(\frac{1}{\sqrt{1+2cL_q}}\big)^{t} D_q (1+2cL_q)^{\frac{T_q}{2}}  \Vert \boldsymbol{\lambda}(2t)\Vert}{2ct}\\
&\overset{(e)}{\leq} f(\mathbf{x}^\ast) + \frac{1}{t} \big(\frac{1}{\sqrt{1+2cL_q}}\big)^t \Big(\frac{2D_q^2 (1+2cL_q)^{T_q}}{c} + \frac{2D_{q} (1+2cL_q)^{\frac{T_{q}}{2}} \Vert \boldsymbol{\lambda}(2t)\Vert}{c}\Big)\\
&\overset{(f)}{\leq} f(\mathbf{x}^\ast) + \frac{1}{t}\big(\frac{1}{\sqrt{1+2cL_q}}\big)^{t}\eta_q
\end{align*} 
where (a) follows from $\widetilde{\mathbf{x}}(2t) = \frac{1}{t}\sum_{\tau=t}^{2t-1} \mathbf{x}(\tau)$ and  the convexity of $f(\mathbf{x})$; (b) follows from \eqref{eq:proof-thm5-eq1}; (c) follows from the Cauchy-Schwarz inequality; (d) follows from Corollary  \ref{cor:local-quadratic-dual-distance-bound}; (e) follows from  $\frac{1}{\sqrt{1+2cL_q}} <1$; and (f) follows from $\Vert \boldsymbol{\lambda}(2t)\Vert \leq \sqrt{\Vert \boldsymbol{\lambda}(0)\Vert^{2} + \Vert \boldsymbol{\lambda}^{\ast}\Vert^{2}} + \Vert \boldsymbol{\lambda}^{\ast}\Vert$ and the definition of $\eta_q$.
\end{IEEEproof}

\begin{Thm}\label{thm:local-quadratic-constraint-convergence-time}
Assume problem \eqref{eq:program-objective}-\eqref{eq:program-set-constraint} satisfies Assumptions \ref{as:strongly-convex-problem}-\ref{as:dual-locally-quadratic}. If $c\leq\frac{\alpha}{\beta^2}$ in Algorithm \ref{alg:dual-subgradient}, then $$g_k (\widetilde{\mathbf{x}}(2t)) \leq \frac{2D_{q} (1+2cL_q)^{\frac{T_{q}}{2}}}{c}\frac{1}{t} \big(\frac{1}{\sqrt{1+2cL_q}}\big)^{t} , \forall k\in\{1,2,\ldots,m\}, \forall t\geq T_q,$$ where $T_q$ is defined in \eqref{eq:Tq}.
\end{Thm}
\begin{IEEEproof}
Fix $t\geq T_q$ and $k\in\{1,2,\ldots,m\}$. Thus, we have
\begin{align*}
g_k(\widetilde{\mathbf{x}}(2t)) &\overset{(a)}{\leq} \frac{1}{t} \sum_{\tau=t}^{2t-1} g_k(\mathbf{x}(\tau)) \overset{(b)}{\leq} \frac{1}{ct} \big(\lambda_k(2t) - \lambda_k(t)\big) \\
&\leq \frac{1}{ct} \Vert \boldsymbol{\lambda}(2t) - \boldsymbol{\lambda}(t)\Vert \\
&\overset{(c)}{\leq} \frac{2D_{q} (1+2cL_q)^{\frac{T_{q}}{2}}}{ct} \big(\frac{1}{\sqrt{1+2cL_q}}\big)^{t}
\end{align*}
where $(a)$ follows from the convexity of $g_k(\mathbf{x})$; $(b)$ follows from Lemma \ref{lm:queue_constraint_inequality}; and $(c)$ follows from Corollary \ref{cor:local-quadratic-dual-distance-bound}.
\end{IEEEproof}

Under Assumptions \ref{as:strongly-convex-problem}-\ref{as:dual-locally-quadratic}, Theorems \ref{thm:local-quadratic-utility-convergence-time} and \ref{thm:local-quadratic-constraint-convergence-time} show that if $c \leq \frac{\alpha}{\beta^2}$,  then $\widetilde{\mathbf{x}}(t)$ provides an $\epsilon$-approximate solution with convergence time $O(\log(\frac{1}{\epsilon}))$. This is formally summarized in the next theorem:

\begin{Thm}
Assume problem \eqref{eq:program-objective}-\eqref{eq:program-set-constraint} satisfies Assumptions \ref{as:strongly-convex-problem}-\ref{as:dual-locally-quadratic}. Let $\mathbf{x}^\ast \in \mathcal{X}$ be the optimal solution and $\boldsymbol{\lambda}^\ast \geq \mathbf{0}$ be the Lagrange multiplier defined in Assumption \ref{as:strong-duality}.   If $c \leq \frac{\alpha}{\beta^2}$ in Algorithm \ref{alg:dual-subgradient}, then for all $t\geq T_q$, 
\begin{align*}
f(\tilde{\mathbf{x}}(2t)) \leq   f(\mathbf{x}^\ast) + \frac{1}{t}\big(\frac{1}{\sqrt{1+2cL_q}}\big)^{t}\eta_q,
\end{align*}
\begin{align*}
g_k (\widetilde{\mathbf{x}}(2t)) \leq \frac{2D_{q} (1+2cL_q)^{\frac{T_{q}}{2}}}{c}\frac{1}{t} \big(\frac{1}{\sqrt{1+2cL_q}}\big)^{t} , \forall k\in\{1,2,\ldots,m\}
\end{align*}
where $\eta_q = \frac{2D_{q}^{2} (1+2cL_q)^{T_{q}} +2D_{q} (1+2cL_q)^{\frac{T_{q}}{2}}(\sqrt{\Vert \boldsymbol{\lambda}(0)\Vert^{2} + \Vert \boldsymbol{\lambda}^{\ast}\Vert^{2}} + \Vert \boldsymbol{\lambda}^{\ast}\Vert)}{c}$ is a fixed constant and $T_c$ is defined in \eqref{eq:Tc}. In summary,  if $c\leq \frac{\alpha}{\beta^2}$ in Algorithm \ref{alg:dual-subgradient}, then $\widetilde{\mathbf{x}}(2t)$ ensures error decays like $O\big(\frac{1}{t}(\frac{1}{\sqrt{1+2cL_q}})^{t}\big)$ and provides an $\epsilon$-approximiate solution with convergence time $O(\log(\frac{1}{\epsilon}))$.
\end{Thm}

\subsection{Convergence Time Analysis of Problems with Locally Strongly Concave Dual Functions}

The following assumption is stronger than  Assumptions \ref{as:dual-locally-quadratic} but can be easier to verify in certain cases.  For example, if the dual function of the convex program is available, Assumption \ref{as:dual-locally-stronlgy-concave} is easier to verify, e.g., by studying the Hessian of the dual function.

\begin{Assumption} [Locally Strongly Concave Dual Functions] \label{as:dual-locally-stronlgy-concave}
Let $\boldsymbol{\lambda}^\ast$ be a Lagrange multiplier vector defined in Assumption \ref{as:strong-duality}. There exists $D_c>0$ and $L_c>0$, where the subscript $c$ denotes locally strongly ``concave", such that the dual function $q(\boldsymbol{\lambda})$ is strongly concave with modulus $L_c$ over $\{ \boldsymbol{\lambda} \in\mathbb{R}^m_+:\Vert \boldsymbol{\lambda} - \boldsymbol{\lambda}^\ast \Vert \leq D_c\}$.
\end{Assumption}

The next lemma summarizes that Assumption \ref{as:dual-locally-stronlgy-concave} implies Assumption \ref{as:dual-locally-quadratic}.

\begin{Lem} \label{lm:strongly-concave-imply-locally-quadratic}
If problem \eqref{eq:program-objective}-\eqref{eq:program-set-constraint} satisfies Assumption \ref{as:dual-locally-stronlgy-concave}, then it also satisfies Assumption \ref{as:dual-locally-quadratic} with $D_q = D_c$ and $L_q = \frac{L_c}{2}$.
\end{Lem}
\begin{IEEEproof}
Since $q(\cdot)$ is strongly concave, $\tilde{q}(\boldsymbol{\lambda}) = -q(\boldsymbol{\lambda})$ is strongly convex and is minimized at $\boldsymbol{\lambda}^\ast$. By Lemma \ref{lm:first-order-optimality}, there exists $\mathbf{d}\in \partial \tilde{q}(\boldsymbol{\lambda}^\ast)$ such that $\mathbf{d}^T ( \boldsymbol{\lambda} - \boldsymbol{\lambda}^\ast) \geq 0$ for all $\boldsymbol{\lambda}\geq \mathbf{0}$.  By Lemma \ref{lm:strong-convex}, we further have $\tilde{q}(\boldsymbol{\lambda}^\ast) - \tilde{q}(\boldsymbol{\lambda}) \leq \mathbf{d}^T (\boldsymbol{\lambda}^\ast - \boldsymbol{\lambda}) - \frac{L_c}{2} \Vert \boldsymbol{\lambda} - \boldsymbol{\lambda}^\ast\Vert^2 \leq  - \frac{L_c}{2} \Vert \boldsymbol{\lambda} - \boldsymbol{\lambda}^\ast\Vert^2$ for all $\boldsymbol{\lambda}\in \{ \boldsymbol{\lambda} \in\mathbb{R}^m_+:\Vert \boldsymbol{\lambda} - \boldsymbol{\lambda}^\ast \Vert \leq D_c\}$. Thus, $q(\boldsymbol{\lambda}^\ast) \geq q(\boldsymbol{\lambda})+ \frac{L_c}{2} \Vert \boldsymbol{\lambda} - \boldsymbol{\lambda}^\ast\Vert^2 $ for all $\boldsymbol{\lambda}\in \{ \boldsymbol{\lambda} \in\mathbb{R}^m_+:\Vert \boldsymbol{\lambda} - \boldsymbol{\lambda}^\ast \Vert \leq D_c\}$.
\end{IEEEproof}

Since Assumption \ref{as:dual-locally-stronlgy-concave} implies Assumption \ref{as:dual-locally-quadratic}, by the results from the previous subsection,  $\widetilde{\mathbf{x}}(t)$ from Algorithm \ref{alg:dual-subgradient} provides an $\epsilon$-approximate solution with convergence time $O(\log(\frac{1}{\epsilon}))$.  In this subsection, we show that if problem \eqref{eq:program-objective}-\eqref{eq:program-set-constraint} satisfies Assumption \ref{as:dual-locally-stronlgy-concave}, then the geometric error decay of both objective violations and constraint violations has a smaller contraction modulus.  

The next lemma relates the smoothness of the dual function and Assumption \ref{as:dual-locally-stronlgy-concave}. 
\begin{Lem}\label{lm:smooth-modulus-strong-concave-modulus}
If function $h$ is both smooth with modulus $\gamma$ and strongly concave with modulus $L_c$ over set $\mathcal{X}$, which is not a singleton, then $ L_c\leq \gamma$. 
\end{Lem}
\begin{proof}
This is a basic fact in convex analysis. See Appendix \ref{app:smooth-modulus-strong-concave-modulus} for the detailed proof.
\end{proof}

Recall that the dual function $q(\boldsymbol{\lambda})$ of problem \eqref{eq:program-objective}-\eqref{eq:program-set-constraint} is smooth with modulus $\gamma =\frac{\beta^2}{\alpha}$. Thus, Lemma \ref{lm:smooth-modulus-strong-concave-modulus} implies that $L_c \leq \frac{\beta^2}{\alpha}$, which further implies that $1 - cL_c \geq 0$ as long as $c \leq \frac{\alpha}{\beta^2}$.

For any problem \eqref{eq:program-objective}-\eqref{eq:program-set-constraint} satisfying Assumptions \ref{as:strongly-convex-problem}-\ref{as:strong-duality} and \ref{as:dual-locally-stronlgy-concave}, we define 
\begin{align}
T_c =  \frac{\Vert \boldsymbol{\lambda}(0) - \boldsymbol{\lambda}^{\ast}\Vert^2}{cL_{c}D_c^{2}}, \label{eq:Tc}
\end{align}
where the subscript $c$ denotes locally strongly ``concave".

\begin{Lem} \label{lm:local-strongly-concave-arrive-time}
Assume problem \eqref{eq:program-objective}-\eqref{eq:program-set-constraint} satisfies Assumptions \ref{as:strongly-convex-problem}-\ref{as:strong-duality} and \ref{as:dual-locally-stronlgy-concave}. Let $D_c$ and $L_{c}$ be defined in Assumption \ref{as:dual-locally-stronlgy-concave}. If  $c\leq  \frac{\alpha}{\beta^2}$ in Algorithm \ref{alg:dual-subgradient}, then   
\begin{enumerate}
\item $\Vert \boldsymbol{\lambda}(t) - \boldsymbol{\lambda}^{\ast}\Vert \leq D_c$ for all $t \geq T_c$ , where $T_c$ is defined in \eqref{eq:Tc}. 
\item $\Vert \boldsymbol{\lambda}(t) - \boldsymbol{\lambda}^\ast\Vert \leq \big( \sqrt{1- cL_c} \big)^{t-T_c}\Vert \boldsymbol{\lambda}(T_c) - \boldsymbol{\lambda}^\ast\Vert\leq  \big(\sqrt{1- cL_c}\big)^{t} \frac{D_c}{(\sqrt{1- c L_c})^{T_c}}$, $\forall t\geq T_c$ , where $T_c$ is defined in \eqref{eq:Tc}. 
\end{enumerate}
\end{Lem}
\begin{IEEEproof}~
\begin{enumerate}
\item By Lemma \ref{lm:strongly-concave-imply-locally-quadratic}, $q(\cdot)$ is locally quadratic with $D_q = D_c$ and $L_{q} = \frac{L_{c}}{2}$. The remaining part of the proof is identical to the proof of Lemma \ref{lm:local-smooth-arrive-time}.
\item By part (1) of this lemma, $\boldsymbol{\lambda}(t) \in \{ \boldsymbol{\lambda} \in\mathbb{R}^m_+:\Vert \boldsymbol{\lambda} - \boldsymbol{\lambda}^\ast \Vert \leq D_c\}, \forall t\geq T_{c}$.  That is, the dynamic of $\boldsymbol{\lambda}(t), t\geq T_{c}$ is the same as that in the projected gradient method with step size $c$ to solve\footnote{Recall that the projected gradient method with constant step size when applied to set constrained smooth and strongly convex optimization converges to the optimal solution at the rate $O(\kappa^{t})$ where $\kappa$ is a parameter depending on the step size, smoothness modulus and strong convexity modulus \cite{book_ConvexOpt_Nesterov}.} $\max_{\boldsymbol{\lambda} \in \{ \boldsymbol{\lambda} \in\mathbb{R}^m_+:\Vert \boldsymbol{\lambda} - \boldsymbol{\lambda}^\ast \Vert \leq D_c\}} \big\{q(\boldsymbol{\lambda})\big\}$. Thus, the part is essentially a local version of the convergence time result of the projected gradient method for set constrained smooth and strongly convex optimization \cite{book_ConvexOpt_Nesterov}. See Appendix \ref{app:local-strongly-concave-dual-geometric-convergence}
 for the detailed proof.

\end{enumerate}
\end{IEEEproof}

The next corollary follows from part (2) of Lemma \ref{lm:local-strongly-concave-arrive-time}.

\begin{Cor}\label{cor:locally-strongly-concave-dual-distance-bound}
Assume problem \eqref{eq:program-objective}-\eqref{eq:program-set-constraint} satisfies Assumptions \ref{as:strongly-convex-problem}-\ref{as:strong-duality} and \ref{as:dual-locally-stronlgy-concave}. If $c\leq \frac{\alpha}{\beta^2}$ in Algorithm \ref{alg:dual-subgradient}, then
\begin{align*}
\Vert \boldsymbol{\lambda}(2t) - \boldsymbol{\lambda}(t)\Vert &\leq \big(\sqrt{1- cL_c}\big)^{t} \frac{2D_c}{(\sqrt{1-cL_c})^{T_c}}, \forall t\geq T_c ,
\end{align*}
where $T_c$ is defined in \eqref{eq:Tc}. 
\end{Cor}

\begin{IEEEproof}~
\begin{align*}
&\Vert \boldsymbol{\lambda}(2t) - \boldsymbol{\lambda}(t)\Vert \\
\leq &\Vert \boldsymbol{\lambda}(2t) - \boldsymbol{\lambda}^\ast\Vert + \Vert \boldsymbol{\lambda}(t) - \boldsymbol{\lambda}^\ast\Vert  \\
\overset{(a)}{\leq} &\big(\sqrt{1-cL_c}\big)^{2t} \frac{D_c}{(\sqrt{1-cL_c})^{T_c}}+ \big(\sqrt{1-cL_c}\big)^{t} \frac{D_c}{(\sqrt{1- cL_c})^{T_c}}  \\
\overset{(b)}{\leq} &\big(\sqrt{1-cL_c}\big)^{t} \frac{2D_c}{(\sqrt{1-cL_c})^{T_c}}
\end{align*}
where $(a)$ follows from part (2) of Lemma \ref{lm:local-strongly-concave-arrive-time} and $(b)$ follows from the fact that $\sqrt{1 - cL_{c}} < 1$.
\end{IEEEproof}
\begin{Thm}\label{thm:locally-strongly-concave-utility-convergence-time}
Assume problem \eqref{eq:program-objective}-\eqref{eq:program-set-constraint} satisfies Assumptions \ref{as:strongly-convex-problem}-\ref{as:strong-duality} and \ref{as:dual-locally-stronlgy-concave}. Let $\mathbf{x}^\ast \in \mathcal{X}$ be the optimal solution and $\boldsymbol{\lambda}^\ast$ be the Lagrange multiplier defined in Assumption \ref{as:strong-duality}.  If $c \leq \frac{\alpha}{\beta^2}$ in Algorithm \ref{alg:dual-subgradient}, then
\begin{align*}
f(\widetilde{\mathbf{x}}(2t)) \leq  f(\mathbf{x}^\ast) + \frac{1}{t}\left(\sqrt{1- cL_c}\right)^{t} \eta_c,\quad \forall t\geq T_c,
\end{align*}
where $\eta_c =  \frac{1}{c}\frac{2D_c^2}{(\sqrt{1-cL_c})^{2T_c}} + \frac{1}{c}\frac{2 D_c \big(\sqrt{\Vert \boldsymbol{\lambda}(0)\Vert^2 + \Vert \boldsymbol{\lambda}^\ast\Vert^2} + \Vert \boldsymbol{\lambda}^\ast\Vert\big)}{ (\sqrt{1-cL_c})^{T_c}}$ is a fixed constant and $T_c$ is defined in \eqref{eq:Tc}. 
\end{Thm}
\begin{IEEEproof}
Fix $t  \geq T_c$. By Lemma \ref{lm:drift-plus-penalty}, we have $\frac{1}{c}\Delta(\tau) + f(\mathbf{x}(\tau)) \leq   f(\mathbf{x}^\ast)$ for all $\tau \in\{0,1,\ldots\}$. Summing over $\tau \in \{t,t+1,\ldots,2t-1\}$ we have:
\begin{align}
&\frac{1}{c}\sum_{\tau=t}^{2t-1} \Delta(\tau) + \sum_{\tau=t}^{2t-1} f(\mathbf{x}(\tau)) \leq t f(\mathbf{x}^\ast) \nonumber \\
\Rightarrow\quad& \frac{1}{c}[L(2t) - L(t)] + \sum_{\tau=t}^{2t-1} f(\mathbf{x}(\tau)) \leq t f(\mathbf{x}^\ast)  \nonumber\\
\Rightarrow\quad& \frac{1}{t}\sum_{\tau=t}^{2t-1} f(\mathbf{x}(\tau)) \leq f(\mathbf{x}^\ast) + \frac{L(t) - L(2t)}{ct}  \label{eq:proof-thm8-eq1}
\end{align}
Thus, we have
\begin{align*}
&f(\widetilde{\mathbf{x}}(2t))\overset{(a)}{\leq} \frac{1}{t}\sum_{\tau=t}^{2t-1} f(\mathbf{x}(\tau)) \overset{(b)}{\leq} f(\mathbf{x}^\ast) + \frac{ L(t) - L(2t)}{ct} \\
= &f(\mathbf{x}^\ast) + \frac{\Vert \boldsymbol{\lambda}(t)\Vert^2 - \Vert\boldsymbol{\lambda}(2t)\Vert^2}{2ct}\\
=& f(\mathbf{x}^\ast) + \frac{\Vert \boldsymbol{\lambda}(t) - \boldsymbol{\lambda}(2t) + \boldsymbol{\lambda}(2t)\Vert^2 - \Vert \boldsymbol{\lambda}(2t)\Vert^2}{2ct} \\
= &f(\mathbf{x}^\ast) + \frac{\Vert \boldsymbol{\lambda}(t) - \boldsymbol{\lambda}(2t)\Vert^2 + 2[\boldsymbol{\lambda}(2t)]^T [\boldsymbol{\lambda}(t) - \boldsymbol{\lambda}(2t)]}{2ct} \\
\overset{(c)}{\leq}& f(\mathbf{x}^\ast) + \frac{ \Vert \boldsymbol{\lambda}(t) - \boldsymbol{\lambda}(2t)\Vert^2 + 2 \Vert \boldsymbol{\lambda}(2t)\Vert \Vert\boldsymbol{\lambda}(t) - \boldsymbol{\lambda}(2t)\Vert}{2ct} \\
\overset{(d)}{\leq}& f(\mathbf{x}^\ast) +  \frac{\Big(\big(\sqrt{1-cL_c}\Big)^{t} \frac{2D_c}{(\sqrt{1-cL_c})^{T_c}}\Big)^2}{2ct} +  \frac{2 \Big(\big(\sqrt{1-cL_c}\big)^{t} \frac{2D_c}{(\sqrt{1-cL_c})^{T_c}}\Big)\Vert \boldsymbol{\lambda}(2t)\Vert}{2ct} \\
\overset{(e)}{\leq}& f(\mathbf{x}^\ast) + \frac{1}{t}\frac{\Big(\sqrt{1-cL_c}\Big)^{t} \Big( \frac{2D_c^2}{(\sqrt{1-cL_c})^{2T_c}} + \frac{2 D_c \Vert \boldsymbol{\lambda}(2t)\Vert}{(\sqrt{1-cL_c})^{T_c}} \Big)}{c} \\
 \overset{(f)}{=}& f(\mathbf{x}^\ast) + \frac{1}{t}\big(\sqrt{1-cL_c}\big)^{t}\eta_c
\end{align*} 
where (a) follows from the fact that $\widetilde{\mathbf{x}}(2t) = \frac{1}{t}\sum_{\tau=t}^{2t-1} \mathbf{x}(\tau)$ and  the convexity of $f(\mathbf{x})$; (b) follows from \eqref{eq:proof-thm8-eq1}; (c) follows from the Cauchy-Schwarz inequality; (d) is true because $$\Vert \boldsymbol{\lambda}(2t) - \boldsymbol{\lambda}(t)\Vert \leq \big(\sqrt{1- cL_c}\big)^{t} \frac{2D_c}{(\sqrt{1-cL_c})^{T_c}}, \forall t\geq T_c$$ by Corollary \ref{cor:locally-strongly-concave-dual-distance-bound}; (e) follows from the fact that $\sqrt{1-cL_c} < 1$ and (f) follows from the fact that $\Vert \boldsymbol{\lambda}(2t)\Vert \leq \sqrt{\Vert \boldsymbol{\lambda}(0)\Vert^2 +\Vert \boldsymbol{\lambda}^\ast\Vert^2} + \Vert \boldsymbol{\lambda}^\ast\Vert$ by Lemma \ref{lm:queue-bound} and the definition of $\eta_c$.
\end{IEEEproof}

\begin{Thm}\label{thm:locally-strongly-concave-constraint-convergence-time}
Assume problem \eqref{eq:program-objective}-\eqref{eq:program-set-constraint} satisfies Assumptions \ref{as:strongly-convex-problem}-\ref{as:strong-duality} and \ref{as:dual-locally-stronlgy-concave}. If $c \leq \frac{\alpha}{\beta^2}$ in Algorithm \ref{alg:dual-subgradient}, then $$g_k (\widetilde{\mathbf{x}}(2t)) \leq \frac{1}{t}\big(\sqrt{1-cL_c}\big)^{t} \frac{2D_c}{c(\sqrt{1-cL_c})^{T_c}} , \forall k\in\{1,2,\ldots,m\}, \forall t\geq T_c,$$ where $T_c$ is defined in \eqref{eq:Tc}. 
\end{Thm}
\begin{IEEEproof}
Fix $t\geq T_c$ and $k\in\{1,2,\ldots,m\}$. Thus, we have
\begin{align*}
g_k(\widetilde{\mathbf{x}}(2t)) &\overset{(a)}{\leq} \frac{1}{t} \sum_{\tau=t}^{2t-1} g_k(\mathbf{x}(\tau)) \\
&\overset{(b)}{\leq} \frac{1}{ct} [\lambda_k(2t) - \lambda_k(t)] \\
&\leq \frac{1}{ct} \Vert \boldsymbol{\lambda}(2t) - \boldsymbol{\lambda}(t)\Vert \\
&\overset{(c)}{\leq} \frac{1}{t}\big(\sqrt{1-cL_c}\big)^{t} \frac{2D_c}{c(\sqrt{1-cL_c})^{T_c}}
\end{align*}
where (a) follows from the convexity of $g_k(\cdot)$; (b) follows from Lemma \ref{lm:queue_constraint_inequality}; and (c) follows from Corollary \ref{cor:locally-strongly-concave-dual-distance-bound}.
\end{IEEEproof}

Under Assumptions \ref{as:strongly-convex-problem}-\ref{as:strong-duality} and \ref{as:dual-locally-stronlgy-concave}, Theorems \ref{thm:locally-strongly-concave-utility-convergence-time} and \ref{thm:locally-strongly-concave-constraint-convergence-time} show that if $c \leq \frac{\alpha}{\beta^2}$,  then $\widetilde{\mathbf{x}}(t)$ provides an $\epsilon$-approximiate solution with convergence time $O(\log(\frac{1}{\epsilon}))$. Since $L_q = \frac{L_c}{2}$ by Lemma \ref{lm:strongly-concave-imply-locally-quadratic} and note that $\sqrt{1-cL_c} \leq \frac{1}{\sqrt{1+2cL_q}}$, the geometric contraction modulus shown in Theorems \ref{thm:locally-strongly-concave-utility-convergence-time} and \ref{thm:locally-strongly-concave-constraint-convergence-time} under Assumption \ref{as:dual-locally-stronlgy-concave} is smaller than the geometric contraction modulus shown in Theorems \ref{thm:local-quadratic-utility-convergence-time} and \ref{thm:local-quadratic-constraint-convergence-time} under Assumption \ref{as:dual-locally-quadratic}.

The next theorem summarizes both Theorem \ref{thm:locally-strongly-concave-utility-convergence-time} and Theorem \ref{thm:locally-strongly-concave-constraint-convergence-time}.
\begin{Thm}
Assume problem \eqref{eq:program-objective}-\eqref{eq:program-set-constraint} satisfies Assumptions \ref{as:strongly-convex-problem}-\ref{as:strong-duality} and \ref{as:dual-locally-stronlgy-concave}. Let $\mathbf{x}^\ast \in \mathcal{X}$ be the optimal solution and $\boldsymbol{\lambda}^\ast \geq \mathbf{0}$ be the Lagrange multiplier defined in Assumption \ref{as:strong-duality}.   If $c \leq \frac{\alpha}{\beta^2}$ in Algorithm \ref{alg:dual-subgradient}, then for all $t\geq T_c$, 
\begin{align*}
f(\tilde{\mathbf{x}}(2t)) \leq  f(\mathbf{x}^\ast) + \frac{1}{t}\big(\sqrt{1-cL_c}\big)^{t} \eta_c,
\end{align*}
\begin{align*}
g_k (\widetilde{\mathbf{x}}(2t)) \leq \frac{1}{t}\big(\sqrt{1-cL_c}\big)^{t} \frac{2D_c}{c(\sqrt{1-cL_c})^{T_c}} , \forall k\in\{1,2,\ldots,m\}, 
\end{align*}
where $\eta_c =  \frac{1}{c}\frac{2D_c^2}{(\sqrt{1-cL_c})^{2T_c}} + \frac{1}{c}\frac{2 D_c \big(\sqrt{\Vert \boldsymbol{\lambda}(0)\Vert^2 + \Vert \boldsymbol{\lambda}^\ast\Vert^2} + \Vert \boldsymbol{\lambda}^\ast\Vert\big)}{ (\sqrt{1-cL_c})^{T_c}}$ is a fixed constant and $T_c$ is defined in \eqref{eq:Tc}. In summary,  if $c\leq \frac{\alpha}{\beta^2}$ in Algorithm \ref{alg:dual-subgradient}, then $\widetilde{\mathbf{x}}(2t)$ ensures error decays like $O\big(\frac{1}{t}(\sqrt{1-cL_c})^{t}\big)$ and provides an $\epsilon$-approximiate solution with convergence time $O(\log(\frac{1}{\epsilon}))$.
\end{Thm}

\subsection{Discussions}

\subsubsection{Practical Implementations}
Assumptions \ref{as:dual-locally-quadratic} and \ref{as:dual-locally-stronlgy-concave} in general are difficult to verify. However, we note that to ensure $\widetilde{\mathbf{x}}(t)$ provides the better $O(\log(\frac{1}{\epsilon}))$ convergence time, we only require $c\leq\frac{\alpha}{\beta^2}$, which is independent of the parameters in Assumptions \ref{as:dual-locally-quadratic} or \ref{as:dual-locally-stronlgy-concave}.  Namely, in practice, we can blindly apply Algorithm \ref{alg:dual-subgradient} to problem \eqref{eq:program-objective}-\eqref{eq:program-set-constraint} with no need to verify Assumption \ref{as:dual-locally-quadratic} or \ref{as:dual-locally-stronlgy-concave}.  If problem \eqref{eq:program-objective}-\eqref{eq:program-set-constraint} happens to satisfy Assumptions \ref{as:dual-locally-quadratic} or \ref{as:dual-locally-stronlgy-concave}, then $\widetilde{\mathbf{x}}(t)$ enjoys the faster  convergence time $O(\log(\frac{1}{\epsilon}))$. If not,  then $\widetilde{\mathbf{x}}(t)$ (or $\overline{\mathbf{x}}(t)$) at least has convergence time $O(\frac{1}{\epsilon})$. 

\subsubsection{Local Assumption and Local Geometric Convergence}
Since Assumption \ref{as:dual-locally-quadratic} only requires the ``quadratic" property to be satisfied in a local radius $D_{q}$ around $\boldsymbol{\lambda}^\ast$, the error of Algorithm \ref{alg:dual-subgradient} starts to decay like $O\Big(\frac{1}{t} \big(\frac{1}{\sqrt{1+2cL_q}}\big)^t\Big)$ only after $\lambda(t)$ arrives at the $D_{q}$ local radius after $T_{q}$ iterations, where $T_{q}$ is independent of the approximation requirement $\epsilon$ and hence is order $O(1)$. Thus, Algorithm \ref{alg:dual-subgradient} provides an $\epsilon$-approximate solution  with convergence time $O(\log(\frac{1}{\epsilon})$. However, it is possible that $T_{q}$ is relatively large if $D_q$ is small. 

In fact, $T_{q} > 0$ because Assumption \ref{as:dual-locally-quadratic} only requires the dual function to have the ``quadratic'' property in a local radius. Thus, the theory developed in this section can deal with a large class of problems. On the other hand, if the dual function has the ``quadratic'' property globally, i.e., for all $\boldsymbol{\lambda}\geq \mathbf{0}$, then $T_{q} = 0$ and the error of Algorithm \ref{alg:dual-subgradient} decays like $O\Big(\frac{1}{t} \big(\frac{1}{\sqrt{1+2cL_q}}\big)^t\Big), \forall t\geq 1$.

A similar tradeoff holds with respect to Assumption \ref{as:dual-locally-stronlgy-concave}.

\section{Applications}
\subsection{Convex Programs Satisfying Non-Degenerate Constraint Qualifications}
\begin{Thm}\label{thm:strongly-concave-dual-fucntion}
Consider the following strongly convex program:
\begin{align}
\min \quad &f(\mathbf{x}) \\
\text{s.t.} \quad  &  g_k(\mathbf{x})\leq 0, \forall k\in\{1,2,\ldots,m\} \\
			 &  \mathbf{x}\in \mathbb{R}^n
\end{align}
where $f(\mathbf{x})$ is a second-order continuously differentiable and strongly convex function; $g_k(\mathbf{x}), \forall k \in\{1,2,\ldots,m\}$ are Lipschitz continuous, second-order continuously differentiable and convex functions. Let $\mathbf{x}^\ast$ be the unique solution to this strongly convex program. 
\begin{enumerate}
\item Let $\mathcal{K}\subseteq\{1,2,\ldots,m\}$ be the set of active constraints, i.e., $g_k(\mathbf{x}^\ast) = 0, \forall k\in\mathcal{K}$, and denote the vector composed by $g_k(\mathbf{x}), k\in \mathcal{K}$ as $\mathbf{g}_\mathcal{K}$. If $\mathbf{g}(\mathbf{x})$ has a bounded Jacobian and 
$\text{rank}(\nabla_{\mathbf{x}} \mathbf{g}_{\mathcal{K}}(\mathbf{x}^\ast)^{T}) = |\mathcal{K}|$, then Assumptions \ref{as:strongly-convex-problem}-\ref{as:dual-locally-quadratic} hold for this problem. 
\item If $\mathbf{g}(\mathbf{x})$ has a bounded Jacobian and 
$\text{rank}(\nabla_{\mathbf{x}} \mathbf{g}(\mathbf{x}^\ast)^{T}) = m$, then Assumptions \ref{as:strongly-convex-problem}-\ref{as:dual-locally-stronlgy-concave} hold for this problem. 
\end{enumerate}
 
\end{Thm}
\begin{IEEEproof}
See Appendix \ref{app:strongly-concave-dual-fucntion} for the detailed proof. 
\end{IEEEproof}

\begin{Cor}\label{cor:locally-quadratic-independent-linear-inequality}
Consider the following strongly convex program with linear inequality constraints:
\begin{align}
\min \quad &f(\mathbf{x}) \\
\text{s.t.} \quad  &  \mathbf{A} \mathbf{x} \leq \mathbf{b} 
\end{align}
where $f(\mathbf{x})$ is second-order continuously differentiable and strongly convex function; and $\mathbf{A}$ is an $m\times n$ matrix.
\begin{enumerate}
\item Let $\mathbf{x}^\ast$ be the optimal solution. Assume $\mathbf{A}\mathbf{x}^{\ast} \leq \mathbf{b}$ has $l$ rows that hold with equality, and let $\mathbf{A}^\prime$ be the $l \times  n$ submatrix of $\mathbf{A}$ corresponding to these ``active" rows. If $\text{rank}(\mathbf{A}^\prime) = l$, then Assumptions \ref{as:strongly-convex-problem}-\ref{as:dual-locally-quadratic} hold for this problem. 
\item If $\text{rank}(\mathbf{A}) = m$, then Assumptions \ref{as:strongly-convex-problem}-\ref{as:dual-locally-stronlgy-concave} hold for this problem with $D_c = \infty$. 
\end{enumerate}
\end{Cor}

\subsection{Network Utility Maximizations with Independent Link Capacity Constraints} \label{sec:num}
Consider a network with $l$ links and $n$ flow streams. Let $\{b_{1}, b_{2},\ldots, b_{l}\}$ be the capacities of each link and $\{x_{1}, x_{2},\ldots, x_{n}\}$ be the rates of each flow stream. Let $\mathcal{N}(k) \subseteq {1,2,\ldots,n}, 1\leq k\leq l$ be the set of flow streams that use link $k$. This problem is to maximize the utility function $\sum_{i=1}^{n} w_{i} \log(x_{i})$ with $w_{i}>0, \forall 1\leq i\leq n$, which represents a measure of network fairness \cite{Kelly97EuropeanTT}, subject to the capacity constraint of each link. This problem is known as the network utility maximization (NUM) problem and can be formulated as follows\footnote{In this paper, the NUM problem is always formulated as a minimization problem. Without loss of optimality, we define $\log(0) = -\infty$ and hence $\log(\cdot)$ is defined over $\mathbb{R}_{+}$. Or alternatively, we can replace the non-negative rate constraints with $x_{i} \geq x_{i}^{\min}, \forall i\in \{1,2,\ldots,n\}$ where $x_{i}^{\min}, \forall i\in \{1,2,\ldots,n\}$ are sufficiently small positive numbers.}:
\begin{align*}
\min \quad & ~\sum_{i=1}^{n} -w_{i}\log(x_{i}) \\
\text{s.t.} \quad  & \sum_{i\in \mathcal{N}(k)} x_{i} \leq b_{k}, \forall k\in\{1,2,\ldots,l\}\\
& x_{i} \geq 0, \forall i\in\{1,2,\ldots,n\}
\end{align*}
Typically, many link capacity constraints in the above formulation are redundant, e.g., if $\mathcal{N} (k_{1}) = \mathcal{N}(k_{2})$ and $b_{k_{1}} \leq b_{k_{2}}$, then the capacity constraint of the $k_{2}$-th link is redundant.  Assume that redundant link capacity constraints are eliminated and the remaining links are reindexed. The above formulation can be rewritten as follows:
\begin{align}
\min \quad & \sum_{i=1}^{n} -w_{i}\log(x_{i}) \label{eq:NUM-reduced-object}\\
\text{s.t.} \quad  & \mathbf{A}\mathbf{x} \leq \mathbf{b}  \label{eq:NUM-reduced-capacity-constraint}\\
& \mathbf{x} \geq \mathbf{0} \label{eq:NUM-reduced-rate-constraint}
\end{align}
where $w_{i}>0, \forall 1\leq i\leq n$; $\mathbf{A} = [\mathbf{a}_{1}, \cdots, \mathbf{a}_{n}]$ is a $0$-$1$ matrix of size $m\times n$ such that $a_{ij}=1$ if and only if flow $x_j$ uses link $i$; and $\mathbf{b}>\mathbf{0}$.

Note that problem  \eqref{eq:NUM-reduced-object}-\eqref{eq:NUM-reduced-rate-constraint} satisfies Assumptions \ref{as:strongly-convex-problem} and \ref{as:strong-duality}. By the results from Section \ref{sec:general-strongly-convex},  $\overline{\mathbf{x}}(t)$ has $O(\frac{1}{\epsilon})$ convergence time for this problem. The next theorem provides sufficient conditions such that $\widetilde{\mathbf{x}}(t)$ can have better convergence time $O(\log(\frac{1}{\epsilon}))$ .

\begin{Thm}\label{thm:NUM-locally-quadratic} The network utility maximization problem \eqref{eq:NUM-reduced-object}-\eqref{eq:NUM-reduced-rate-constraint} has the following properties:
\begin{enumerate}
\item Let $b^{\max} = \max_{1\leq i\leq n} b_{i}$ and $\mathbf{x}^{\max}>\mathbf{0}$ such that $x_{i}^{\max} > b^{\max}, \forall i\in \{1,\ldots,n\}$. The network utility maximization problem \eqref{eq:NUM-reduced-object}-\eqref{eq:NUM-reduced-rate-constraint} is equivalent to the following problem
\begin{align}
\min \quad & \sum_{i=1}^{n} -w_{i}\log(x_{i}) \label{eq:NUM-compact-object} \\
\text{s.t.} \quad  & \mathbf{A}\mathbf{x} \leq \mathbf{b}  \label{eq:NUM-compact-capacity-constraint}\\
& \mathbf{0} \leq \mathbf{x} \leq \mathbf{x}^{\max} \label{eq:NUM-compact-rate-constraint}
\end{align}
\item Let $\mathbf{x}^\ast$ be the optimal solution. Assume $\mathbf{A}\mathbf{x}^{\ast} \leq \mathbf{b}$ has $m^\prime$ rows that hold with equality, and let $\mathbf{A}^\prime$ be the $m^\prime \times  n$ submatrix of $\mathbf{A}$ corresponding to these ``active" rows. If $\text{rank}(\mathbf{A}^\prime) = m^\prime$,  then Assumptions \ref{as:strongly-convex-problem}-\ref{as:dual-locally-quadratic} hold for  problem \eqref{eq:NUM-compact-object}-\eqref{eq:NUM-compact-rate-constraint}. 
\item If $\text{rank}(\mathbf{A}) = m$, then Assumptions \ref{as:strongly-convex-problem}-\ref{as:dual-locally-stronlgy-concave} hold for problem \eqref{eq:NUM-compact-object}-\eqref{eq:NUM-compact-rate-constraint}. 
\end{enumerate}

\end{Thm}
\begin{IEEEproof}
See Appendix \ref{app:NUM-locally-quadratic} for the detailed proof.
\end{IEEEproof}

\begin{Rem}
Theorem \ref{thm:NUM-locally-quadratic} and Corollary \ref{cor:locally-quadratic-independent-linear-inequality} complement each other. If $\text{rank}(\mathbf{A}) = m$, we can apply Theorem \ref{thm:NUM-locally-quadratic} to problem \eqref{eq:NUM-reduced-object}-\eqref{eq:NUM-reduced-rate-constraint}. However, to apply Corollary \ref{cor:locally-quadratic-independent-linear-inequality}, we require $\text{rank}(\mathbf{B}) = m+n$, where $\mathbf{B} = \left[\begin{array}{c}\mathbf{A} \\ \mathbf{I}_{n}\end{array}\right]$. This is always false since the size of $\mathbf{A}^{\prime}$ is $(m+n)\times n$.  Thus, Corollary \ref{cor:locally-quadratic-independent-linear-inequality} can not be applied to problem \eqref{eq:NUM-reduced-object}-\eqref{eq:NUM-reduced-rate-constraint} even if $\text{rank}(\mathbf{A}) = m$. On the other hand, Corollary \ref{cor:locally-quadratic-independent-linear-inequality} considers general utilities while Theorem \ref{thm:NUM-locally-quadratic} is restricted to the utility $\sum_{i=1}^{n}  -w_{i} \log(x_{i})$.
\end{Rem}

Now we give an example of network utility maximization such that Assumption \ref{as:dual-locally-quadratic} is not satisfied. Consider the problem  \eqref{eq:NUM-reduced-object}-\eqref{eq:NUM-reduced-rate-constraint} with $\mathbf{w} = [1,1,1,1]^{T}$, $$\mathbf{A} = [\mathbf{a}_{1}, \mathbf{a}_{2}, \mathbf{a}_{3}, \mathbf{a}_{4}] = \left[\begin{array}{cccc} 1 & 1& 0 &0 \\ 0 & 0 & 1 & 1\\ 1 & 0 &1 &0\\ 0 &1 &0 &1\end{array}\right]$$ and $\mathbf{b} = [3,7,2,8]^{T}$. Note that $\text{rank}(\mathbf{A}) = 3 < m$; and $ \boldsymbol{\mu}^{T}\mathbf{A} = [0, 0, 0, 0]$ and $\boldsymbol{\mu}^{T}\mathbf{b} = 0$ if $\boldsymbol{\mu} = [1, 1, -1, -1]^{T}$.

It can be checked that the optimal solution to this NUM problem is $[ x_{1}^{\ast}, x_{2}^{\ast}, x_{3}^{\ast}, x_{4}^{\ast}]^{T} = [ 0.8553, 2.1447, 1.1447, 5.8553 ]^{T}$. Note that all capacity constraints are tight and $[ \lambda_{1}^{\ast}, \lambda_{2}^{\ast}, \lambda_{3}^{\ast}, \lambda_{4}^{\ast}]^{T} =[0.3858, 0.0903, 0.7833, 0.0805]^T$ is the optimal dual variable that attains the strong duality.

Next, we show that $q(\boldsymbol{\lambda})$ is not locally quadratic at $\boldsymbol{\lambda} = \boldsymbol{\lambda}^{\ast}$ by contradiction.  Assume that there exists $D_q>0$ and $L_{q} > 0$ such that $q(\boldsymbol{\lambda}) \leq q(\boldsymbol{\lambda}^{\ast}) - L_{q} \Vert \boldsymbol{\lambda} -\boldsymbol{\lambda}^{\ast}\Vert^{2}$ for any $\boldsymbol{\lambda}\in \mathbb{R}^{m}_{+}$ and $\Vert \boldsymbol{\lambda} - \boldsymbol{\lambda}^{\ast}\Vert\leq D_q$.  Put $\boldsymbol{\lambda} = \boldsymbol{\lambda}^{\ast} + t \boldsymbol{\mu}$ with $|t|$ sufficiently small such that  $\boldsymbol{\lambda}^{\ast} + t \boldsymbol{\mu} \in \mathbb{R}^{m}_{+}$ and $\Vert \boldsymbol{\lambda}^{\ast} + t \boldsymbol{\mu} - \boldsymbol{\lambda}^{\ast} \Vert <D_q $.  Note that by \eqref{eq:NUM-dual-function-gradient} and \eqref{eq:NUM-dual-function-hessian}, we have
\begin{align}
\boldsymbol{\mu}^{T} \nabla_{\boldsymbol{\lambda}} q(\boldsymbol{\lambda}^{\ast}) = \sum_{i=1}^{n} \frac{\boldsymbol{\mu}^{T}\mathbf{a}_{i}}{(\boldsymbol{\lambda}^{\ast})^{T}\mathbf{a}_{i}} + \boldsymbol{\mu}^{T}\mathbf{b} = 0, \label{eq:NUM-example-zero-first-order-term}\\
\boldsymbol{\mu}^{T} \nabla_{\boldsymbol{\lambda}}^{2} q(\boldsymbol{\lambda}^{\ast}) \boldsymbol{\mu} = 0.\label{eq:NUM-example-zero-second-order-term}
\end{align}

Thus, we have 
\begin{align*}
&q(\boldsymbol{\lambda}^{\ast} + t \boldsymbol{\mu}) \\
\overset{(a)}{=} &q(\boldsymbol{\lambda}^{\ast}) + t\boldsymbol{\mu}^{T} \nabla_{\boldsymbol{\lambda}} \tilde{q}(\boldsymbol{\lambda}^{\ast}) + t^{2}\boldsymbol{\mu}^{T} \nabla_{\boldsymbol{\lambda}}^{2} \tilde{q}(\boldsymbol{\lambda}^{\ast})\boldsymbol{\mu} + o(t^{2}\Vert \boldsymbol{\mu} \Vert^{2}) \\
\overset{(b)}{=} &q(\boldsymbol{\lambda}^{\ast}) +  o(t^{2}\Vert \boldsymbol{\mu} \Vert^{2})
\end{align*}
where $(a)$ follows from the second-order Taylor's expansion and $(b)$ follows from equations \eqref{eq:NUM-example-zero-first-order-term} and \eqref{eq:NUM-example-zero-second-order-term}. By definition of $o(t^{2}\Vert \boldsymbol{\mu} \Vert^{2})$, there exists $\delta>0$ such that $\frac{|o(t^{2}\Vert \boldsymbol{\mu} \Vert^{2})|}{\Vert t\boldsymbol{\mu}\Vert^{2}} <  L_{q}, \forall t\in (-\delta, \delta)$, i.e., $o(t^{2}\Vert \boldsymbol{\mu} \Vert^{2}) > -L_{q} \Vert t\boldsymbol{\mu}\Vert^{2}, \forall t\in (-\delta, \delta)$. This implies $q(\boldsymbol{\lambda}^{\ast} + t \boldsymbol{\mu}) = q(\boldsymbol{\lambda}^{\ast}) +  o(t^{2}\Vert \boldsymbol{\mu} \Vert^{2}) > q(\boldsymbol{\lambda}^{\ast}) - L_{q} \Vert t\boldsymbol{\mu}\Vert^{2}$. A contradiction! Thus, $q(\boldsymbol{\lambda})$ is not locally quadratic at $\boldsymbol{\lambda} = \boldsymbol{\lambda}^{\ast}$.

In view of the above example, the sufficient condition in part (2) of Theorem \ref{thm:NUM-locally-quadratic} for Assumption \ref{as:dual-locally-quadratic} is sharp.

\section{Numerical Results}
\subsection{Network Utility Maximization Problems}
Consider the simple NUM problem described in Figure \ref{fig:network-flow}. Let $x_{1}, x_{2}$ and $x_{3}$ be the data rates of stream $1, 2$ and $3$ and let the network utility be minimizing $- \log (x_{1}) - 2\log(x_{2}) - 3\log(x_{3})$. It can be checked that capacity constraints other than $x_{1} + x_{2} + x_{3} \leq 10, x_{1}+x_{2} \leq 8, $ and $x_{2} + x_{3}\leq 8$ are redundant. By Theorem \ref{thm:NUM-locally-quadratic}, the NUM problem can be formulated as follows:
\begin{align*}
\min \quad & -\log(x_1)- 2\log(x_2) - 3\log(x_3) \\
\text{s.t.} \quad  &  \mathbf{A}\mathbf{x} \leq \mathbf{b}\\
& \mathbf{0}\leq \mathbf{x} \leq \mathbf{x}^{\max}  
\end{align*} 
where $\mathbf{A} = \left[\begin{array}{ccc} 1 & 1 &1 \\ 1&1&0 \\ 0&1&1\end{array}\right]$, $\mathbf{b} = \left[\begin{array}{c}10\\8\\8\end{array}\right]$ and $\mathbf{x}^{\max} = \left[\begin{array}{c}11\\11\\11\end{array}\right]$. The optimal solution to this NUM problem is $x_{1}^{\ast} =2, x_{2}^{\ast} = 3.2, x_{3}^{\ast} = 4.8$ and the optimal value is $-7.7253$. Note that the second capacity constraint $x_{1}+x_{2} \leq 8$ is loose and the other two capacity constraints are tight.

Since the objective function is decomposable, the dual subgradient method  can yield a distributed solution. This is why the dual subgradient method is widely used to solve NUM problems \cite{Low99TON}.  It can be checked that the objective function is strongly convex with modulus $\alpha = \frac{2}{121}$ on $\mathcal{X} = \{\mathbf{0}\leq \mathbf{x} \leq \mathbf{x}^{\max}\}$ and  $\mathbf{g}$ is Lipschitz continuous with modulus $\beta \leq \sqrt{6}$ on $\mathcal{X}$. Figure \ref{fig:num-alg1-convergence} verifies the convergence of $\overline{\mathbf{x}}(t)$ with $c = \frac{\alpha}{\beta^2} = 1/363$ and $\lambda_1(0)  = \lambda_2(0) = \lambda_{3}(0) = 0$. Since $\lambda_1(0)= \lambda_2(0) =\lambda_{3}(0) =  0$, by Theorem  \ref{thm:utility-convergence-time}, we know $f(\overline{\mathbf{x}}(t)) \leq f(\mathbf{x}^\ast), \forall t>0$, which is also verified in Figure \ref{fig:num-alg1-convergence}. To verify the convergence time of constraint violations,  Figure \ref{fig:num-alg1-convergence-time} plots $g_1(\overline{\mathbf{x}}(t))$, $g_2(\overline{\mathbf{x}}(t))$, $g_{3}(\overline{\mathbf{x}}(t))$ and $1/t$ with both x-axis and y-axis in $\log_{10}$ scales. As observed in Figure \ref{fig:num-alg1-convergence-time}, the curves of $g_1(\overline{\mathbf{x}}(t))$ and $g_3(\overline{\mathbf{x}}(t))$ are parallel to the curve of $1/t$ for large $t$. Note that $g_2(\overline{\mathbf{x}}(t))\leq 0$ is satisfied early because this constraint is loose. Figure \ref{fig:num-alg1-convergence-time} verifies the convergence time of $\overline{\mathbf{x}}(t)$ in Theorem \ref{thm:overall-convergence-time} by showing that error decays like $O(\frac{1}{t})$ and suggests that the convergence time is actually $\Theta(\frac{1}{\epsilon})$ for this NUM problem.

\begin{figure}
\centering
   \includegraphics[width=0.6\textwidth,height=0.6\textheight,keepaspectratio=true]{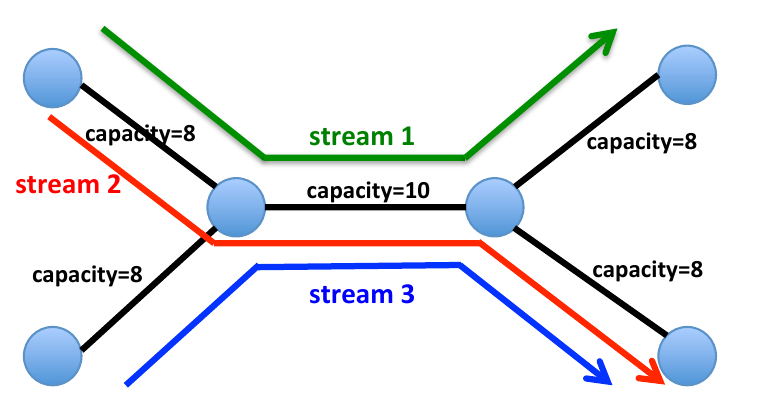} 
   \caption{A simple NUM problem with $3$ flow streams}
   \label{fig:network-flow}
\end{figure}

\begin{figure}
\centering
   \includegraphics[width=0.6\textwidth,height=0.6\textheight,keepaspectratio=true]{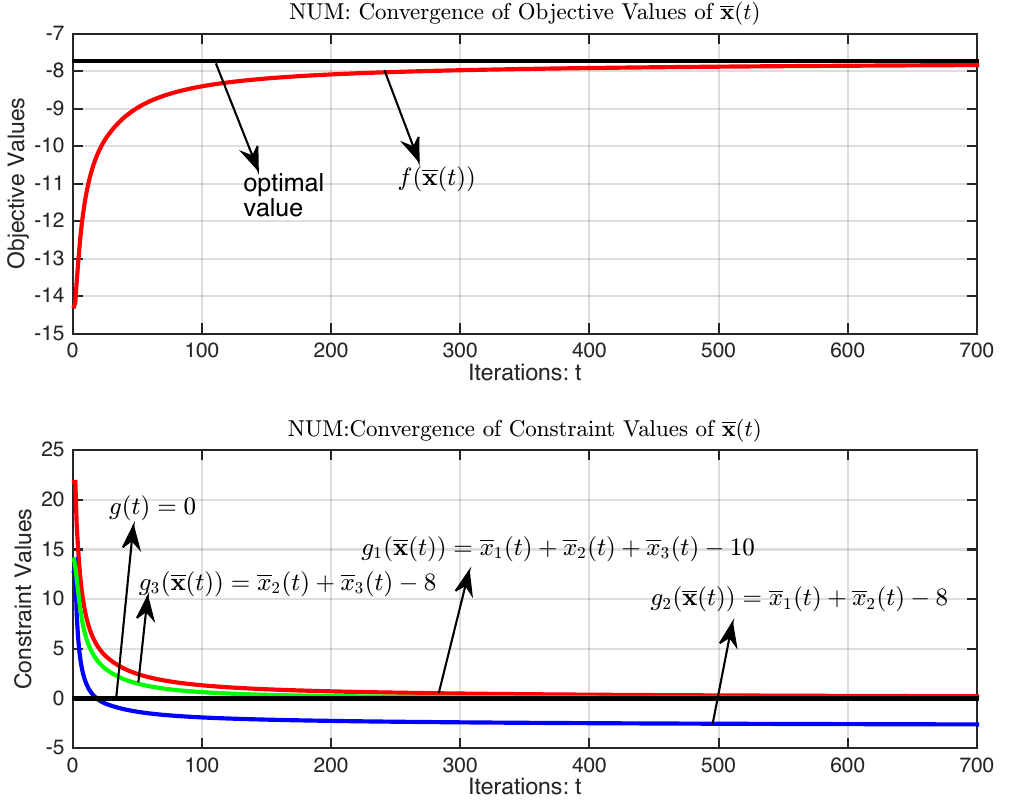} 
   \caption{The convergence of $\overline{\mathbf{x}}(t)$ for a NUM problem.}
   \label{fig:num-alg1-convergence}
\end{figure}
\begin{figure}
\centering
   \includegraphics[width=0.6\textwidth,height=0.6\textheight,keepaspectratio=true]{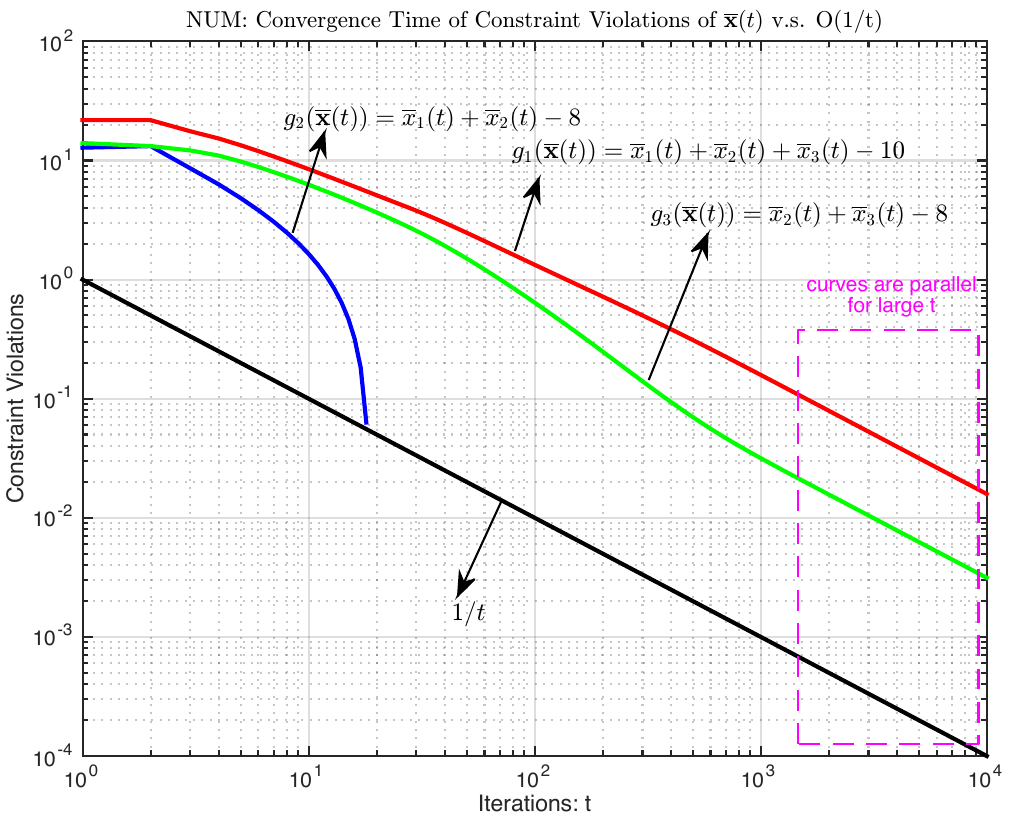} 
   \caption{The convergence time of $\overline{\mathbf{x}}(t)$ for a NUM problem.}
   \label{fig:num-alg1-convergence-time}
\end{figure}

Note that $\text{rank}(\mathbf{A}) = 3$. By Theorem \ref{thm:NUM-locally-quadratic}, this NUM problem satisfies Assumptions \ref{as:strongly-convex-problem}-\ref{as:dual-locally-stronlgy-concave}.  Apply Algorithm \ref{alg:dual-subgradient} with $c = \frac{\alpha}{\beta^2} = 1/363$ and $\lambda_1(0)  = \lambda_2(0) = \lambda_{3}(0) = 0$ to this NUM problem. Figure \ref{fig:num-alg3-convergence} verifies the convergence of the objective and constraint functions for $\widetilde{\mathbf{x}}(t)$. Figure \ref{fig:num-alg3-convergence-time} verifies the results in Theorem \ref{thm:NUM-locally-quadratic} that the convergence time of $\widetilde{\mathbf{x}}(t)$ is $O(\log(\frac{1}{\epsilon}))$ by showing that error decays like $O(\frac{1}{t} 0.998^t)$. If we compare Figure \ref{fig:num-alg3-convergence-time} and Figure \ref{fig:num-alg1-convergence-time}, we can observe that $\widetilde{\mathbf{x}}(t)$ converges much faster than $\overline{\mathbf{x}}(t)$.

\begin{figure}
\centering
   \includegraphics[width=0.6\textwidth,height=0.6\textheight,keepaspectratio=true]{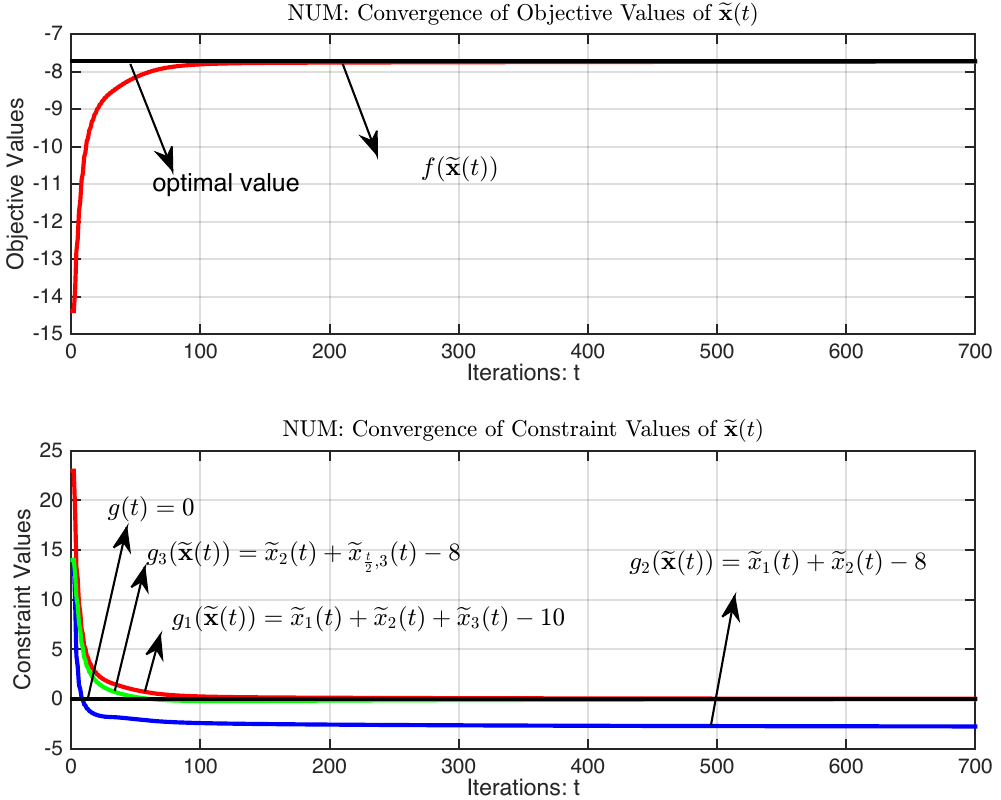} 
   \caption{The convergence of $\widetilde{\mathbf{x}}(t)$ for a NUM problem.}
   \label{fig:num-alg3-convergence}
\end{figure}
\begin{figure}
\centering
   \includegraphics[width=0.6\textwidth,height=0.6\textheight,keepaspectratio=true]{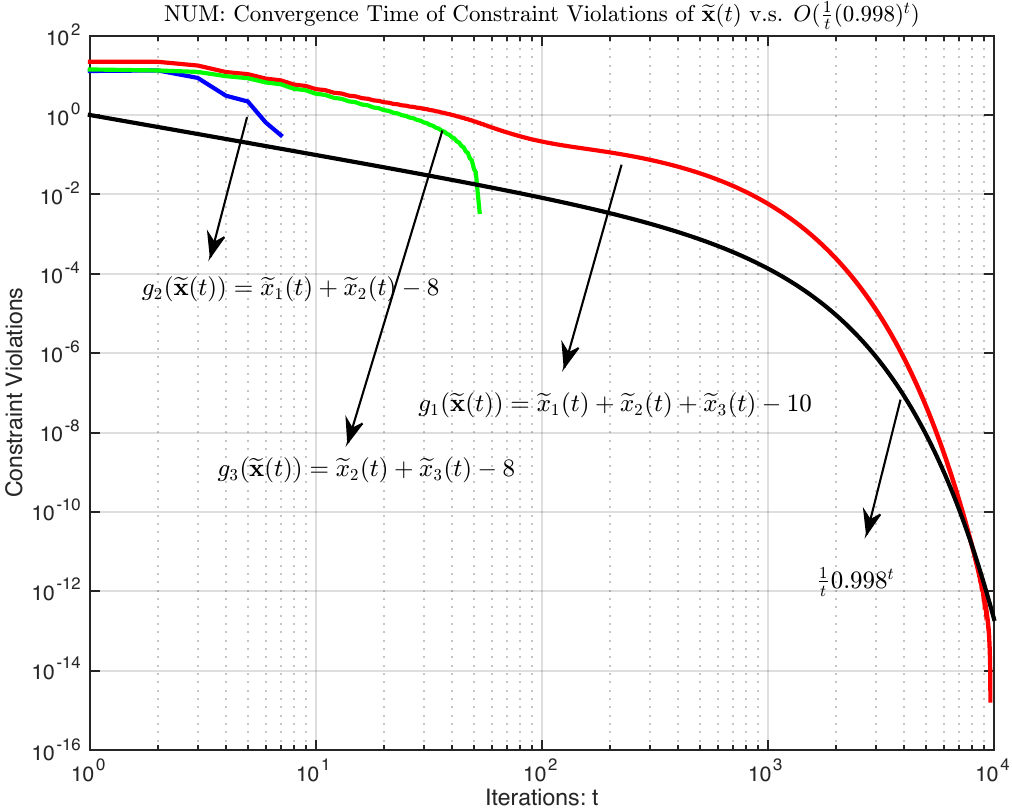} 
   \caption{The convergence time of $\widetilde{\mathbf{x}}(t)$ for a NUM problem.}
   \label{fig:num-alg3-convergence-time}
\end{figure}

\subsection{Quadratic Programs (QPs) with Linear Constraints}
Consider the following quadratic program (QP)
\begin{align*}
\min \quad & \mathbf{x}^T \mathbf{P} \mathbf{x} + \mathbf{c}^T\mathbf{x} \\
\text{s.t.} \quad  & \mathbf{A}\mathbf{x} \leq \mathbf{b}\\
\end{align*}
where $\mathbf{P} = \left[\begin{array}{cc} 1 &2 \\ 2 &5\end{array}\right]$, $\mathbf{c} = [1,1]^T$, $\mathbf{A} = \left[\begin{array}{cc} 1 &1 \\ 0 &1\end{array}\right]$ and $\mathbf{b} = [-2,-1]^T$. The optimal solution to this quadratic program is $\mathbf{x}^{\ast} = \left[\begin{array}{c}-1\\ -1\end{array}\right]$ and the optimal value is $8$.

If $P$ is a diagonal matrix, the dual subgradient method can yield a distributed solution.  It can be checked that the objective function is strongly convex with modulus $\alpha = 0.34$ and each row of the linear inequality constraint is Lipschitz continuous with modulus $\zeta = \sqrt{2}$.  Figure \ref{fig:qp-alg1-convergence} verifies the convergence of  $\overline{\mathbf{x}}(t)$ for the objective and constraint functions yielded by Algorithm \ref{alg:dual-subgradient} with $c = \frac{\alpha}{2\zeta^2} = 0.34/4$, $\lambda_1(0) = 0$ and $\lambda_2(0) = 0$.   Figure \ref{fig:qp-alg1-convergence-time} verifies the convergence time of $\overline{\mathbf{x}}(t)$ proven in Theorem \ref{thm:overall-convergence-time} by showing that error decays like $O(\frac{1}{t})$ and suggests that the convergence time is actually $\Theta(\frac{1}{\epsilon})$ for this quadratic program.
\begin{figure}
\centering
   \includegraphics[width=0.6\textwidth,height=0.6\textheight,keepaspectratio=true]{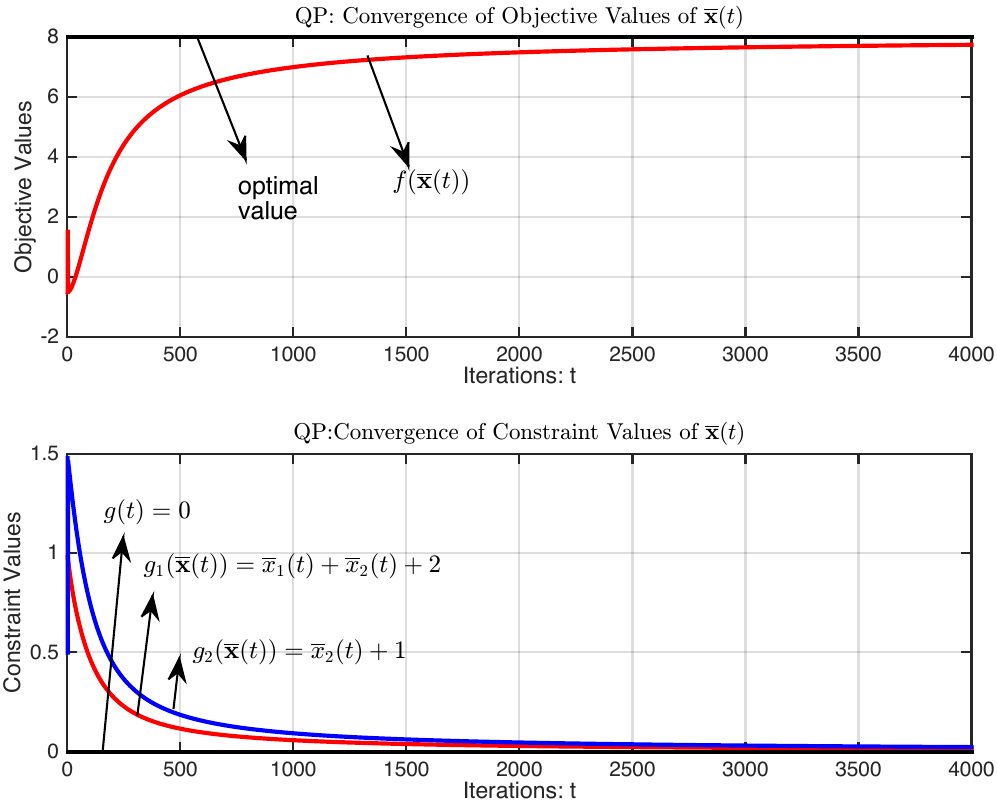} 
   \caption{The convergence of $\overline{\mathbf{x}}(t)$ for a quadratic program.}
   \label{fig:qp-alg1-convergence}
\end{figure}

\begin{figure}
\centering
   \includegraphics[width=0.6\textwidth,height=0.6\textheight,keepaspectratio=true]{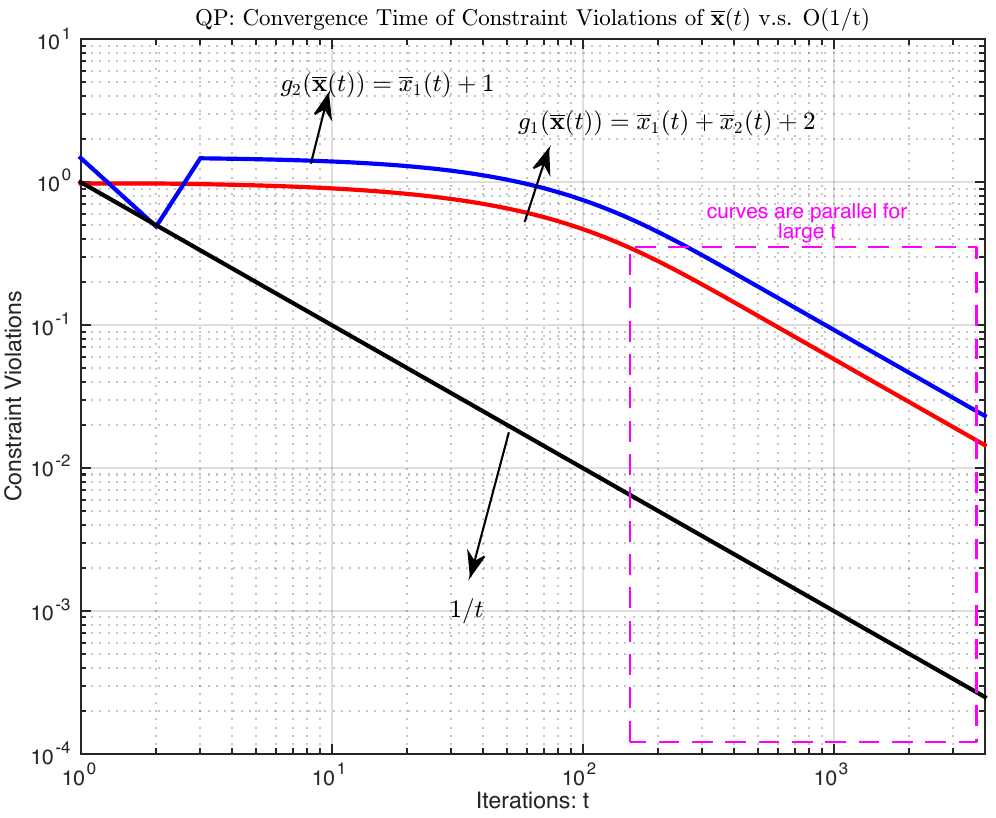} 
   \caption{The convergence time of $\overline{\mathbf{x}}(t)$ for a quadratic program.}
   \label{fig:qp-alg1-convergence-time}
\end{figure}

\begin{figure}
\centering
   \includegraphics[width=0.6\textwidth,height=0.6\textheight,keepaspectratio=true]{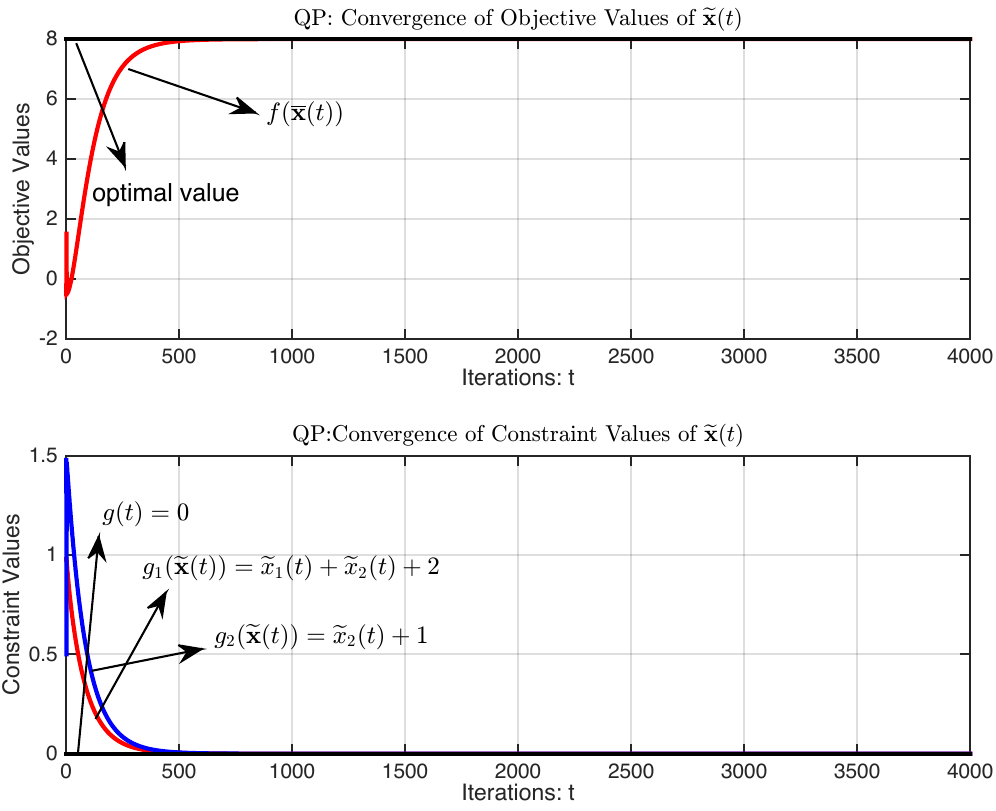} 
   \caption{The convergence of $\widetilde{\mathbf{x}}(t)$ for to a quadratic program.}
   \label{fig:qp-alg3-convergence}
\end{figure}
\begin{figure}
\centering
   \includegraphics[width=0.6\textwidth,height=0.6\textheight,keepaspectratio=true]{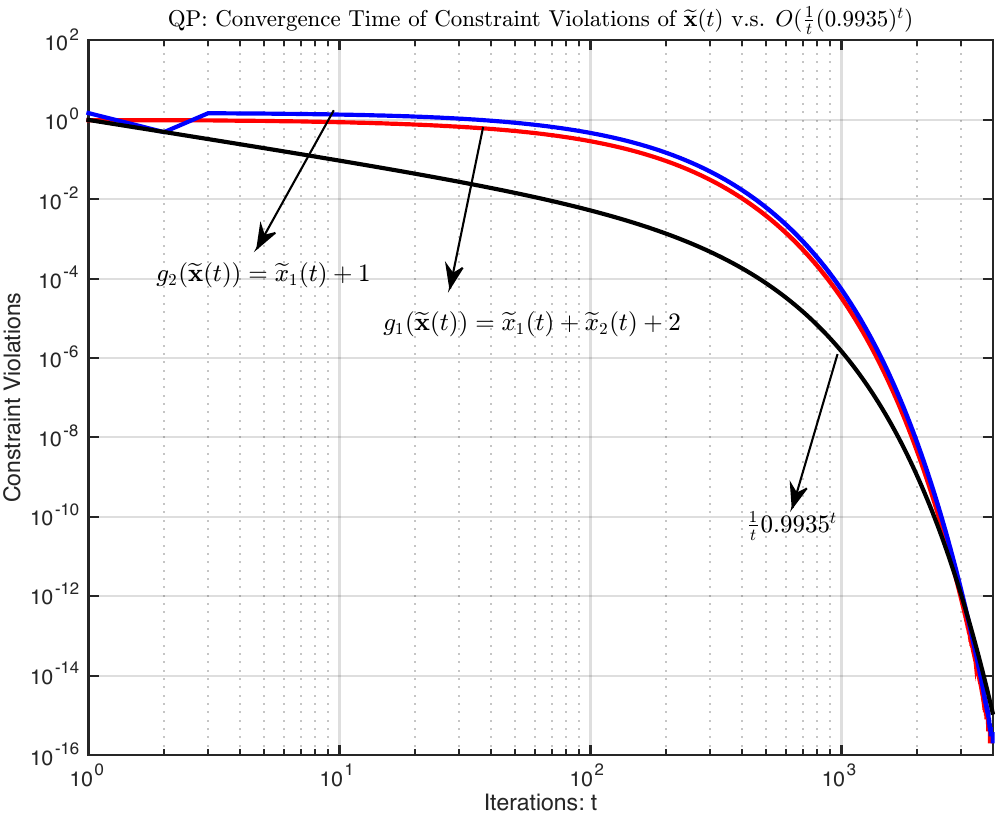} 
   \caption{The convergence time of $\widetilde{\mathbf{x}}(t)$ for a quadratic program.}
   \label{fig:qp-alg3-convergence-time}
\end{figure}

Note that $\text{rank}(A) = 2$. By Corollary \ref{cor:locally-quadratic-independent-linear-inequality} this quadratic program satisfies Assumptions \ref{as:strongly-convex-problem}-\ref{as:dual-locally-stronlgy-concave}.  

Apply Algorithm \ref{alg:dual-subgradient} with $c = \frac{\alpha}{2\zeta^2} = 0.34/4$ and $\lambda_1(0)  = \lambda_2(0) = \lambda_{3}(0) = 0$ to this quadratic program. Figure \ref{fig:qp-alg3-convergence} verifies the convergence of the objective and constraint functions. Figure \ref{fig:qp-alg3-convergence-time} verifies the results in Corollary \ref{cor:locally-quadratic-independent-linear-inequality} that the convergence time of $\widetilde{\mathbf{x}}(t)$  is $O(\log(\frac{1}{\epsilon}))$ by showing that error decays like $O(\frac{1}{t}0.9935^t)$. If we compare Figure \ref{fig:qp-alg3-convergence-time} and Figure \ref{fig:qp-alg1-convergence-time}, we can observe that Algorithm $\widetilde{\mathbf{x}}(t)$  converges much faster than $\overline{\mathbf{x}}(t)$. 

\subsection{Large Scale Quadratic Programs }
Consider quadratic program $\min_{\mathbf{x}\in \mathbb{R}^N} \{ \mathbf{x}^T \mathbf{Q} \mathbf{x} + \mathbf{d}^T\mathbf{x}: \mathbf{A}\mathbf{x} \leq \mathbf{b}\}$ where $\mathbf{Q}, \mathbf{A}\in \mathbb{R}^{N\times N}$ and $\mathbf{d}, \mathbf{b}\in \mathbb{R}^N$. $\mathbf{Q} = \mathbf{U}\Sigma \mathbf{U}^H \in \mathbb{R}^{N\times N}$ where $\mathbf{U}$ is the orthonormal basis for a random $N\times N$ zero mean and unit variance normal matrix and $\Sigma$ is the diagonal matrix with entries from uniform  $[1,3]$. $\mathbf{A}$ is a random $N\times N$ zero mean and unit variance normal matrix. $\mathbf{d}$ and $\mathbf{b}$ are random vectors with entries from uniform  $[0,1]$. In a PC with a 4 core 2.7GHz Intel i7 Cpu and 16GB Memory, we run both Algorithm \ref{alg:dual-subgradient} and quadprog from Matlab, which by default is using the interior point method, over randomly generated large scale quadratic programs with $N=400, 600, 800,1000$ and $1200$.  For different problem size $N$,  the running time is the average over $100$ random quadratic programs and is plotted in Figure \ref{fig:large-scale}.  To solve this large scale quadratic programs, the dual subgradient method has updates $\mathbf{x}(t) = -\frac{1}{2} \mathbf{Q}^{-1} [ \boldsymbol{\lambda}^T(t) \mathbf{A} + \mathbf{d}]$ and $\boldsymbol{\lambda}(t+1) = \max\{\boldsymbol{\lambda}(t) + c [\mathbf{A}\mathbf{x}(t) - b], \mathbf{0}\}$ at each iteration $t$. Note that we only need to compute the inverse of large matrix $\mathbf{\mathbf{Q}}$ once and then use it during all iterations. In our numerical simulations, Algorithm \ref{alg:dual-subgradient} is terminated when the error (both objective violations and constraint violations) of $\widetilde{\mathbf{x}}(t)$ is less than 1e-5.

\begin{figure}
\centering
   \includegraphics[width=0.6\textwidth,height=0.6\textheight,keepaspectratio=true]{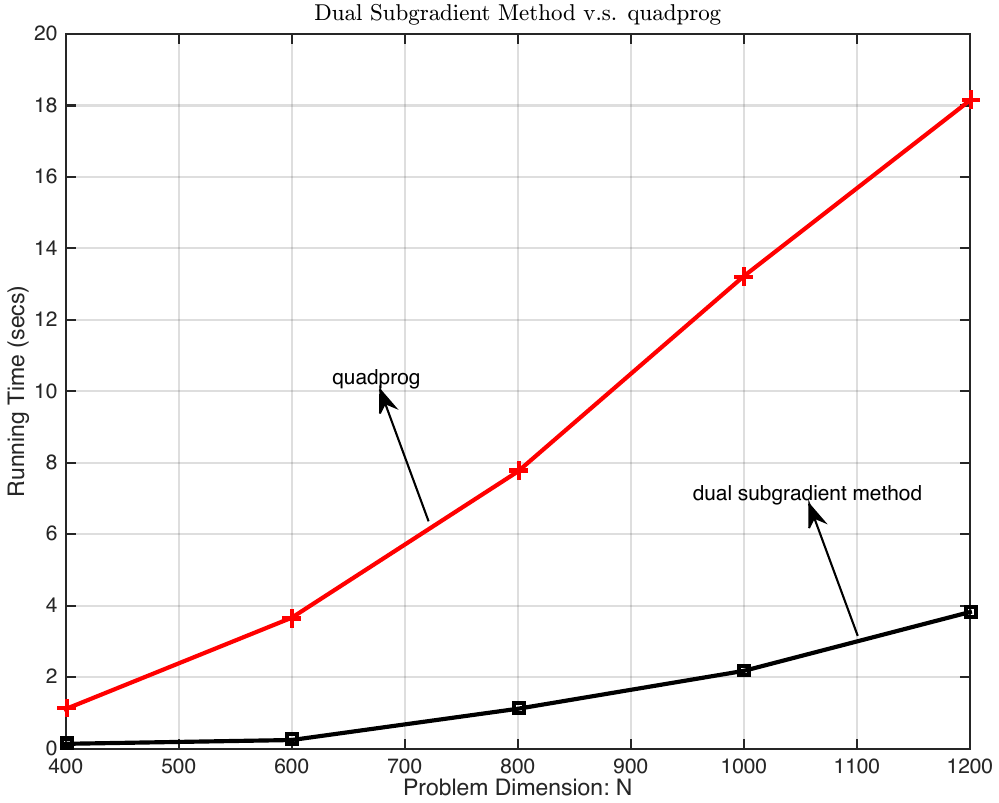} 
   \caption{The average running time for large scale quadratic programs.}
   \label{fig:large-scale}
\end{figure}

\section{Conclusions}
This paper studies the convergence time of the dual subgradient method strongly convex programs. This paper shows that the convergence time of the dual subgradient method with simple running averages for general strongly convex programs is $O(\frac{1}{\epsilon})$. This paper also considers a variation of the primal averages, called the sliding running averages, and shows that if the dual function is  locally quadratic then the convergence time is $O(\log(\frac{1}{\epsilon}))$.

\newpage
\appendices
\section{Proof of Lemma \ref{lm:dual_problem_convergence}} \label{app:dual_problem_convergence}

Note that $\lambda_k(t+1) = \max\{\lambda_k(t)+ c g_k(\mathbf{x}(t)),0\}, \forall k\in\{1,2,\ldots,m\}$ can be interpreted as the $\boldsymbol{\lambda}(t+1) = \mathcal{P}_{\mathbb{R}^m_+}\big[\boldsymbol{\lambda}(t)+c \mathbf{g}(\mathbf{x}(t))\big]$ where $\mathcal{P}_{\mathbb{R}^m_+}$ is the projection onto $\mathbb{R}^m_+$.  As observed before, the dynamic of $\boldsymbol{\lambda}(t)$ can be interpreted as the projected gradient method with  step size $c$ to solve $\max\limits_{\boldsymbol{\lambda}\in \mathbb{R}^m_+} \{q(\boldsymbol{\lambda})\}$.  Thus, the proof given below is essentially the same as the convergence time proof of the projected gradient method for set constrained smooth optimization in \cite{book_ConvexOpt_Nesterov}.

\begin{Lem}[Descent Lemma, Proposition A.24 in \cite{book_NonlinearProgramming_Bertsekas}] \label{lm:descent-lemma}
Let $h$ be a continuously differentiable function. If $h$ is smooth on $\mathcal{X}$ with modulus $L$, then for any $\mathbf{x}, \mathbf{y} \in \mathcal{X}$ we have
\begin{align*}
h(\mathbf{y})  &\geq  h(\mathbf{x}) + \nabla h(\mathbf{x})^T (\mathbf{y} - \mathbf{x}) - \frac{L}{2} || \mathbf{y} - \mathbf{x}||^2, \\
h(\mathbf{y})  &\leq  h(\mathbf{x}) + \nabla h(\mathbf{x})^T (\mathbf{y} - \mathbf{x}) + \frac{L}{2} || \mathbf{y} - \mathbf{x}||^2.
\end{align*} 
\end{Lem}

\begin{Fact}  \label{fact:app_f1} 
$\boldsymbol{\lambda}(t+1) = \argmax_{\boldsymbol{\lambda}\in \mathbb{R}^m_+} \big[ q(\boldsymbol{\lambda}(t)) + [\mathbf{g}(\mathbf{x}(t))]^{T} [\boldsymbol{\lambda} - \boldsymbol{\lambda}(t)] -\frac{1}{2c} \Vert \boldsymbol{\lambda} - \boldsymbol{\lambda}(t)\Vert^{2}\big], \forall t\geq 0$.
\end{Fact}
\begin{proof} Fix $t\geq 0$,
\begin{align*}
\boldsymbol{\lambda}(t+1) =& \mathcal{P}_{\mathbb{R}^m_+}\left[\boldsymbol{\lambda}(t)+c \mathbf{g}(\mathbf{x}(t))\right]\\
=& \argmin_{\boldsymbol{\lambda}\in \mathbb{R}^m_+}  \Vert \boldsymbol{\lambda} - [\boldsymbol{\lambda}(t)+c \mathbf{g}(\mathbf{x}(t))]\Vert^{2}\\
=& \argmin_{\boldsymbol{\lambda}\in \mathbb{R}^m_+} \left[ c^{2} \Vert \mathbf{g}(\mathbf{x}(t))\Vert^{2} - 2c [\mathbf{g}(\mathbf{x}(t))]^{T} [\boldsymbol{\lambda} - \boldsymbol{\lambda}(t)] + \Vert \boldsymbol{\lambda} - \boldsymbol{\lambda}(t)\Vert^{2}\right]\\
\overset{(a)}{=}& \argmin_{\boldsymbol{\lambda}\in \mathbb{R}^m_+} \left[ -q(\boldsymbol{\lambda}(t)) - [\mathbf{g}(\mathbf{x}(t))]^{T} [\boldsymbol{\lambda} - \boldsymbol{\lambda}(t)] + \frac{1}{2c} \Vert \boldsymbol{\lambda} - \boldsymbol{\lambda}(t)\Vert^{2}\right]\\
=&\argmax_{\boldsymbol{\lambda}\in \mathbb{R}^m_+} \left[ q(\boldsymbol{\lambda}(t)) + [\mathbf{g}(\mathbf{x}(t))]^{T} [\boldsymbol{\lambda} - \boldsymbol{\lambda}(t)] -\frac{1}{2c} \Vert \boldsymbol{\lambda} - \boldsymbol{\lambda}(t)\Vert^{2}\right]
\end{align*}
where (a) follows because the minimizer is unchanged when we remove constant term $c^{2} \Vert \mathbf{g}(\mathbf{x}(t))\Vert^{2}$,  divide by factor $2c$, and add constant term $-q(\boldsymbol{\lambda}(t))$ in the objective function. 
\end{proof}

Recall that $q(\boldsymbol{\lambda})$ is smooth with modulus $\gamma = \frac{\beta^2}{\alpha}$ by Lemma \ref{lm:dual-smooth}.

\begin{Fact}\label{fact:app_f2} 
If $c\leq\frac{1}{\gamma}= \frac{\alpha}{\beta^2}$, then $q(\boldsymbol{\lambda}(t+1)) \geq q(\boldsymbol{\lambda}(t)), \forall t \geq 0$.
\end{Fact}
\begin{proof} Fix $t\geq 0$,
\begin{align*}
q(\boldsymbol{\lambda}(t+1)) \overset{(a)}{\geq}& q(\boldsymbol{\lambda}(t)) + [\mathbf{g}(\mathbf{x}(t))]^{T}[\boldsymbol{\lambda}(t+1) - \boldsymbol{\lambda}(t)] - \frac{\gamma}{2} \Vert \boldsymbol{\lambda}(t+1) - \boldsymbol{\lambda}(t)\Vert^{2}\\
\overset{(b)}{\geq}& q(\boldsymbol{\lambda}(t)) + [\mathbf{g}(\mathbf{x}(t))]^{T}[\boldsymbol{\lambda}(t+1) - \boldsymbol{\lambda}(t)] - \frac{1}{2c} \Vert \boldsymbol{\lambda}(t+1) - \boldsymbol{\lambda}(t)\Vert^{2}\\
\overset{(c)}{\geq}& q(\boldsymbol{\lambda}(t)) + [\mathbf{g}(\mathbf{x}(t))]^{T}[\boldsymbol{\lambda}(t) - \boldsymbol{\lambda}(t)] - \frac{1}{2c} \Vert \boldsymbol{\lambda}(t) - \boldsymbol{\lambda}(t)\Vert^{2}\\
=&q(\boldsymbol{\lambda}(t))
\end{align*}
where (a) follows from Lemma \ref{lm:descent-lemma} and the fact that $\nabla_{\boldsymbol{\lambda}} q(\boldsymbol{\lambda}(t)) =  \mathbf{g}(\mathbf{x}(t))$; (b) follows from $c \leq \frac{1}{\gamma}$; and (c) follows form Fact \ref{fact:app_f1}.
\end{proof}

\begin{Fact} \label{fact:app_f3} $\mathbf{g}^T(\mathbf{x}(t))[\boldsymbol{\lambda}(t+1) -\boldsymbol{\lambda}^\ast ] \geq \frac{1}{c}[\boldsymbol{\lambda}(t+1) - \boldsymbol{\lambda}(t)]^T [\boldsymbol{\lambda}(t+1) - \boldsymbol{\lambda}^\ast], 
\forall t\geq 0$
\end{Fact}
\begin{IEEEproof}
Fix $t\geq0$. By the projection theorem (Proposition B.11(b) in \cite{book_NonlinearProgramming_Bertsekas}), we have $ \big[\boldsymbol{\lambda}(t+1) - \big(\boldsymbol{\lambda}(t) + c\mathbf{g}(\mathbf{x}(t))\big)\big]^T [\boldsymbol{\lambda}(t+1) - \boldsymbol{\lambda}^\ast ]\leq 0$. Thus, $\mathbf{g}^T(\mathbf{x}(t))[\boldsymbol{\lambda}(t+1) -\boldsymbol{\lambda}^\ast ] \geq \frac{1}{c}[\boldsymbol{\lambda}(t+1) - \boldsymbol{\lambda}(t)]^T [\boldsymbol{\lambda}(t+1) - \boldsymbol{\lambda}^\ast]$. \end{IEEEproof}

\begin{Fact}\label{fact:app_f4} If $c\leq \frac{1}{\gamma} = \frac{\alpha}{\beta^2}$, then $q(\boldsymbol{\lambda}^\ast) - q(\boldsymbol{\lambda}(t+1)) \leq \frac{1}{2c} \Vert \boldsymbol{\lambda}(t) - \boldsymbol{\lambda}^\ast\Vert^2 - \frac{1}{2c} \Vert \boldsymbol{\lambda}(t+1) - \boldsymbol{\lambda}^\ast\Vert^2, \forall t\geq 0$.
\end{Fact}
\begin{IEEEproof} Fix $t\geq 0$, 
\begin{align*}
&q(\boldsymbol{\lambda}(t+1)) \\
\overset{(a)}{\geq} &  q(\boldsymbol{\lambda}(t)) + [\mathbf{g}(\mathbf{x}(t))]^T [\boldsymbol{\lambda}(t+1) - \boldsymbol{\lambda}(t)] - \frac{\gamma}{2} \Vert \boldsymbol{\lambda}(t+1) - \boldsymbol{\lambda}(t)\Vert^2 \\
=& q(\boldsymbol{\lambda}(t)) + [\mathbf{g}(\mathbf{x}(t))]^T [\boldsymbol{\lambda}(t+1) - \boldsymbol{\lambda}^\ast + \boldsymbol{\lambda}^\ast - \boldsymbol{\lambda}(t)] - \frac{\gamma}{2} \Vert \boldsymbol{\lambda}(t+1) - \boldsymbol{\lambda}(t)\Vert^2 \\
\overset{(b)}{\geq} &q(\boldsymbol{\lambda}(t)) + [\mathbf{g}(\mathbf{x}(t))]^T [\boldsymbol{\lambda}^\ast - \boldsymbol{\lambda}(t)] - \frac{\gamma}{2} \Vert \boldsymbol{\lambda}(t+1) - \boldsymbol{\lambda}(t)\Vert^2 + \frac{1}{c}[\boldsymbol{\lambda}(t+1) - \boldsymbol{\lambda}(t)]^T [\boldsymbol{\lambda}(t+1) - \boldsymbol{\lambda}^\ast]\\
\overset{(c)}{\geq} &q(\boldsymbol{\lambda}(t)) + [\mathbf{g}(\mathbf{x}(t))]^T [\boldsymbol{\lambda}^\ast - \boldsymbol{\lambda}(t)] - \frac{\gamma}{2} \Vert \boldsymbol{\lambda}(t+1) - \boldsymbol{\lambda}(t)\Vert^2 + \frac{1}{2c} \Vert \boldsymbol{\lambda}(t+1) - \boldsymbol{\lambda}(t)\Vert^2 + \frac{1}{2c} \Vert \boldsymbol{\lambda}(t+1) - \boldsymbol{\lambda}^\ast\Vert^2 \\ &- \frac{1}{2c} \Vert \boldsymbol{\lambda}(t) - \boldsymbol{\lambda}^\ast\Vert^2\\
\overset{(d)}{\geq} & q(\boldsymbol{\lambda}(t)) + [\mathbf{g}(\mathbf{x}(t))]^T [\boldsymbol{\lambda}^\ast - \boldsymbol{\lambda}(t)] + \frac{1}{2c} \Vert \boldsymbol{\lambda}(t+1) - \boldsymbol{\lambda}^\ast\Vert^2 - \frac{1}{2c} \Vert \boldsymbol{\lambda}(t) - \boldsymbol{\lambda}^\ast\Vert^2\\
\overset{(e)}{\geq} & q(\boldsymbol{\lambda}^\ast)+ \frac{1}{2c} \Vert \boldsymbol{\lambda}(t+1) - \boldsymbol{\lambda}^\ast\Vert^2 - \frac{1}{2c} \Vert \boldsymbol{\lambda}(t) - \boldsymbol{\lambda}^\ast\Vert^2
\end{align*}
where (a) follows from Lemma \ref{lm:descent-lemma} and the fact that $\nabla_{\boldsymbol{\lambda}} q(\boldsymbol{\lambda}(t)) =  \mathbf{g}(\mathbf{x}(t))$; (b) follows from Fact \ref{fact:app_f3}; (c) follows from the identity $\mathbf{u}^T\mathbf{v} = \frac{1}{2} \Vert \mathbf{u}\Vert^2 + \frac{1}{2} \Vert \mathbf{v}\Vert^2 - \frac{1}{2} \Vert \mathbf{u} - \mathbf{v}\Vert^2, \forall \mathbf{u}, \mathbf{v}\in 
\mathbb{R}^m$; (d) follows from $c\leq\frac{1}{\gamma}$; and (e) follows from the concavity of $q(\cdot)$.

Rearranging terms yields the desired result.
\end{IEEEproof}

Fix $c\leq \frac{1}{\gamma}$ and $t>0$. By Fact \ref{fact:app_f4}, we have $q(\boldsymbol{\lambda}^\ast) - q(\boldsymbol{\lambda}(\tau+1)) \leq \frac{1}{2c} \Vert \boldsymbol{\lambda}(\tau) - \boldsymbol{\lambda}^\ast\Vert^2 - \frac{1}{2c} \Vert \boldsymbol{\lambda}(\tau+1) - \boldsymbol{\lambda}^\ast\Vert^2, \forall \tau\in\{0,1,\ldots, t-1\}$. Summing over $\tau$  and dividing by fact $t$ yields
\begin{align*}
\frac{1}{t}\sum_{\tau=0}^{t-1} \big[q(\boldsymbol{\lambda}^\ast) - q(\boldsymbol{\lambda}(\tau+1))\big] \leq & \frac{1}{2c t} [ \Vert \boldsymbol{\lambda}(0) - \boldsymbol{\lambda}^\ast\Vert^2 - \Vert \boldsymbol{\lambda}(t) - \boldsymbol{\lambda}^\ast\Vert^2] \\
\leq & \frac{1}{2c t} \Vert \boldsymbol{\lambda}(0) - \boldsymbol{\lambda}^\ast\Vert^2
\end{align*}

Note that $q(\boldsymbol{\lambda}^\ast) - q(\boldsymbol{\lambda}(\tau+1)), \forall \tau\in\{0,1,\ldots, t-1\}$ is a decreasing sequence by Fact \ref{fact:app_f2}. Thus, we have
\begin{align*}
q(\boldsymbol{\lambda}^\ast) - q(\boldsymbol{\lambda}(t) \leq \frac{1}{t}\sum_{\tau=0}^{t-1} \big[q(\boldsymbol{\lambda}^\ast) - q(\boldsymbol{\lambda}(\tau+1))\big] \leq  \frac{1}{2c t} \Vert \boldsymbol{\lambda}(0) - \boldsymbol{\lambda}^\ast\Vert^2.
\end{align*}

\section{Proof of Part (2) of Lemma \ref{lm:local-quadratic-dual-arrive-time}} \label{app:local-quadratic-dual-geometric-convergence}
This part is essentially a local version of Theorem 12 in \cite{Necoara15LinearConvergence}, which shows that the projected gradient method for set constrained smooth convex optimization converge geometrically if the objective function satisfies a quadratic growth condition.

In this appendix, we provide a simple proof that directly follows from Fact \ref{fact:app_f4} and Assumption \ref{as:dual-locally-quadratic}. By Fact \ref{fact:app_f4}, we have
\begin{align}
q(\boldsymbol{\lambda}^\ast) - q(\boldsymbol{\lambda}(t+1)) \leq \frac{1}{2c} \Vert \boldsymbol{\lambda}(t) - \boldsymbol{\lambda}^\ast\Vert^2 - \frac{1}{2c} \Vert \boldsymbol{\lambda}(t+1) - \boldsymbol{\lambda}^\ast\Vert^2, \forall t\geq 0. \label{eq:pf-local-quadratic-dual-arrive-time-eq1}
\end{align}
By part (1), we know $\Vert \boldsymbol{\lambda}(t) -\boldsymbol{\lambda}^\ast\Vert \leq D_q, \forall t\geq T_q$. By Assumption \ref{as:dual-locally-quadratic}, we have
\begin{align}
q(\boldsymbol{\lambda}^\ast) - q(\boldsymbol{\lambda}(t+1))  \geq L_q \Vert  \boldsymbol{\lambda}(t+1) - \boldsymbol{\lambda}^\ast\Vert^2, \forall t\geq T_q. \label{eq:pf-local-quadratic-dual-arrive-time-eq2}
\end{align}

Combining \eqref{eq:pf-local-quadratic-dual-arrive-time-eq1} and \eqref{eq:pf-local-quadratic-dual-arrive-time-eq2} yields
\begin{align*}
(L_q + \frac{1}{2c}) \Vert \boldsymbol{\lambda}(t+1) - \boldsymbol{\lambda}^\ast\Vert^2 \leq \frac{1}{2c} \Vert \boldsymbol{\lambda}(t) - \boldsymbol{\lambda}^\ast\Vert^2, \forall t\geq T_q.
\end{align*}
This can be written as 
\begin{align*}
 \Vert \boldsymbol{\lambda}(t+1) - \boldsymbol{\lambda}^\ast\Vert \leq \sqrt{\frac{1}{1+2cL_q}}  \Vert \boldsymbol{\lambda}(t) - \boldsymbol{\lambda}^\ast\Vert, \forall t\geq T_q.
\end{align*}
By induction, we have
\begin{align*}
\Vert \boldsymbol{\lambda}(t) - \boldsymbol{\lambda}^\ast\Vert \leq \big( \sqrt{\frac{1}{1+2cL_q }} \big)^{t-T_q}\Vert \boldsymbol{\lambda}(T_q) - \boldsymbol{\lambda}^\ast\Vert, \forall t\geq T_q.
\end{align*}

\section{Proof of Lemma \ref{lm:smooth-modulus-strong-concave-modulus}}\label{app:smooth-modulus-strong-concave-modulus}
Define $\tilde{h}(\mathbf{x}) = - h(\mathbf{x})$. Then $\tilde{h}$ is smooth with modulus $\gamma$ and strongly convex with modulus $L_c$ over the set $\mathcal{X}$. By definition of smooth functions,  $f$ must be differentiable over set $\mathcal{X}$.   By Lemma \ref{lm:strong-convex}, we have
\begin{align*}
\tilde{h} (\mathbf{y}) \geq \tilde{h} (\mathbf{x}) +  [\nabla \tilde{h}(\mathbf{x})]^T (\mathbf{y} - \mathbf{x}) + \frac{L_c}{2} \Vert \mathbf{y} - \mathbf{x}\Vert^2, \forall \mathbf{x}, \mathbf{y}\in \mathcal{X}
\end{align*}
By Lemma \ref{lm:descent-lemma}, 
\begin{align*}
\tilde{h} (\mathbf{y})  \leq  \tilde{h} (\mathbf{x}) + [\nabla \tilde{h}(\mathbf{x})]^T (\mathbf{y} - \mathbf{x}) + \frac{\gamma}{2} \Vert \mathbf{y} - \mathbf{x}\Vert^2, \forall \mathbf{x}, \mathbf{y}\in \mathcal{X}
\end{align*}
Since $\mathcal{X}$ is not a singleton, we can choose distinct $\mathbf{x}, \mathbf{y}\in \mathcal{X}$. Combining the above two inequalities yields $ L_c\leq \gamma$.

\section{Proof of Part (2) of Lemma \ref{lm:local-strongly-concave-arrive-time}} \label{app:local-strongly-concave-dual-geometric-convergence}
By the first part of this lemma, $\boldsymbol{\lambda}(t) \in \{ \boldsymbol{\lambda} \in\mathbb{R}^m_+:\Vert \boldsymbol{\lambda} - \boldsymbol{\lambda}^\ast \Vert \leq D_c\}, \forall t\geq T_{c}$. The remaining part of the proof is essentially a local version of the convergence time proof of the projected gradient method for set constrained smooth and strongly convex optimization \cite{book_ConvexOpt_Nesterov}.

Recall that $q(\boldsymbol{\lambda})$ is smooth with modulus $\gamma = \frac{\beta^2}{\alpha}$ by Lemma \ref{lm:dual-smooth}.
The next fact is an enhancement of Fact \ref{fact:app_f4} using the locally strong concavity of the dual function.

\begin{Fact}\label{fact:app_f5} If $c\leq \frac{1}{\gamma} = \frac{\alpha}{\beta^2}$, then $q(\boldsymbol{\lambda}^\ast) - q(\boldsymbol{\lambda}(t+1)) \leq (\frac{1}{2c} - \frac{L_c}{2}) \Vert \boldsymbol{\lambda}(t) - \boldsymbol{\lambda}^\ast\Vert^2 - \frac{1}{2c} \Vert \boldsymbol{\lambda}(t+1) - \boldsymbol{\lambda}^\ast\Vert^2, \forall t\geq T_c$.
\end{Fact}
\begin{IEEEproof} Fix $t\geq T_c$, 
\begin{align*}
&q(\boldsymbol{\lambda}(t+1)) \\
\overset{(a)}{\geq} &  q(\boldsymbol{\lambda}(t)) + [\mathbf{g}(\mathbf{x}(t))]^T [\boldsymbol{\lambda}(t+1) - \boldsymbol{\lambda}(t)] - \frac{\gamma}{2} \Vert \boldsymbol{\lambda}(t+1) - \boldsymbol{\lambda}(t)\Vert^2 \\
=& q(\boldsymbol{\lambda}(t)) + [\mathbf{g}(\mathbf{x}(t))]^T [\boldsymbol{\lambda}(t+1) - \boldsymbol{\lambda}^\ast + \boldsymbol{\lambda}^\ast - \boldsymbol{\lambda}(t)] - \frac{\gamma}{2} \Vert \boldsymbol{\lambda}(t+1) - \boldsymbol{\lambda}(t)\Vert^2 \\
\overset{(b)}{\geq} &q(\boldsymbol{\lambda}(t)) + [\mathbf{g}(\mathbf{x}(t))]^T [\boldsymbol{\lambda}^\ast - \boldsymbol{\lambda}(t)] - \frac{\gamma}{2} \Vert \boldsymbol{\lambda}(t+1) - \boldsymbol{\lambda}(t)\Vert^2 + \frac{1}{c}[\boldsymbol{\lambda}(t+1) - \boldsymbol{\lambda}(t)]^T [\boldsymbol{\lambda}(t+1) - \boldsymbol{\lambda}^\ast]\\
\overset{(c)}{\geq} &q(\boldsymbol{\lambda}(t)) + [\mathbf{g}(\mathbf{x}(t))]^T [\boldsymbol{\lambda}^\ast - \boldsymbol{\lambda}(t)] - \frac{\gamma}{2} \Vert \boldsymbol{\lambda}(t+1) - \boldsymbol{\lambda}(t)\Vert^2 + \frac{1}{2c} \Vert \boldsymbol{\lambda}(t+1) - \boldsymbol{\lambda}(t)\Vert^2 + \frac{1}{2c} \Vert \boldsymbol{\lambda}(t+1) - \boldsymbol{\lambda}^\ast\Vert^2 \\ &- \frac{1}{2c} \Vert \boldsymbol{\lambda}(t) - \boldsymbol{\lambda}^\ast\Vert^2\\
\overset{(d)}{\geq} & q(\boldsymbol{\lambda}(t)) + [\mathbf{g}(\mathbf{x}(t))]^T [\boldsymbol{\lambda}^\ast - \boldsymbol{\lambda}(t)] + \frac{1}{2c} \Vert \boldsymbol{\lambda}(t+1) - \boldsymbol{\lambda}^\ast\Vert^2 - \frac{1}{2c} \Vert \boldsymbol{\lambda}(t) - \boldsymbol{\lambda}^\ast\Vert^2\\
\overset{(e)}{\geq} & q(\boldsymbol{\lambda}^\ast)+ \frac{1}{2c} \Vert \boldsymbol{\lambda}(t+1) - \boldsymbol{\lambda}^\ast\Vert^2 +(\frac{L_c}{2}- \frac{1}{2c}) \Vert \boldsymbol{\lambda}(t) - \boldsymbol{\lambda}^\ast\Vert^2
\end{align*}
where (a) follows from Lemma \ref{lm:descent-lemma} and the fact that $\nabla_{\boldsymbol{\lambda}} q(\boldsymbol{\lambda}(t)) =  \mathbf{g}(\mathbf{x}(t))$; (b) follows from Fact \ref{fact:app_f3}; (c) follows from the identity $\mathbf{u}^T\mathbf{v} = \frac{1}{2} \Vert \mathbf{u}\Vert^2 + \frac{1}{2} \Vert \mathbf{v}\Vert^2 - \frac{1}{2} \Vert \mathbf{u} - \mathbf{v}\Vert^2, \forall \mathbf{u}, \mathbf{v}\in 
\mathbb{R}^m$; (d) follows from $c\leq\frac{1}{\gamma}$; and (e) follows from the fact that $q(\cdot)$ is strongly concave over the set $\{\boldsymbol{\lambda}\in \mathbb{R}^m_+ : \Vert \boldsymbol{\lambda} - \boldsymbol{\lambda}^\ast\Vert \leq D_c \}$ such that $q(\boldsymbol{\lambda}^\ast)  \leq  q(\boldsymbol{\lambda}(t)) + [\mathbf{g}(\mathbf{x}(t))]^T [\boldsymbol{\lambda}^\ast - \boldsymbol{\lambda}(t)] - \frac{L_c}{2} \Vert \boldsymbol{\lambda}^\ast - \boldsymbol{\lambda}(t)\Vert^2$ by Lemma \ref{lm:strong-convex}\footnote{Note that the dual function $q(\boldsymbol{\lambda})$ is differentiable and has gradient $\nabla_{\boldsymbol{\lambda}} q(\boldsymbol{\lambda}(t)) = \mathbf{g}(\mathbf{x}(t))$ by the strong convexity of $f(\mathbf{x})$ and Proposition B.25 in \cite{book_NonlinearProgramming_Bertsekas}. Applying Lemma \ref{lm:strong-convex} to $\tilde{q}(\boldsymbol{\lambda}) = -q(\boldsymbol{\lambda})$, which has gradient $\nabla_{\boldsymbol{\lambda}} \tilde{q}(\boldsymbol{\lambda}(t)) = -\mathbf{g}(\mathbf{x}(t))$ and is strongly convex over the set $\{\boldsymbol{\lambda}\in \mathbb{R}^m_+ :\Vert \boldsymbol{\lambda} - \boldsymbol{\lambda}^\ast\Vert \leq D_c \}$, yields $q(\boldsymbol{\lambda}^\ast) \leq  q(\boldsymbol{\lambda}(t)) +  [\mathbf{g}(\mathbf{x}(t))]^T [\boldsymbol{\lambda}^\ast - \boldsymbol{\lambda}(t)] - \frac{L_c}{2} \Vert \boldsymbol{\lambda}^\ast - \boldsymbol{\lambda}(t)\Vert^2$.}.

Rearranging terms yields the desired inequality.
\end{IEEEproof}

Note that $q(\boldsymbol{\lambda}^\ast) - q(\boldsymbol{\lambda}(t+1)) \geq 0, \forall t>0$.  Combining with Fact \ref{fact:app_f5} yields $(\frac{1}{2c} - \frac{L_c}{2}) \Vert \boldsymbol{\lambda}(t) - \boldsymbol{\lambda}^\ast\Vert^2 - \frac{1}{2c} \Vert \boldsymbol{\lambda}(t+1) - \boldsymbol{\lambda}^\ast\Vert^2 \geq 0, \forall t\geq T_c$. Recall that $c \leq \frac{1}{\gamma}$ implies that $c\leq \frac{1}{L_c}$ by Lemma \ref{lm:smooth-modulus-strong-concave-modulus}. Thus, we have
\begin{align*}
\Vert \boldsymbol{\lambda}(t+1) - \boldsymbol{\lambda}^\ast\Vert \leq \sqrt{1 - cL_c} \Vert \boldsymbol{\lambda}(t) - \boldsymbol{\lambda}^\ast\Vert, \forall t\geq T_c.   
\end{align*}
By induction, we have
\begin{align*}
\Vert \boldsymbol{\lambda}(t) - \boldsymbol{\lambda}^\ast\Vert \leq \big( \sqrt{1- cL_c} \big)^{t-T_c}\Vert \boldsymbol{\lambda}(T_c) - \boldsymbol{\lambda}^\ast\Vert, \forall t\geq T_c.
\end{align*}

\section{Proof of  Theorem \ref{thm:strongly-concave-dual-fucntion}} \label{app:strongly-concave-dual-fucntion}

\begin{Lem}\label{lm:sufficient-condition-locally-quadratic}
Let $q(\boldsymbol{\lambda}): \mathbb{R}^{m}_{+}\rightarrow \mathbb{R}$ be a concave function and $q(\boldsymbol{\lambda})$ be maximized at $\boldsymbol{\lambda} = \boldsymbol{\lambda}^{\ast} \geq \mathbf{0}$. Denote $\mathbf{d} = \nabla_{\boldsymbol{\lambda}} q(\boldsymbol{\lambda}^{\ast})$ and $\mathbf{U}\boldsymbol{\Sigma}\mathbf{U}^{T} = \nabla_{\boldsymbol{\lambda}}^{2} q(\boldsymbol{\lambda}^{\ast})$. Suppose that the following conditions are satisfied:
\begin{enumerate}
\item Suppose $\mathbf{d}\leq \mathbf{0}$ and $\lambda_{k}^{\ast} d_{k} = 0, \forall k\in \{1,\ldots,m\}$. Denote $\mathcal{K} = \{k\in\{1,\ldots,m\}: d_{k} = 0\}$ and $l = |\mathcal{K}|$.
\item Suppose $\boldsymbol{\Sigma}\prec 0$ and  $\text{rank}(\mathbf{U}^{\prime}) = l$ where $\mathbf{U}^{\prime}$ is an $l\times n$ submatrix of  $\mathbf{U}$ composed by rows with indices in $\mathcal{K}$.
\end{enumerate}
Then, there exists $D_q>0$ and $L_{q}>0$ such that $q(\boldsymbol{\lambda}) \leq q(\boldsymbol{\lambda}^{\ast}) - L_{q} \Vert \boldsymbol{\lambda} - \boldsymbol{\lambda}^{\ast}\Vert^{2}$ for any $\boldsymbol{\lambda}\in \mathbb{R}^{m}_{+}$ and $\Vert \boldsymbol{\lambda} - \boldsymbol{\lambda}^{\ast}\Vert \leq D_q$.
\end{Lem}
\begin{IEEEproof}
Without loss of generality, assume that $\mathcal{K} = \{1,\ldots, l\}$. Denote $\mathbf{U} = \left[\begin{array}{l} \mathbf{U}^{\prime} \\ \mathbf{U}^{\prime\prime}\end{array}\right]$ where $\mathbf{U}^{\prime\prime}$ is the $(m-l)\times n$ matrix composed by $(l+1)$-th to $m$-th rows of $\mathbf{U}$. Since $\mathbf{d}\leq \mathbf{0}$, we have $d_i <0$ for all $i\notin \mathcal{K}$ by definition of $\mathcal{K}$. Let $\delta = \min_{\{l+1\leq k \leq m\}} \{|d_{k}|\}$ such that $d_{k} \leq -\delta, \forall k\in\{l+1,\ldots,m\}$.   For each $\boldsymbol{\lambda}$, we define $\boldsymbol{\mu}$ via $\mu_{k} = \lambda_{k} - \lambda_{k}^{\ast}, \forall k\in\{1,\ldots,l\}, \mu_{k}=0, \forall k\in\{l+1,\ldots,m\}$ and $\boldsymbol{\nu}$ via $\nu_{k}=0, \forall k\in\{1,\ldots,l\}, \nu_{k} = \lambda_{k} - \lambda_{k}^{\ast}, \forall k\in\{l+1,\ldots,m\}$ such that $\boldsymbol{\lambda} - \boldsymbol{\lambda}^{\ast} = \boldsymbol{\mu} + \boldsymbol{\nu}$ and $\Vert  \boldsymbol{\lambda} - \boldsymbol{\lambda}^{\ast}\Vert^{2} = \Vert \boldsymbol{\mu}\Vert^{2} + 
\Vert \boldsymbol{\nu}\Vert^{2}$. Define $l$-dimension vector $\boldsymbol{\mu}^\prime = [\mu_1, \ldots, \mu_l]$. Note that $\Vert \boldsymbol{\mu}^\prime\Vert = \Vert \boldsymbol{\mu}\Vert$.  By the first condition, $d_{k}\neq 0, \forall k\in\{l+1,\ldots,m\}$ implies that $\lambda_{k}^{\ast}=0, \forall k\in\{l+1,\ldots,m\}$, which together with the fact that $\boldsymbol{\lambda}\geq \mathbf{0}$ implies that $\boldsymbol{\nu}\geq \mathbf{0}$. If $\Vert \boldsymbol{\lambda} - \boldsymbol{\lambda}^{\ast}\Vert$ is sufficiently small, we have
\begin{align*}
q(\boldsymbol{\lambda}) &\overset{(a)}{=}  q(\boldsymbol{\lambda}^\ast) +  (\boldsymbol{\lambda} - \boldsymbol{\lambda}^\ast)^T \nabla_{\boldsymbol{\lambda}} q(\boldsymbol{\lambda}^\ast) + (\boldsymbol{\lambda} - \boldsymbol{\lambda}^\ast)^T \nabla^2_{\boldsymbol{\lambda}} q(\boldsymbol{\lambda}^\ast) (\boldsymbol{\lambda} - \boldsymbol{\lambda}^\ast) +o (\Vert \boldsymbol{\lambda} - \boldsymbol{\lambda}^\ast \Vert^2) \\
&= q(\boldsymbol{\lambda}^\ast) +  \sum_{k=1}^{l} \mu_{k} d_{k} + \sum_{k=l+1}^{m} \nu_{k} d_{k} + \boldsymbol{\mu}^{T} \mathbf{U}\boldsymbol{\Sigma} \mathbf{U}^{T} \boldsymbol{\mu}^{T}  + \boldsymbol{\nu}^{T} \mathbf{U}\boldsymbol{\Sigma} \mathbf{U}^{T} \boldsymbol{\nu}^{T} +o (\Vert \boldsymbol{\lambda} - \boldsymbol{\lambda}^\ast \Vert^2)\\
&\overset{(b)}{\leq}  q(\boldsymbol{\lambda}^\ast) -  \sum_{k=l+1}^{m} \nu_{k} \delta + \boldsymbol{\mu}^{\prime,T} \mathbf{U}^{\prime}\boldsymbol{\Sigma} \mathbf{U}^{\prime,T} \boldsymbol{\mu}^{\prime, T} +o (\Vert \boldsymbol{\lambda} - \boldsymbol{\lambda}^\ast \Vert^2) \\
&\overset{(c)}{\leq}  q(\boldsymbol{\lambda}^\ast) -  \sum_{k=l+1}^{m} \nu_{k} \delta - \kappa \Vert \boldsymbol{\mu}^\prime\Vert^{2} + o (\Vert \boldsymbol{\lambda} - \boldsymbol{\lambda}^\ast \Vert^2)\\
&\overset{(d)}{<} q(\boldsymbol{\lambda}^\ast) -   \kappa \Vert \boldsymbol{\nu}\Vert^{2} - \kappa \Vert \boldsymbol{\mu}\Vert^{2} + o (\Vert \boldsymbol{\lambda} - \boldsymbol{\lambda}^\ast \Vert^2)\\
&= q(\boldsymbol{\lambda}^\ast) -   \kappa \Vert\boldsymbol{\lambda} - \boldsymbol{\lambda}^\ast\Vert^{2}  + o (\Vert \boldsymbol{\lambda} - \boldsymbol{\lambda}^\ast \Vert^2)
\end{align*} 
where (a) follows from the second-order Taylor's expansion; (b) follows from the facts that $d_{k}=0, \forall k\in\{1,\ldots,l\}$; $\boldsymbol{\nu}\geq 0$ and $d_{k} \leq -\delta$; the last $m-l$ elements of vector $\boldsymbol{\mu}$ are zeros; and $\boldsymbol{\Sigma}\prec 0$; (c) is true because $\kappa > 0$ exists when $\text{rank}(\mathbf{U}^{\prime}) =l$ and $\boldsymbol{\Sigma}\prec 0$; and  (d) follows from $-\delta \leq -\kappa \nu_{k}, \forall k\in\{l+1,\ldots,m\}$, which is true as long as $\Vert \boldsymbol{\nu}\Vert$ is sufficiently small; and $\Vert \boldsymbol{\mu}^\prime\Vert = \Vert \boldsymbol{\mu}\Vert$.

By the definition of $o (\Vert \boldsymbol{\lambda} - \boldsymbol{\lambda}^\ast \Vert^2)$, for any $\kappa>0$, we have $o (\Vert \boldsymbol{\lambda} - \boldsymbol{\lambda}^\ast \Vert^2) \leq \frac{\kappa}{2} \Vert \boldsymbol{\lambda} - \boldsymbol{\lambda}^\ast \Vert^2$ as long as $\Vert \boldsymbol{\lambda} - \boldsymbol{\lambda}^\ast \Vert$ is sufficiently small. Thus, there exists $D_q>0$ such that  
\begin{align*}
q(\boldsymbol{\lambda}) \leq q(\boldsymbol{\lambda}^\ast) - L_q\Vert \boldsymbol{\lambda} - \boldsymbol{\lambda}^\ast \Vert^2, \forall \boldsymbol{\lambda}\in \{\boldsymbol{\lambda}\in \mathbb{R}^{m}_{+}: \Vert \boldsymbol{\lambda} - \boldsymbol{\lambda}^{\ast}\Vert \leq D_q\}
\end{align*}
where $L_{s} = \kappa/2$.
\end{IEEEproof}

\begin{Lem}\label{lm:negative-definite-hessian-imply-strongly-concave}
Let $q(\boldsymbol{\lambda}): \mathbb{R}^{m}_{+}\rightarrow \mathbb{R}$ be a second-order continuously differentiable concave function and $q(\boldsymbol{\lambda})$ be maximized at $\boldsymbol{\lambda} = \boldsymbol{\lambda}^{\ast} \geq \mathbf{0}$.  If $\nabla_{\boldsymbol{\lambda}}^2 q(\boldsymbol{\lambda}^{\ast})\prec 0$, then there exist $D_c>0$ and $L_{c}>0$ such that $q(\cdot)$ is strongly concave on the set $\boldsymbol{\lambda}\in \{\boldsymbol{\lambda}\in \mathbb{R}^{m}_{+}: \Vert \boldsymbol{\lambda} - \boldsymbol{\lambda}^{\ast}\Vert \leq D_c\}$
\end{Lem}
\begin{IEEEproof}
This lemma trivially follows from the continuity of $\nabla_{\boldsymbol{\lambda}}^2 q(\boldsymbol{\lambda})$.
\end{IEEEproof}

{\it Proof of part (1) of Theorem \ref{thm:strongly-concave-dual-fucntion}}: 

Note that Assumption \ref{as:strongly-convex-problem} is trivially true.  Assumption \ref{as:strong-duality} follows from the assumption\footnote{The assumption that $\text{rank}(\nabla_{\mathbf{x}} \mathbf{g}_{\mathcal{K}}(\mathbf{x}^\ast)^T) = l$ is known as the non-degenerate constraint qualification or linear independence constraint qualification, which along with the famous Slater's constraint qualification, is one of various constraint qualifications implying the strong duality \cite{book_NonlinearProgrammingTA,book_MathematicsEconomists}.} that $\text{rank}(\nabla \mathbf{g}_{\mathcal{K}}(\mathbf{x}^\ast)^T) = l$.  To show that Assumption \ref{as:dual-locally-quadratic} holds, we need to apply Lemma \ref{lm:sufficient-condition-locally-quadratic}.

By the strong convexity of $f(\mathbf{x})$ and Proposition B.25 in \cite{book_NonlinearProgramming_Bertsekas}, the dual function $q(\boldsymbol{\lambda})$ is differentiable and has gradient $\nabla_{\boldsymbol{\lambda}} q(\boldsymbol{\lambda}^{\ast}) = \mathbf{g}(\mathbf{x}^{\ast})$. Thus, $\mathbf{d} = \nabla_{\boldsymbol{\lambda}} q(\boldsymbol{\lambda}^{\ast}) \leq 0$. By Assumption \ref{as:strong-duality}, i.e., the strong duality, we have $\lambda_{k}^{\ast} d_{k} = 0, \forall k\in\{1,\ldots,m\}$. Thus, the first condition in Lemma \ref{lm:sufficient-condition-locally-quadratic} is satisfied.

For $\boldsymbol{\lambda}\geq 0$, define $\mathbf{x}^\ast (\boldsymbol{\lambda}) = \argmin_{\mathbf{x}\in \mathbb{R}^n} \big[f(\mathbf{x}) + \boldsymbol{\lambda}^T \mathbf{g}(\mathbf{x})\big]$ and note $\mathbf{x}^\ast = \mathbf{x}^\ast(\boldsymbol{\lambda}^\ast) $. Note that $\mathbf{x}^\ast (\boldsymbol{\lambda})$ is a well-defined function because $f(\mathbf{x}) + \boldsymbol{\lambda}^T \mathbf{g}(\mathbf{x})$ is strongly convex and hence is minimized at a unique point.  By equation $(6.9)$, page 598, in \cite{book_NonlinearProgramming_Bertsekas}, we have
\begin{align}
\nabla^2_{\boldsymbol{\lambda}} q(\boldsymbol{\lambda}^\ast) = &- \Big[\nabla_{\mathbf{x}} \mathbf{g}(\mathbf{x}^\ast)\Big]^T \Big[\nabla_{\mathbf{x}}^2 f(\mathbf{x}^\ast) + \sum_{k=1}^m \lambda_k^\ast \nabla_{\mathbf{x}}^2 g_k(\mathbf{x}^\ast) \Big]^{-1} \Big[\nabla_{\mathbf{x}} \mathbf{g}(\mathbf{x}^\ast)\Big] \label{eq:hession-of-dual-function}
\end{align} 

Note that $\nabla_{\mathbf{x}}^2 f(\mathbf{x}^\ast) + \sum_{k=1}^m \lambda_k^\ast \nabla_{\mathbf{x}}^2 g_k(\mathbf{x}^\ast) \succ 0$ because $f$ is strongly convex and $g_k, k\in\{1,\ldots, m\}$ are convex. Thus, if $\text{rank}(\nabla_{\mathbf{x}} \mathbf{g}_{\mathcal{K}}(\mathbf{x}^\ast)^{T}) = |\mathcal{K}|$, then the second condition of Lemma \ref{lm:sufficient-condition-locally-quadratic} is satisfied.

{\it Proof of part (2) of Theorem \ref{thm:strongly-concave-dual-fucntion}}: 
Using the same argument, we can show that Assumptions \ref{as:strongly-convex-problem}-\ref{as:strong-duality} hold.  By equation \eqref{eq:hession-of-dual-function} and the assumption that $\text{rank}(\nabla_{\mathbf{x}} \mathbf{g}(\mathbf{x}^\ast)^{T})=m$, Assumption \ref{as:dual-locally-stronlgy-concave} follows from Lemma \ref{lm:negative-definite-hessian-imply-strongly-concave}.

\section{Proof of Theorem \ref{thm:NUM-locally-quadratic}} \label{app:NUM-locally-quadratic}
\begin{itemize}
\item Proof of Part (1):  Let $\mathbf{x}^{\ast}$ be the optimal solution to problem \eqref{eq:NUM-reduced-object}-\eqref{eq:NUM-reduced-rate-constraint}.  Since each column of $\mathbf{A}$ has at least one non-zero entry, we have $x_{i}^{\ast} \leq b^{\max} , \forall i\in\{1,2,\ldots, n\}$ with $b^{\max} = \max_{1\leq i\leq n} b_{i}$. Thus, problem  \eqref{eq:NUM-compact-object}-\eqref{eq:NUM-compact-rate-constraint} is equivalent to problem \eqref{eq:NUM-reduced-object}-\eqref{eq:NUM-reduced-rate-constraint} since only a redundant constraint $\mathbf{x}\leq \mathbf{x}^{\max}$ is introduced.

\item Proof of Part (2):  
\begin{itemize}
\item To show Assumption \ref{as:strongly-convex-problem} holds:  It follows from the strong convexity of $\sum_{i=1}^{n} - w_{i}\log(x_{i})$ over set $\mathcal{X}= \{\mathbf{0} \leq \mathbf{x} \leq \mathbf{x}^{\max}\}$.
\item To show Assumption \ref{as:strong-duality} holds:  It follows from the assumption that $\mathbf{A}\mathbf{x}^\ast \leq \mathbf{b}$, which is Slater's condition for convex programs only with linear inequality constraints. (Or alternatively, Assumption \ref{as:strong-duality} is also implied by the assumption that $\text{rank}(\mathbf{A}^\prime) = m^\prime$, which is the linear independence constraint qualification for convex programs \cite{book_NonlinearProgrammingTA,book_MathematicsEconomists}.) Note that  we can prove that Assumption \ref{as:strong-duality} also holds for problem \eqref{eq:NUM-reduced-object}-\eqref{eq:NUM-reduced-rate-constraint} with a similar argument. 
\item To show Assumption \ref{as:dual-locally-quadratic} holds:  Define the dual function of problem \eqref{eq:NUM-reduced-object}-\eqref{eq:NUM-reduced-rate-constraint} as
\begin{align*}
\tilde{q}(\boldsymbol{\lambda}) = \min_{\mathbf{x}\geq \mathbf{0}}\Big\{ \sum_{i=1}^{n} -w_{i}\log(x_{i}) + \boldsymbol{\lambda}^{T}(\mathbf{A}\mathbf{x} - \mathbf{b}) \Big\}.
\end{align*}
Since Assumption \ref{as:strong-duality} holds for problem \eqref{eq:NUM-reduced-object}-\eqref{eq:NUM-reduced-rate-constraint}, we assume $(\mathbf{x}^{\ast}, \boldsymbol{\lambda}^\ast)$ be a primal-dual pair that attains the strong duality of problem \eqref{eq:NUM-reduced-object}-\eqref{eq:NUM-reduced-rate-constraint}. By the strong duality, $\mathbf{x}^{\ast} = \argmin_{\mathbf{x}\geq \mathbf{0}}\big[ \sum_{i=1}^{n} -w_{i}\log(x_{i}) + (\boldsymbol{\lambda}^{\ast})^{T}(\mathbf{A}\mathbf{x} - \mathbf{b}) \big]$, i.e, $x_{i}^{\ast} = \argmin_{x_{i}\geq 0} \big[ - w_{i}\log(x_{i}) + (\boldsymbol{\lambda}^{\ast})^{T} \mathbf{a}_{i} x_{i}\big], \forall 1\leq i\leq n$.  Thus, we have $x_{i}^{\ast} = \big[ \frac{w_{i}}{ (\boldsymbol{\lambda}^{\ast})^{T} \mathbf{a}_{i}}\big]_{0}^{\infty}, \forall 1\leq i\leq n$. In the proof of part (1), we show that $0 < x_{i}^{\ast} \leq b^{\max}, \forall 1\leq i\leq n$. (Note that  $x_{i}^{\ast} >0, \forall 1\leq i\leq n$ because the domain of $\log$ function is $(0, +\infty)$.) Thus, $0 < \frac{w_{i}}{(\boldsymbol{\lambda}^{\ast})^{T}\mathbf{a}_{i}} \leq b^{\max} < x_{i}^{\max}, \forall 1\leq i\leq n$ and $x_{i}^{\ast} = \big[ \frac{w_{i}}{ (\boldsymbol{\lambda}^{\ast})^{T} \mathbf{a}_{i}}\big]_{0}^{\infty} =  \frac{w_{i}}{(\boldsymbol{\lambda}^{\ast})^{T}\mathbf{a}_{i}}, \forall 1\leq i\leq n$.

Now consider the equivalent problem \eqref{eq:NUM-compact-object}-\eqref{eq:NUM-compact-rate-constraint}, whose Lagrangian dual function is given by $$q(\boldsymbol{\lambda}) = \min_{ \mathbf{0}\leq \mathbf{x} \leq \mathbf{x}^{\max}}\big\{\sum_{i=1}^{n} -w_{i}\log(x_{i}) + \boldsymbol{\lambda}^{T}(\mathbf{A}\mathbf{x} - \mathbf{b})\big\}.$$ Note that  $\mathbf{x}^{\ast}$ is still the optimal solution to problem \eqref{eq:NUM-compact-object}-\eqref{eq:NUM-compact-rate-constraint}. Note that $\argmin_{ \mathbf{0}\leq \mathbf{x} \leq \mathbf{x}^{\max}} \big[ \sum_{i=1}^{n} -w_{i}\log(x_{i}) + \boldsymbol{\lambda}^{T}(\mathbf{A}\mathbf{x} - \mathbf{b})\big] = \big[ \big[\frac{w_{1}}{\boldsymbol{\lambda}^{T}\mathbf{a}_{1}}\big]_{0}^{x_{1}^{\max}}, \ldots, \big[\frac{w_{n}}{\boldsymbol{\lambda}^{T}\mathbf{a}_{n}}\big]_{0}^{x_{n}^{\max}}\big]^{T}$.  By the fact that $0< \frac{w_{i}}{(\boldsymbol{\lambda}^{\ast})^{T}\mathbf{a}_{i}} < x_{i}^{\max}, \forall i\in\{1,2,\ldots, n\}$, we know $ \big[ \big[\frac{w_{1}}{(\boldsymbol{\lambda}^\ast)^{T}\mathbf{a}_{1}}\big]_{0}^{x_{1}^{\max}}, \ldots, \big[\frac{w_{n}}{(\boldsymbol{\lambda}^\ast)^{T}\mathbf{a}_{n}}\big]_{0}^{x_{n}^{\max}}\big]^{T}  =  \big[ \big[\frac{w_{1}}{(\boldsymbol{\lambda}^\ast)^{T}\mathbf{a}_{1}}\big]_{0}^{\infty}, \ldots, \big[\frac{w_{n}}{(\boldsymbol{\lambda}^\ast)^{T}\mathbf{a}_{n}}\big]_{0}^{\infty}\big]^{T}$.  Thus, $q(\boldsymbol{\lambda}^\ast) = \tilde{q}(\boldsymbol{\lambda}^\ast)$ and $\boldsymbol{\lambda}^\ast$ also attains the strong duality for the equivalent problem \eqref{eq:NUM-compact-object}-\eqref{eq:NUM-compact-rate-constraint}.  By the continuity of functions $\frac{w_{i}}{\boldsymbol{\lambda}^{T} \mathbf{a}_{i}}$ and recall that $0 < \frac{w_{i}}{(\boldsymbol{\lambda}^{\ast})^{T}\mathbf{a}_{i}} < x_{i}^{\max}, \forall 1\leq i\leq n$ , we know $ \big[ \big[\frac{w_{1}}{\boldsymbol{\lambda}^{T}\mathbf{a}_{1}}\big]_{0}^{x_{1}^{\max}}, \ldots, \big[\frac{w_{n}}{\boldsymbol{\lambda}^{T}\mathbf{a}_{n}}\big]_{0}^{x_{n}^{\max}}\big]^{T}  =  \big[ \big[\frac{w_{1}}{\boldsymbol{\lambda}^{T}\mathbf{a}_{1}}\big]_{0}^{\infty}, \ldots, \big[\frac{w_{n}}{\boldsymbol{\lambda}^{T}\mathbf{a}_{n}}\big]_{0}^{\infty}\big]^{T} = \big[ \frac{w_{1}}{\boldsymbol{\lambda}^{T}\mathbf{a}_{1}}, \ldots, \frac{w_{n}}{\boldsymbol{\lambda}^{T}\mathbf{a}_{n}} \big]^{T}$ when $\boldsymbol{\lambda}$ is sufficiently close to $\boldsymbol{\lambda}^{\ast}$.

Next, we show that the dual function $q(\boldsymbol{\lambda})$ is locally quadratic in a neighborhood of $\boldsymbol{\lambda}^{\ast}$ by using Lemma \ref{lm:sufficient-condition-locally-quadratic}. Consider $\boldsymbol{\lambda} \in \mathbb{R}^{m}_{+}$ such that $\Vert \boldsymbol{\lambda} - \boldsymbol{\lambda}^\ast \Vert $ is sufficiently small, or equivalently,  $\boldsymbol{\lambda}$ is sufficiently close to $\boldsymbol{\lambda}^{\ast}$.  For such $\boldsymbol{\lambda}$, we have
\begin{align*}
\argmin_{\mathbf{0} \leq \mathbf{x}\leq \mathbf{x}^{\max}} \big[ \sum_{i=1}^{n} -w_{i}\log(x_{i}) + \boldsymbol{\lambda}^{T}(\mathbf{A}\mathbf{x} - \mathbf{b}) \big] = \big[ \big[\frac{w_{1}}{\boldsymbol{\lambda}^{T}\mathbf{a}_{1}}\big]_{0}^{x_{1}^{\max}}, \ldots, \big[\frac{w_{n}}{\boldsymbol{\lambda}^{T}\mathbf{a}_{n}}\big]_{0}^{x_{n}^{\max}}\big]^{T}=  \big[ \frac{w_{i}}{\boldsymbol{\lambda}^{T}\mathbf{a}_{1}}, \ldots, \frac{w_{n}}{\boldsymbol{\lambda}^{T}\mathbf{a}_{n}} \big]^{T}
\end{align*} 
Thus,
\begin{align*}
q(\boldsymbol{\lambda}) &= \sum_{i=1}^{n}  \Big[- w_{i}\log\big(\frac{1}{\boldsymbol{\lambda}^{T} \mathbf{a}_{i}}\big) + \frac{w_{i}\boldsymbol{\lambda}^{T} \mathbf{a}_{i}}{\boldsymbol{\lambda}^{T} \mathbf{a}_{i}}\Big] - \boldsymbol{\lambda}^{T} \mathbf{b} \\
 &= \sum_{i=1}^{n}  \Big[- w_{i}\log\big(\frac{1}{\boldsymbol{\lambda}^{T} \mathbf{a}_{i}}\big) \Big] + \sum_{i=1}^{n} w_{i} - \boldsymbol{\lambda}^{T} \mathbf{b}
\end{align*}
for $\boldsymbol{\lambda} \in \mathbb{R}^{m}_{+}$ such that $\Vert \boldsymbol{\lambda} - \boldsymbol{\lambda}^\ast \Vert $ is sufficiently small. Note that $q(\boldsymbol{\lambda})$ given above is infinitely differentiable. Taking the first-order and second-order derivatives at $\boldsymbol{\lambda} = \boldsymbol{\lambda}^{\ast}$ yields
\begin{align}
\nabla_{\boldsymbol{\lambda}} q(\boldsymbol{\lambda}^{\ast}) &= \sum_{i=1}^{n} \frac{w_{i}\mathbf{a}_{i}}{(\boldsymbol{\lambda}^{\ast})^{T}\mathbf{a}_{i}} - \mathbf{b}, \label{eq:NUM-dual-function-gradient}
\end{align}
\begin{align}
\nabla_{\boldsymbol{\lambda}}^{2} q(\boldsymbol{\lambda}^{\ast}) = &- \sum_{i=1}^{n} \frac{w_{i}\mathbf{a}_{i} \mathbf{a}_{i}^{T}}{((\boldsymbol{\lambda}^{\ast})^{T}\mathbf{a}_{i})^{2}}  \nonumber \\
= & \mathbf{A} \text{diag}\big([-\frac{w_{1}}{((\boldsymbol{\lambda}^{\ast})^{T}\mathbf{a}_{i})^{2}}, \ldots, -\frac{w_{n}}{((\boldsymbol{\lambda}^{\ast})^{T}\mathbf{a}_{n})^{2}}]\big)\mathbf{A}^{T},  \label{eq:NUM-dual-function-hessian}
\end{align}
where $ \text{diag}\big([-\frac{w_{1}}{((\boldsymbol{\lambda}^{\ast})^{T}\mathbf{a}_{i})^{2}}, \ldots, -\frac{w_{n}}{((\boldsymbol{\lambda}^{\ast})^{T}\mathbf{a}_{n})^{2}}]\big)$ denotes the diagonal matrix with diagonal entries given by 
\begin{align*}
-\frac{w_{1}}{((\boldsymbol{\lambda}^{\ast})^{T}\mathbf{a}_{i})^{2}}, \ldots, -\frac{w_{n}}{((\boldsymbol{\lambda}^{\ast})^{T}\mathbf{a}_{n})^{2}}.
\end{align*}
Note that $\frac{w_{i}}{(\boldsymbol{\lambda}^{\ast})^{T}\mathbf{a}_{i}} > 0, \forall 1\leq i\leq n$. Thus, if $\text{rank}(\mathbf{A}^{\prime}) = m^\prime$, then Assumption \ref{as:dual-locally-quadratic} holds by Lemma \ref{lm:sufficient-condition-locally-quadratic}.
\end{itemize}

\item Proof of Part (3): Using the same arguments in the proof of part (2), we can show that Assumptions \ref{as:strongly-convex-problem}-\ref{as:strong-duality} hold.  By equation \eqref{eq:NUM-dual-function-hessian} and the fact that $\text{rank}(\mathbf{A}) = m$, Assumption \ref{as:dual-locally-stronlgy-concave} follows from Lemma \ref{lm:negative-definite-hessian-imply-strongly-concave}.
\end{itemize}

\bibliographystyle{IEEEtran}
\bibliography{IEEEfull,mybibfile}

\end{document}